\documentclass[11pt]{amsart}
\usepackage[usenames,dvipsnames]{color}
\usepackage{amsmath,amssymb,amsthm,graphicx,mathrsfs,url,latexsym,enumerate}
\usepackage{color}
\usepackage{a4wide}
\usepackage{ifthen}
\usepackage{verbatim}
\usepackage{fancyhdr}
\usepackage{url}
\usepackage{bbm}
\usepackage[sans]{dsfont}
\usepackage{MnSymbol}

\usepackage{amsbsy}
\usepackage{wrapfig}
\usepackage{tikz}

\definecolor{bleu}{RGB}{27,88,145}
\definecolor{mauve}{RGB}{138,20,79}

\renewcommand{\Im}{\operatorname{Im}}
\renewcommand{\Re}{\operatorname{Re}}
\newcommand{\dist}{\operatorname{dist}}
\newcommand{\Id}{\operatorname{Id}}

\newcommand{\rk}{\operatorname{rk}}

\newcommand{\supp}{\operatorname{supp}}

\newcommand{\diag}{\operatorname{diag}}
\newcommand{\C}{\mathbb C}

\newcommand{\N}{\mathbb N}
\newcommand{\R}{\mathbb R}

\newcommand{\Hess}{\operatorname{Hess}}
\newcommand{\Cl}{\operatorname{Cl}}
\newcommand{\Ran}{\operatorname{Ran}}

\newcommand{\argmin}{\operatorname{argmin}}
\newcommand{\bsigma}{\boldsymbol{\sigma}}

\newcommand{\tr}{\operatorname{tr}}
\renewcommand{\div}{\operatorname{div}}
\renewcommand{\u}{{\bf u}}

\def\<{\langle}
\def\>{\rangle}

\def\m{\mathbf{m}}
\def\n{\mathbf{n}}
\def\s{\mathbf{s}}
\def\v{\mathbf{v}}
\def\x{\mathbf{x}}

\numberwithin{equation}{section}
\numberwithin{figure}{section}

\newtheorem{theorem}{Theorem}
\newtheorem{defin}{Definition}[section]
\newtheorem{corollary}[defin]{Corollary}
\newtheorem{lemma}[defin]{Lemma}

\newtheorem{proposition}[defin]{Proposition}
\newtheorem{remark}[defin]{Remark}

\def\aaa{{\mathcal A}}\def\bbb{{\mathcal B}}\def\ccc{{\mathcal C}}\def\ddd{{\mathcal D}}
\def\eee{{\mathcal E}}\def\fff{{\mathcal F}} 
 \def\jjj{{\mathcal J}}\def\lll{{\mathcal L}}
\def\mmm{{\mathcal M}} \def\ooo{{\mathcal O}}\def\ppp{{\mathcal P}}
\def\rrr{{\mathcal R}}\def\sss{{\mathcal S}}
\def\uuu{{\mathcal U}}\def\vvv{{\mathcal V}}

\def\Cr{{\mathscr C}}
\def\Er{{\mathscr E}}

\def\Mr{{\mathscr M}}
\def\Sr{{\mathscr S}}

\begin{document}
\title{Eyring--Kramers law for  Fokker--Planck type differential operators }


\author[J.-F. Bony]{Jean-Fran\c{c}ois Bony}
\address{J.-F. Bony, Institut Math\'ematique de Bordeaux, Universit\'e de Bordeaux}
\email{jean-francois.bony@math.u-bordeaux.fr}

\author[D. Le Peutrec]{Dorian Le Peutrec}
\address{D. Le Peutrec, Institut Denis Poisson, Universit\'e d'Orl\'eans}
\email{dorian.le-peutrec@univ-orleans.fr}

\author[L. Michel]{Laurent Michel}
\address{L. Michel, Institut Math\'ematique de Bordeaux, Universit\'e de Bordeaux}
\email{laurent.michel@math.u-bordeaux.fr}

\begin{abstract}
We consider Fokker--Planck type differential operators associated with general Langevin processes admitting a Gibbs stationary distribution. Under assumptions insuring suitable resolvent estimates, we prove Eyring--Kramers formulas for the bottom of the spectrum of these operators in the low temperature regime. Our approach is based on the construction of sharp Gaussian quasimodes which avoids supersymmetry or PT-symmetry assumptions.
\end{abstract}

\maketitle

\setcounter{tocdepth}{1}
\tableofcontents

\section{Introduction} \label{b29}

\subsection{Motivations}

Let $P$ be the real semiclassical second order differential operator
\begin{equation} \label{g22}
P = - h \div\circ A \circ h \nabla + \frac{1}{2} \big( b \cdot h \nabla + h \div \circ b \big) + c ,
\end{equation}
where the matrix field $A$, the vector field $b$ and the function $c$ depend smoothly on $x \in \R^{d}$, and where $h > 0$ is a small parameter. We assume that the matrix field $A$ is pointwise symmetric and positive semidefinite and that the function $c$ is nonnegative. In stochastic analysis, such operators arise naturally as the generators of time homogeneous Langevin processes
\begin{equation} \label{g50}
d X_{t} = \xi ( X_{t} ) + \sqrt{2 h} \sigma ( X_{t} ) \, d B_{t} ,
\end{equation}
where $( B_{t} )$ denotes the Brownian motion on $\R^{d}$, the vector field $\xi$ is the drift coefficient, the matrix field $\sigma$ is the diffusion coefficient and the parameter $h$ is proportional to the temperature of the system. Given any test function $\varphi$, the expectation  $u ( t , x ) = {\mathbb E} ( \varphi ( X_{t} ) \vert \ X_{0} = x )$ is solution of the Fokker--Planck equation
\begin{equation} \label{g23}
\left\{ \begin{aligned}
&\partial_{t} u + \lll u = 0 ,  \\
&u_{\vert t = 0} = \varphi ,
\end{aligned} 
\right.
\end{equation}
where
\begin{equation*} \label{defL}
\lll = - h \sum_{i , j=1}^{d} a_{i , j} \partial_{x_{i}} \partial_{x_{j}} - \sum_{k=1}^{d} \xi_{k} \partial_{x_{k}} , 
\end{equation*}
with $A = ( a_{i , j} ) = \sigma \sigma^{t}$. 
Up to the multiplicative factor $h$, this operator has the form \eqref{g22} for some suitable $b$ and $c$. Denoting $\lll^{\dagger}$ the formal $L^{2} ( d x )$ adjoint of $\lll$, \eqref{g23} is equivalent to say that the probability density $\varrho(t,\cdot)$ of the process $(X_{t})$ is solution of the adjoint equation
\begin{equation*}
\partial_{t} \varrho = \lll^{\dagger} \varrho .
\end{equation*}
Among many examples of such operators, let us mention two cases of particular interest.

Taking $\xi = - \nabla f$ for some smooth function $f$ on $\R^{d}$ and $\sigma = \Id_{\R^{d}}$, the generator $\lll$ of the overdamped Langevin process 
\begin{equation} \label{g24}
d X_{t} = - \nabla f ( X_{t} ) + \sqrt{2 h} \, d B_{t} ,
\end{equation}
writes
\begin{equation} \label{g25}
\lll = \lll_{K S} = - h \Delta + \nabla f \cdot \nabla ,
\end{equation}
which is sometimes called the Kramers--Smoluchowski operator.
Depending on the field of research, this operator is also known as the weighted Laplacian or Bakry--\'Emery Laplacian
and is unitarily equivalent to the Witten Laplacian.

Another famous example is given by the case where $\sigma : \R^{2 d} \rightarrow \R^{2 d}$ is the projection onto the subspace $0 \oplus \R^{d}$, $\sigma ( x , v ) = ( 0 , v )$ and $\xi : \R^{2 d} \rightarrow \R^{2 d}$, defined by $\xi ( x , v ) = ( v , - \nabla_{x} V - v )$, is related to the energy function
$f( x , v )=\frac 12 \vert v \vert^2 + V(x)$ depending on
a smooth potential $V$ on $\R^{d}$.
The associated Langevin equation writes 
\begin{equation}
\label{eq.lang}
\left\{ \begin{aligned}
d x_{t} &= v_{t} d t ,   \\
d v_{t} &= - \nabla_{x} V ( x_{t} ) d t - v_{t} d t + \sqrt{2 h} \, d B_{t} ,
\end{aligned} \right.
\end{equation}
where $( B_{t} )$ is the Brownian motion on $\R^d$. The associated generator is the Kramers--Fokker--Planch
operator
\begin{equation} \label{g2}
\lll = \lll_{K F P} = -v \cdot\nabla_{x} + \nabla_{x} V \cdot \nabla_{v} + v \cdot \nabla_{v} - h \Delta_{v} ,
\end{equation}
where $\Delta_{v}$ is the Laplace operator in the $v$ variable only.

The study of the operators $\lll_{K S}$ and $\lll_{K F P}$ has been the subject of many works in the last decades. It is particularly motivated by its applications to computational physics. The above processes are indeed ergodic with respect to their Gibbs measure and can thus be used to sample this distribution. We refer to \cite{LeRoSt10_01} for details on these topics.

From a theoretical point of view, the study of the qualitative  properties (well-posedness, asymptotic behavior) as well as of the quantitative properties (precise spectral asymptotics) of 
 the Fokker-Planck equation \eqref{g23} has recently  known some major progresses on the impulse of microlocal techniques. 
 When the matrix field $A$ is positive definite, the operator $P$ is elliptic and standard tools apply to prove general properties on the operator $P$ 
 (maximal accretivity, compactness of the resolvent). When $A$ is not invertible, the operator $P$ is not elliptic anymore (it is sometimes called a degenerate diffusion) and major progresses were recently made using hypoelliptic methods in the spirit of H\"ormander. For the Kramers--Fokker--Planck operator $\lll_{KFP}$, 
 exponential convergence to equilibrium was proved in \cite{Ta99} and explicit rate of decay in the non-semiclassical setting was given in \cite{HeNi04}.
More generally, hypocoercive methods  developed for various kinetic models provide now robust tools to prove
return to equilibrium and
 spectral gap estimates  (see \cite{Vi09} for an overview).

In the semiclassical setting $h\to0$, computing sharp spectral asymptotics for the low spectrum of~$P$ is a classical problem having a long history.
In the elliptic  self-adjoint case, i.e. when~$A$ is uniformly positive definite and  $b=0$, the low-lying eigenvalues of $P$ are localized near the absolute minimum value of the zeroth order part of $P$, that is the minimum value of the function $c$. Moreover, the harmonic and WKB approximations of $P$ near the absolute minima of $c$  yield spectral expansions in powers of $h$
of the  low-lying eigenvalues 
of $P$ (see \cite[Chapters~3~and~4]{DiSj99_01} 
for a detailed study in the case of Schr\"odinger operators). 

However, in certain situations,
these expansions are identical and, thus, do not permit to discriminate between these low-lying eigenvalues. This is for instance the case for Witten Laplacians associated with a confining Morse function $f$ (in this
case, the corresponding function $c$ also depends on $h$), for which we know from the early works of Witten \cite{Wi82}
and Helffer--Sj\"ostrand \cite{HeSj85_01} that
 $P$ admits exponentially small eigenvalues (that is of order $\mathcal O(e^{- C/h})$ for some $C>0$,
 and hence indistinguishable on the basis of their expansions in powers of $h$),
in one-to-one correspondence with the minima of $f$, and that the rest of its spectrum is above $\varepsilon_{*} h$
for some $\varepsilon_{*}>0$. 

Up to a unitary conjugation, the Witten Laplacian is nothing else but the Kramers--Smoluchowski operator \eqref{g25} and
its small eigenvalues govern the time of return to equilibrium for the
overdamped Langevin process~\eqref{g24}. 
The computation of these eigenvalues is a historical problem which at least goes back to Kramers \cite{Kr40}.
In the early '00s, sharp asymptotics of these small eigenvalues were obtained in \cite{BoGaKl05_01} and 
\cite{HeKlNi04_01}. Known as  Eyring--Kramers formulas, such spectral asymptotics were also obtained recently in \cite{BoRe16_01,LePMi20} in elliptic non-self-adjoint settings, associated with non-reversible overdamped Langevin processes generalizing \eqref{g24}. 
Concerning  the transition times of these processes,
 Eyring--Kramers formulas have been respectively established in both reversible and non-reversible settings in \cite{BoEcGaKl04} and \cite{LaMaSe19,LeSe20} under similar assumptions. We also refer to \cite{Be13} for a nice introduction to these topics.

In the non-elliptic case, major progresses in the analysis of the operator $P$ were made by H\'erau, Hitrik and Sj\"ostrand in a series of works. In  
\cite{HeHiSj08-2}, the authors proved resolvent estimates and established harmonic approximation of the spectrum under dynamical assumptions
on the symbol of the operator $P$.
In  \cite{HeHiSj11_01}, they applied these results to operators satisfying additional symmetries (supersymmetry and PT-symmetry). 
Under these assumptions, the operator $P$ admits a natural Gibbs stationary distribution $e^{-f/h}$ and the authors   proved 
spectral Eyring--Kramers formulas in this setting, where the 
small eigenvalues govern the time of return to equilibrium for the
Langevin process~\eqref{eq.lang}. 

Though it is satisfied by many interesting examples (as Kramers--Fokker--Planck operators), the 
assumptions of supersymmetry and PT-symmetry
 do not look necessary to prove sharp spectral asymptotics, as shown by the paper \cite{LePMi20}.
The aim of the present  paper is to prove spectral Eyring--Kramers formulas for general operators $P$ satisfying the assumptions of 
\cite{HeHiSj08-2} and admitting an
 explicit Gibbs stationary distribution, but without the additional symmetry assumptions of  \cite{HeHiSj11_01}.
 
\subsection{Framework and results} \label{s1}

Let $P = P ( x , h \partial_{x} , h )$ be a semiclassical second order differential operator on $\R^d$, $d \geq 1$, with smooth real coefficients,
\begin{equation} \label{a3}
P = - \sum_{i , j = 1}^{d} h \partial_{x_{i}} \circ a_{i , j} ( x , h ) \circ h \partial_{x_{j}} +\frac{1}{2} \sum_{j = 1}^{d} \big( b_{j} ( x , h ) \circ h \partial_{x_{j}} + h \partial_{x_{j}} \circ b_{j} ( x , h ) \big) + c ( x , h ) ,
\end{equation}
where the real functions $a_{i , j}$, $b_{j}$, and $c$ depend smoothly on $x\in\R^d $, $a_{i , j} = a_{j , i}$ for all $i , j=1,\dots,d$, and where $h \in ] 0 , 1 ]$ denotes the semiclassical parameter. We suppose that the coefficients of $P$ satisfy the following growth condition at infinity
\begin{equation}\label{a7}
\begin{aligned}
&\forall \vert \alpha \vert \geq 0 , \qquad &&\partial_{x}^{\alpha} a_{i , j} ( x , h ) = \ooo ( 1 ) ,  \\
&\forall \vert \alpha \vert \geq 1 ,  &&\partial_{x}^{\alpha} b_{j} ( x , h ) = \ooo ( 1 ) ,  \\
&\forall \vert \alpha \vert \geq 2 ,  &&\partial_{x}^{\alpha} c ( x , h ) = \ooo ( 1 ) ,
\end{aligned}
\end{equation}
uniformly with respect to $h$. We assume moreover that the above coefficients admit classical expansions: $a_{i , j} ( x , h ) \sim \sum_{k\in\N} h^{k} a_{i , j}^{k} ( x )$, $b_{j} ( x , h ) \sim \sum_{k\in\N} h^{k} b_{j}^{k} ( x )$ and $c ( x , h ) \sim \sum_{k\in\N} h^{k} c^{k} ( x )$ in the sense
\begin{equation}
\partial_{x}^{\alpha} \Big( e ( x , h ) - \sum_{0 \leq k \leq K} e^{k} ( x ) h^{k} \Big) = \ooo ( h^{K + 1} )
\end{equation}
for all $\alpha\in\N^{d}$, $K \in\N$ and $e \in \{ a_{i , j} , b_{j} , c \}$. This yields classical expansions for the matrix
field $A ( x , h ) = ( a_{i , j} ( x , h ) )_{i , j} \sim \sum_{k} h^{k} A^{k} ( x )$ and the vector field $b ( x , h ) \sim\sum_{k} h^{k} b^{k} ( x )$. Considering symbols which have a classical expansion allows  to deal with, e.g., Witten Laplacians and Kramers--Fokker--Planck operators, which naturally have subprincipal terms. Eventually, we also assume a partial positivity of the symbols of the operator $P$: for all $x \in \R^{d}$ and $h \in ] 0 , 1 ]$,
\begin{equation} \label{a8}
c^{0} ( x ) \geq 0 \qquad \text{and} \qquad A ( x,h ) = ( a_{i , j} ( x,h ) )_{i , j} \text{ is positive semidefinite.}
\end{equation}
Such operators were studied in details in \cite{HeHiSj08-2}, where the authors establish resolvent estimates together with spectral asymptotics. In particular, the graph closure of the operator $P$ initially defined on the Schwartz space $\sss ( \R^{d} )$,
still denoted by $P$, is maximal accretive and has domain
\begin{equation*}
\ddd ( P ) = \{ u \in L^{2} ( \R^{d} ) ; \ P u \in L^{2} ( \R^{d} ) \} ,
\end{equation*}
from \cite[Proposition~3.1 and Corollary~3.2]{HeHiSj08-2}. The same properties also hold for $P^{*}$ mutatis mutandis. Throughout the paper, we assume \eqref{a3}--\eqref{a8} without reminder.

Let us introduce the symbol $p ( x , \xi , h )$ of $P$ in the semiclassical Weyl quantization. It satisfies
\begin{equation*} \label{a4}
p ( x , \xi , h ) = \xi \cdot A ( x , h ) \xi + i b ( x , h ) \cdot\xi + c ( x , h ) +\frac{h^{2}}{4} \sum_{i , j = 1}^{d} \partial_{x_{i}} \partial_{x_{j}} a_{i , j} ( x , h ) ,
\end{equation*}
 where, throughout the paper, $x \cdot y$ denotes the usual scalar product of $x$ and $y$ in $\R^{d}$ (in order to facilitate the reading, we will also sometimes use the notation $\< x , y \> = x \cdot y$). It admits a classical expansion $p ( x , \xi , h ) \sim \sum_{k} h^{k} p^{k} ( x , \xi )$ and the principal symbol $p^{0}$ is given by $p^{0} = p^{0}_{2} + i p^{0}_{1} + p^{0}_{0}$ with $p^{0}_{2} ( x , \xi ) = \xi \cdot A^{0} ( x ) \xi$, $p^{0}_{1} ( x , \xi ) = b^{0} ( x ) \cdot \xi$ and $p^{0}_{0} ( x ) = c^{0} ( x )$. In order to lighten the notation, we will drop the superscript $0$ when it is unambiguous. Consider the symbol
\begin{equation} \label{a9}
\widetilde{p} ( x , \xi ) = p^{0}_{0} ( x ) +\frac{p^{0}_{2} ( x , \xi )}{\< \xi \>^{2}}.
\end{equation}
Thanks to \eqref{a8}, one has $p^{0}_{0} , p^{0}_{2} \geq 0$ and hence $\widetilde{p} \geq 0$. Given $T > 0$, let us define the symbol $\< \widetilde{p} \>_{T}$ by 
\begin{equation} \label{a10}
\< \widetilde{p} \>_{T} = \frac{1}{2T} \int_{- T}^{T} \widetilde{p} \circ e^{t H_{p_{1}^{0}}} d t ,
\end{equation}
where $H_{p_{1}^{0}} = \partial_{\xi} p^{0}_{1} \cdot \partial_{x} - \partial_{x} p^{0}_{1} \cdot \partial_{\xi}$ denotes the Hamilton vector field associated with the symbol~$p^{0}_{1}$. The critical set associated with $p$ is defined by 
\begin{equation} \label{a11}
\ccc = \big\{ ( x , 0 ) \in T^{*} \R^{d} ; \ b^{0} ( x ) = 0 \text{ and } c^{0} ( x ) = 0 \big\} .
\end{equation}
As in \cite{HeHiSj08-2}, we suppose that the set $\ccc$ is finite, $\ccc = \{ \rho_{1} , \ldots , \rho_{N} \}$, and that for some fixed $T > 0$ (see (4.21) and (4.23) in \cite{HeHiSj08-2}): there exists some constant $C > 0$ such that 
\begin{equation} \label{a12} \tag{Harmo}
\text{for all } \rho \text{ near any } \rho_{j} , \text{ we have } \< \widetilde{p} \>_{T} \geq \frac{1}{C} \vert \rho - \rho_{j} \vert^{2} ,
\end{equation}
and, for any neighborhood $U$ of $\pi_{x} \ccc$, one has
\begin{equation} \label{a14} \tag{Hypo}
\exists C > 0 , \ \forall x \in \R^{d} \setminus U , \qquad \operatorname{meas} \Big\{ t \in [- T , T ] ; \ c^{0} \big( \exp ( t b^{0} \cdot \nabla ) ( x ) \big) \geq \frac{1}{C} \Big\} \geq \frac{1}{C} .
\end{equation}
Assumption \eqref{a12} may look difficult to check in the applications, but Corollary \ref{g10} and Remark \ref{g11} provide a concrete criterion to verify it. Observe also that \eqref{a14} holds true for instance when $c^{0}$ is uniformly positive outside each neighborhood of $\pi_{x} \ccc$. Note that it is not necessary to assume (4.22) of \cite{HeHiSj08-2} here since this is a consequence of \eqref{a12} and \eqref{a14}, as explained in \cite[page 223]{HeHiSj08-2}. 

Under these assumptions, H\'erau, Hitrik and Sj\"ostrand obtained spectral informations that we sum up below. For this purpose, we introduce the fundamental matrix $F_{p^{0}}$ of the symbol~$p^{0}$ at a critical point $\rho \in \ccc$ (see \eqref{a11}) as the linearization of the Hamilton field $H_{p^{0}}$ at $\rho$. Its eigenvalues are of the form $\pm \lambda_{\rho , \ell}$, $1 \leq \ell \leq d$, with $\Im \lambda_{\rho , \ell} > 0$. We use the notation 
\begin{equation*}
\widetilde{\tr} ( p , \rho ) = - i \sum_{\ell=1}^{d} \lambda_{\rho , \ell} + 2 c^{1} ( \pi_{x} \rho ) .
\end{equation*}
Combining Proposition 7.2, Theorem 8.3 and Theorem 8.4 of \cite{HeHiSj08-2}, we get

\begin{theorem}[H\'erau--Hitrik--Sj\"ostrand]\sl \label{a15}
Assume that \eqref{a12} and \eqref{a14} hold true. For any $B>0$, there exists $C > 0$ such that
for $h$ small enough,
 the operator $P$ has no spectrum in
\begin{equation*}
\{ z \in \C ; \ \Re z < B h \text{ and } \vert \Im z \vert > C h \} .
\end{equation*}
Moreover, for any $ B >0$, there exists $\alpha>0$ such that, for $h$ small enough, the spectrum of $P$ in $D ( 0 , B h )$ is discrete and made of  
eigenvalues (with multiplicity) of the form
\begin{equation*}
\mu_{\rho , k} ( h ) = h ( \mu^{0}_{\rho , k}  + \ooo ( h^{\alpha} ) ) ,
\end{equation*}
where the $(\mu^{0}_{\rho , k})_{\rho\in\mathcal C, k\in\N}$
are all 
the possible numbers of the form
\begin{equation*}
\mu^{0}_{\rho , k} = \frac{1}{i} \sum_{\ell = 1}^{d} \nu_{\rho , k , \ell} \lambda_{\rho , \ell} + \frac{1}{2} \widetilde{\tr} ( p , \rho ) \quad \text{with} \quad \nu_{\rho , k , \ell} \in \N .
\end{equation*}
Finally, for every $B > 0$, there exists $C > 0$ such that 
\begin{equation} \label{a18}
\Vert ( P - z )^{- 1} \Vert \leq \frac{C}{h} 
\end{equation}
for $h$ small enough and $z \in \C$ such that $\Re z < B h$ and $\dist ( z , \sigma ( P ) ) \geq h / B$.
\end{theorem}

In addition, they showed that the remainder terms $\ooo ( h^{\alpha} )$   have a classical expansion in fractional powers of $h$. It is assumed in \cite{HeHiSj08-2} that the coefficients $a_{i,j} , b_{j} , c$ of the operator $P$ (see \eqref{a3}) do not depend on $h$, but a direct adaption to our setting gives Theorem \ref{a15}. It turns out that in many situations, some leading coefficients $\mu_{\rho , k}^{0}$ vanish and one aims at having a more accurate description of the spectrum. This is the case for instance for Witten Laplacians or Kramers--Fokker--Planck operators, which both admit an invariant distribution.
In the present paper, we consider the situation where the operator $P$ satisfies the assumptions of Theorem~\ref{a15} and there exists a smooth function $f : \R^{d} \rightarrow \R$ such that
\begin{equation} \label{h1} \tag{Confin}
e^{- f / h} \in L^{2} ( \R^{d} ) , \qquad \lim_{\vert x \vert \to + \infty} f ( x ) = + \infty ,
\end{equation}
with
\begin{equation} \label{a17} \tag{Gibbs}
 P ( e^{- f / h} ) = 0 \qquad \text{and} \qquad P^{\dagger} ( e^{- f / h} ) = 0 ,
\end{equation}
where $P^{\dagger}$ denotes the formal adjoint of $P$. In particular, $e^{- f / h} \in \ddd ( P ) \cap \ddd ( P^{*} )$. Moreover, we will assume that 
\begin{equation}\label{a24} \tag{Morse}
f \text{ is a Morse function with a finite number of critical points.}
\end{equation}
In the sequel, we denote by $\uuu$ the set of critical points of the Morse function $f$ and by $\uuu^{( j )}$ the set of its critical points of index $j = 0 , \ldots , d$ (that is the set of its critical points $\u$ such that $\Hess f(\u)$ has
signature $(d-j,j)$). 
Moreover, we denote by $n_{0} := \sharp \uuu^{( 0 )}$ the number of local minima of $f$,
by $H ( x ) := \Hess f ( x )$ the Hessian matrix of $f$ at $x \in \R^{d}$,
and we call saddle points the elements of $\uuu^{( 1 )}$.

For $j\in\{0,1,2\}$, let us denote by $P_{j}$ the $j$th order part of the operator $P=P_{2} + P_{1} + P_{0}$
with 
$P_{2} = - h \div \circ A \circ h \nabla$ formally self-adjoint,
$P_{1} = \frac{1}{2} (  b \cdot h \nabla + h \div \circ b )$ formally anti-adjoint,
and $P_{0} = c $ formally self-adjoint.  It then follows from \eqref{a17} that $P_{1} ( e^{- f / h} ) = 0$ and $( P_{2} + P_{0} ) ( e^{- f / h} ) = 0$. Using the classical expansions of the coefficients and looking at the terms of order $0$ in the expansion, we obtain  the following eikonal equations: for all $x \in \R^{d}$,
\begin{equation} \label{a19}
\< A^{0} ( x ) \nabla f ( x ) , \nabla f ( x ) \>_{\R^{d}} = c^{0} ( x )
\end{equation}
and
\begin{equation} \label{a20}
\< b^{0} ( x ) , \nabla f ( x ) \>_{\R^{d} } = 0 .
\end{equation}
 The first consequence of these identities is the following lemma whose proof is postponed to the next section.

\begin{lemma}\sl \label{a21}
If \eqref{a17} and \eqref{a24} hold true, then $\uuu \times \{ 0 \} \subset \ccc$. If in addition \eqref{a12}  is satisfied, then $\ccc = \uuu \times \{ 0 \}$.
\end{lemma}

Using this lemma, we obtain our first localization result on the spectrum of $P$. Its proof will also be  given in the next section.

\begin{proposition}\sl \label{a22}
Assume the hypotheses of Theorem \ref{a15}, \eqref{a17} and \eqref{a24}. There exist $h_{0} , \varepsilon_{*} > 0$ such that, for every $h \in ] 0 , h_{0} ]$, $P$ has exactly $n_{0}$ eigenvalues in $\{  z \in \C ; \ \Re z < \varepsilon_{*} h \}$. Moreover, these eigenvalues are of order $\ooo ( h^{1 + \alpha} )$, where $\alpha > 0$ is  given by Theorem \ref{a15}.
\end{proposition}

The aim of this paper is to give sharp asymptotics on these $n_{0}$ small eigenvalues of $P$. For this purpose, we recall the general labeling of minima introduced in \cite{HeKlNi04_01} and generalized in~\cite{HeHiSj11_01}. The presentation below 
originates from \cite{Mi19} and \cite{LePMi20}, where extra material can be found. The main ingredient is the notion of separating saddle point which is defined as follows. Note that, for a saddle point $\s \in \uuu^{( 1 )}$ of $f$ and $r > 0$ small enough, the set 
\begin{equation*}
\{ x \in D ( \s , r ) ; \ f ( x ) < f ( \s ) \} 
\end{equation*}
has exactly two connected components $C_{j} ( \s , r )$, $j = 1 , 2$.
Observe that this set is empty when $\s \in \uuu^{( 0 )}$ and  connected when $\s\in  \R^{d}\setminus(\uuu^{( 0 )}\cup\uuu^{( 1 )}$). The following definition comes from \cite[Definition 4.1]{HeHiSj11_01}.

\begin{defin}\sl \label{a23}
We say that $\s \in  \uuu^{( 1 )}$ is a separating saddle point of $f$  if, for every $r>0$ small enough, $C_{1} ( \s , r )$ and $C_{2} ( \s , r )$ are contained in two different connected components of $\{ x \in \R^{d} ; \ f ( x ) < f ( \s ) \}$. We will denote by $\vvv^{( 1)}$ the set made of these points.

We say that $\sigma \in \R$ is a separating saddle value of $f$ if it has the form $\sigma = f ( \s )$ for some~$\s \in \vvv^{( 1 )}$.

We say that $E \subset \R^{d}$ is a critical component of $f$ if there exists $\sigma \in  f ( \vvv^{( 1 )} )$ such that $E$ is a connected component of $\{ f < \sigma \}$ and $\partial E \cap \vvv^{( 1 )} \neq \emptyset$.
\end{defin}

Let us now describe the labeling procedure of \cite{HeHiSj11_01}. 
Assume  that $f(x)\to+\infty$ when $|x|\to+\infty$ and that $f$
satisfies \eqref{a24}. The set $f ( \vvv^{( 1 )} )$ is then finite.
We denote by $\sigma_{2} > \sigma_{3} > \cdots > \sigma_{N}$ its elements  and, for convenience, we also introduce a fictive infinite saddle value $\sigma_{1} = + \infty$. Starting from $\sigma_{1}$, we will recursively associate to each $\sigma_{i}$ a finite family of local minima $( \m_{i , j} )_{j}$ and a finite family of critical components $( E_{i , j} )_{j}$:

\begin{itemize}
\item[$\star$] Let $X_{\sigma_{1}} = \{ x \in \R^{d} ; \ f ( x ) < \sigma_{1} = + \infty \} = \R^{d}$. We let $\m_{1 , 1}$ be any global minimum of $f$ (not necessarily unique) and $E_{1 , 1} = \R^{d}$. In the following, we will write  $\underline{\m} = \m_{1 , 1}$.

\item[$\star$] Next, we consider $X_{\sigma_{2}} = \{ x \in \R^{d} ; \ f ( x ) < \sigma_{2} \}$. This is the union of its finitely many connected components. Exactly one of these components contains $\m_{1 , 1}$ and the other components are denoted by $E_{2 , 1} , \ldots , E_{2 , N_{2}}$. They are all critical and, in each component $E_{2 , j}$, we pick up a point $\m_{2 , j}$ which is a global minimum of $f_{\vert E_{2 , j}}$.

\item[$\star$] Suppose now that the families $( \m_{k , j} )_{j}$ and $( E_{k , j} )_{j}$ have been constructed until rank $k = i - 1$. The set $X_{\sigma_{i}} = \{ x \in \R^{d} ; \ f ( x ) < \sigma_{i} \}$ has again finitely many connected components and we label $E_{i , j}$, $j = 1 , \ldots , N_{i}$, those of these components that do not contain any $\m_{k , \ell}$ with $k < i$.
They are all critical and, 
in each $E_{i , j}$, we pick up a point $\m_{i , j}$ which is a global minimum of $f_{\vert E_{i , j}}$.
\end{itemize}
At the end of this procedure, all the minima have been labeled.

We now recall some constructions of \cite{Mi19} and \cite{LePMi20} that will be useful in the sequel. 
Throughout, we denote by $\s_{1}$ a fictive saddle point such that $f ( \s_{1} ) = \sigma_{1} = + \infty$  and, for any set $A$, ${\mathcal P} ( A )$ denotes the power set of $A$. From the above labelling, we define   two mappings
\begin{equation*}
E : {\mathcal U}^{( 0 )} \to {\mathcal P} ( \R^{d} ) \qquad \text{and} \qquad {\bf j} : {\mathcal U}^{( 0 )} \to {\mathcal P} ( \vvv^{( 1 )} \cup \{ \s_{1} \} ) 
\end{equation*}
as follows: for every $i \in \{ 1 , \dots , N \}$ and $j \in \{ 1 , \dots , N_{i} \}$,
\begin{equation} \label{a25}
E ( \m_{i , j} ) : = E_{i , j} ,
\end{equation}
and
\begin{equation} \label{a26}
{\bf j} ( \underline{\m} ) : = \{ \s_{1} \} \qquad \text{and} \qquad {\bf j} ( \m_{i , j} ) : = \partial E_{i , j} \cap \vvv^{( 1 )} \text{ for } i \geq 2 .
\end{equation}
In particular, we have $E ( \underline{\m} ) = \R^{d}$ and,
 for all
$ i, j \in \{ 1 , \dots , N \} $,  one has $\emptyset \neq {\bf j} ( \m_{i , j} ) \subset
\{ f = \sigma_{i} \} $.
We then define the mappings
\begin{equation*}
\bsigma : \uuu^{( 0 )} \rightarrow f ( \vvv^{( 1 )} ) \cup \{ \sigma_{1} \} \qquad \text{and} \qquad S : \uuu^{( 0 )} \rightarrow ] 0 , + \infty ] ,
\end{equation*}
by
\begin{equation} \label{a27}
\forall \m \in \uuu^{( 0 )} , \qquad \bsigma ( \m ) : = f ( {\bf j} ( \m ) ) \qquad \text{and} \qquad S ( \m ) : = \bsigma ( \m ) - f ( \m ) ,
\end{equation}
where, with a slight abuse of notation, we have identified the set $f ( {\bf j} ( \m ) )$ with its unique element. Note that $S ( \m ) = + \infty$ if and only if $\m = \underline{\m}$.

We are now in position to introduce our last assumption. In addition to \eqref{a17}, \eqref{h1}, and \eqref{a24}, we assume the following
\begin{equation} \label{a28} \tag{Gener}
\begin{aligned}
&\star\text{ for any } \m \in \uuu^{( 0 )} , \ \m \text{ is the unique global minimum of } f_{\vert E ( \m )} ,  \\
&\star\text{ for all } \m \neq \m^{\prime} \text{ in } \uuu^{( 0 )} , \ {\bf j} ( \m ) \cap {\bf j} ( \m^{\prime} ) = \emptyset .
\end{aligned}
\end{equation}
In particular, \eqref{a28} implies that $f$ uniquely attains its global minimum at $\underline \m \in \uuu^{( 0 )}$. This assumption is a generalization of  \cite[Assumption 5.1]{HeHiSj11_01} which was already used in \cite{LePMi20}. In Section \ref{b28}, we discuss how to remove this hypothesis and deal with  the general case, in the spirit of \cite{Mi19}.

In order to state our main result, we need the following lemma which is proved in Section~\ref{b35}. Throughout the paper, we denote $\C_{\pm} = \{ z \in \C ; \ \pm \Re z > 0 \}$ and $M^{t}$ the transpose of any matrix~$M$.

\begin{lemma}\sl \label{a29}
Assume \eqref{a12}, \eqref{a17}, and \eqref{a24}, and let $k\in\{0,\ldots,d\}$. Let $\u \in \uuu^{( k )}$ be a critical point
of $f$ of index $k$. Denote $B ( \u) = d b^{0} ( \u )$ and recall that $H ( \u )$ is the Hessian matrix of $f$ at $\u$. Then,

$i)$ the matrix $\Lambda ( \u ) := 2 H ( \u ) A^{0} ( \u ) + B^{t} ( \u )$ admits exactly $k$ eigenvalues in $\C_{-}$ and $d-k$ eigenvalues in $\C_+$,

$ii)$ if $k=1$, its unique  eigenvalue $\mu(\u)$ in $\C_{-}$  is real (and thus $\mu ( \u ) < 0)$.
\end{lemma}

We are now in position to state our main result.

\begin{theorem}\sl \label{a66}
Suppose that the assumptions of Theorem \ref{a15} are satisfied. Assume also that \eqref{a17}, \eqref{h1}, \eqref{a24} and \eqref{a28} hold true. Let $\varepsilon_{*}$ be given by Proposition \ref{a22}. There exists $h_{0} > 0$ such that, for all $h \in ] 0 , h_{0} ]$, one has, counting the eigenvalues with multiplicity,
\begin{equation*}
\sigma ( P ) \cap \{ \Re z < \varepsilon_{*} h \} = \{ \lambda ( \m , h ) ; \ \m \in \uuu^{( 0 )} \} ,
\end{equation*}
where $\lambda ( \underline{\m} , h ) = 0$ and, for all $\m \neq \underline{\m}$, $\lambda ( \m , h )$ satisfies the following Eyring--Kramers type formula
\begin{equation} \label{e12}
\lambda ( \m , h ) =  z ( \m ) h e^{- 2 S ( \m ) / h} a ( h ) ,
\end{equation}
where $a ( h ) \in \C$ admits a classical expansion $a ( h ) \sim 1 + \sum_{j \geq 1} a_{j} h^{j}$
with every $a_{j}$ real.
 Here, $S : \uuu^{( 0 )} \rightarrow ] 0 , + \infty ]$ is defined in \eqref{a27} and, for every $\m \in \uuu^{( 0 )} \setminus \{ \underline{\m} \}$,
\begin{equation} \label{e13}
z ( \m ) = \frac{( \det \Hess f ( \m ) )^{\frac{1}{2}}}{2 \pi} \Big( \sum_{\s \in {\bf j} ( \m )} \frac{\vert \mu ( \s ) \vert}{\vert \det \Hess f ( \s ) \vert^{\frac{1}{2}}} \Big) ,
\end{equation}
where ${\bf j} : \mathcal U^{( 0 )} \to \mathcal P ( \vvv^{( 1 )} \cup \{ \s_{1} \} )$ is defined in \eqref{a26} and $\mu ( \s )$ is given by Lemma \ref{a29}.
\end{theorem}

Let us make some comments on the above result.
 In \cite{HeHiSj11_01}, the authors studied the case where the operator satisfies 
 some supersymmetry property.
 More precisely, they assumed the existence of a smooth matrix-valued function $G ( x )$ such that $P = \Delta_{f , G}$, where 
\begin{equation} \label{a16}
 \Delta_{f , G} = d_{f}^{*} \circ G ( x ) \circ d_{f} ,
\end{equation}
and $d_{f}$ denotes the semiclassical Hodge derivative twisted ``\`a la Witten'': $d_{f} = e^{- f / h} \circ h d \circ e^{f / h}$ for some smooth function $f$.
  Under suitable assumptions on $f$ and $G$, $\Delta_{f , G}$ satisfies the general hypotheses of Theorem \ref{a15}. Moreover, one has obviously $\Delta_{f , G} ( e^{- f / h} ) = \Delta_{f , G}^{\dagger} ( e^{- f / h} ) = 0$. Assuming additionally that $f$ is a Morse function, the authors proved that the smallest eigenvalues of $ \Delta_{f , G}$ are exponentially small with respect to $h$, and established  Eyring--Kramers type formulas under suitable topological assumptions (see \cite[Theorem 5.10, Proposition 6.7, and Formula (6.71)]{HeHiSj11_01}). 
In their paper, the supersymmetry assumption is fundamental since, combined with the PT-symmetry property, it permits to follow the strategy used by Helffer, Klein and Nier \cite{HeKlNi04_01} in the supersymmetric framework of Witten Laplacians. More recently, the two last authors \cite{LePMi20} studied the case of non-reversible diffusions 
\begin{equation} \label{g27}
P = \Delta_{f} + b \cdot d_{f},
\end{equation}
where $\Delta_f=-h^2\Delta +|\nabla f|^2-h\Delta f$ denotes the Witten Laplacian and 
$b$ is a vector field verifying $\div b=0$ and $b\cdot\nabla f=0$. In this setting, which is not supersymmetric in general, the
authors
built accurate quasimodes and used them to prove Eyring--Kramers asymptotics.

The interest of the approach developed in the present paper is to deal simultaneously with all these models without 
assuming
additional symmetry. In particular, Theorem~\ref{a66} applies to both \eqref{a16} and \eqref{g27}. Moreover, we give in Appendix \ref{s4} examples of operators $P$ satisfying our assumption but which do not enjoy a nice supersymmetric structure \eqref{a16}. Compared to the results of \cite{HeHiSj11_01}, our approach has also the advantage to give formulas which are completely explicit in terms of the coefficients of the operator and of the function $f$. Moreover, compared to the results of \cite{LePMi20}, we would like to emphasize that we obtain a full asymptotic expansion of the prefactor $z(\m)a(h)$. The proof relies on the resolvent estimates of \cite{HeHiSj08-2} and on the construction of accurate quasimodes near the saddle points of $f$. These constructions, inspired by \cite{BoEcGaKl04,DiLe,LePMi20}, are the main novelty of this paper.
To be more precise, we generalize the use of Gaussian cut-off functions
introduced in
\cite{LePMi20} by using geometric constructions,
which lead to complete asymptotic expansions.

Theorem \ref{a66} permits to give the long time behavior of the solutions of the evolution equation associated with $P$,
\begin{equation} \label{g32}
\left\{ \begin{aligned}
& h \partial_{t} u + P u = 0 , \\
&u_{\vert t = 0} = u_{0} .
\end{aligned} \right.
\end{equation}
Since the operator $P$ is maximal accretive, this Cauchy problem has,
for every $u_{0} \in L^{2} ( \R^{d} )$, a unique solution
in $ C^{0} ( [ 0 , + \infty [ ; L^{2} ( \R^{d} ) )\cap C^{1} ( ] 0 , + \infty [ ; L^{2} ( \R^{d} ) )$, denoted by $u ( t ) = e^{- t P / h} u_{0}$. Modulo the conjugation by $e^{- f / h}$, \eqref{g32} is the Fokker--Planck equation \eqref{g23} in the case of the general Langevin process \eqref{g50}. 

We first state the spectral expansion of the propagator.

\begin{corollary}\sl \label{g30}
In the setting of Theorem \ref{a66}, there exist $C , \varepsilon > 0$ such that, for all $u_{0} \in L^{2} ( \R^{d} )$ and $h$ small enough, there exists $(u_{\m , n})_{\m,n} \subset \C$ such that the solution $u ( t )$ of \eqref{g32} satisfies
\begin{equation} \label{g31}
\forall t \geq 0 , \qquad \Big\Vert u ( t ) - \sum_{\m \in \uuu^{( 0 )}} \sum_{0 \leq n \leq n_{0} - 1} u_{\m , n} t^{n} e^{- \lambda ( \m , h ) t / h} \Big\Vert \leq C e^{- \varepsilon t} \Vert u_{0} \Vert .
\end{equation}
Moreover, for all $N \in \N$, there exists $C_{N} > 0$ such that, for all $u_{0} \in L^{2} ( \R^{d} )$ and $h$ small enough, the solution $u ( t )$ of \eqref{g32} satisfies
\begin{equation} \label{g49}
\forall t \geq 0 , \qquad \bigg\Vert u ( t ) - \frac{\< e^{- f / h} , u_{0} \>}{\Vert e^{- f / h} \Vert^{2}} e^{- f / h} \bigg\Vert \leq C_{N} e^{- t \min\limits_{\m \neq \underline{\m}} \Re(\lambda ( \m , h )) ( 1 - h^{N} ) / h} \Vert u_{0} \Vert .
\end{equation}
\end{corollary}

The double sum appearing in \eqref{g31} is nothing but $e^{- t P / h}\Pi_{h}$, where $\Pi_{h}$ denotes the spectral projector of $P$ associated with its $n_{0}$ exponentially small eigenvalues. In particular, when the $\lambda ( \m , h )$ are pairwise distinct, \eqref{g31} writes
\begin{equation}
u ( t ) = \sum_{\m \in \uuu^{( 0 )}} \Pi_{\m} ( u_{0} ) e^{- \lambda ( \m , h ) t / h} + \ooo ( e^{- \varepsilon t} ) \Vert u_{0} \Vert ,
\end{equation}
where $\Pi_{\m}$ is the rank-one spectral projector of $P$ associated with the eigenvalue $\lambda ( \m , h )$.

Estimate \eqref{g49} implies that $u_{\underline{\m} , 0}=\Pi_{\underline{\m}}u_{0} = \Vert e^{- f / h} \Vert^{- 2} \< e^{- f / h} , u_{0} \> e^{- f / h}$
and that $u_{\underline{\m} , n} = 0$ for every $n \geq 1$ in \eqref{g31}, where
$\Pi_{\underline{\m}}$
 is the (orthogonal) spectral projector 
 of $P$ associated with 
 the eigenvalue $0$ with corresponding eigenspace $e^{- f / h} \C$. Equation \eqref{g49} is a return to equilibrium formula generalizing \cite[Theorem 1.11]{LePMi20}. We see that the spectral gap (that is $\min_{\m \neq \underline{\m}} \Re(\lambda ( \m , h )) $) indeed gives the rate of convergence to the equilibrium state modulo $\mathcal O(h^{\infty})$.
In addition, when there exists precisely one $\m^{\star}\in \mathcal U^{(0)}\setminus\{\m\}$ such that 
\begin{equation*}
\lambda ( \m^{\star} , h ) = \min \limits_{\m \neq \underline{\m}} \Re ( \lambda ( \m , h ) ) \big ( 1 + \mathcal O ( h^{\infty} ) \big) ,
\end{equation*}
then the eigenvalue $\lambda(\m^{\star},h)$ is simple and real and we can replace $\min \limits_{\m \neq \underline{\m}} \Re ( \lambda ( \m , h ) ) ( 1 - h^{N} )$ by $\lambda ( \m^{\star} , h )$ in \eqref{g49}.

One can also show the metastable behavior of the solutions of \eqref{g32}. More precisely,

\begin{corollary}[Metastability]\sl \label{g33}
In the setting of Theorem \ref{a66}, let $S_{1} < \cdots < S_{K} = + \infty$ denote the increasing sequence of the $S ( \m )$'s defined in \eqref{a27} and let $\Pi^{\leq}_{k}$ be the spectral projector of $P$ associated with its eigenvalues of modulus less than $e^{- 2 S_{k} / h}$. For two positive functions $g_{\pm} ( h )$ such that $g_{-} ( h ) = \ooo ( h^{\infty} )$ and $g_{+}^{-1} ( h ) = o ( \vert \ln h \vert^{- 1} )$, we define the times
\begin{equation*}
t_{0}^{+} =  g_{+} ( h ) \qquad \text{and} \qquad \forall 1 \leq k \leq K , \quad t_{k}^{\pm} = g_{\pm} ( h ) e^{2 S_{k} / h}
\end{equation*}
(in particular, $t_{K}^{-} = + \infty$). Then, for every $h$ small enough, the solution $u ( t )$ of \eqref{g32} satisfies
\begin{equation} \label{g39}
\forall t_{k - 1}^{+} \leq t \leq t_{k}^{-} , \qquad u ( t ) = \Pi^{\leq}_{k} u_{0} + \ooo ( h^{\infty}) \Vert u_{0} \Vert ,
\end{equation}
uniformly with respect to $t$, $1 \leq k \leq K$, and $u_{0} \in L^{2} ( \R^{d} )$.
\end{corollary}

In other words, $e^{- t P / h} \approx \Pi^{\leq}_{k}$ in the time interval $[ t_{k - 1}^{+} , t_{k}^{-} ]$, whereas transitions occur around the times $t_{k} = e^{2 S_{k} / h}$. In this result, one can take for instance $g_{-} ( h ) = e^{- \delta / h}$, with $\delta > 0$, and $g_{+} ( h ) = \vert \ln h \vert^{2}$. Note that $\Pi^{\leq}_{j} \Pi^{\leq}_{k} = \Pi^{\leq}_{\max ( j , k )}$, $\Pi^{\leq}_{1}=\Pi_{h}$ is the spectral projector of $P$ associated with its $n_{0}$ exponentially small eigenvalues, and $\Pi^{\leq}_{K}=\Pi_{\underline{\m}}$ is the orthogonal projector on $e^{- f / h} \C$. The proofs of Corollaries \ref{g30} and \ref{g33} are done at the end of Section \ref{s3}.

We conclude this introduction by applying Theorem \ref{a66} to a generalization of the Kramers--Fokker--Planck operator defined in \eqref{g2}. For two smooth functions $V ( x )$ and $W ( v )$ and a friction coefficient $\gamma > 0$, the generator associated with
the SDE
\begin{equation}
\left\{
\begin{aligned}
d x_{t} &= \partial_{v} W ( v_{t} ) d t , \\
d v_{t} &= - \partial_{x} V ( x_{t} ) d t - \gamma \partial_{v} W ( v_{t} ) d t + \sqrt{2 \gamma h} \, d B_{t} ,
\end{aligned}
\right.
\end{equation}
is given by
\begin{equation} \label{g1}
\lll = -\partial_{v} W \cdot\partial_x+\partial_x V\cdot\partial_v+ \gamma \partial_{v} W \cdot\partial_v-\gamma h\Delta_v .
\end{equation}
This is an example of Hamiltonian SDE discussed in  \cite[Section 2.2.3]{LeRoSt10_01}. 
Defining $M_{h} ( x , v ) := e^{-  ( V ( x ) + W ( v ) ) / h}$, we have
\begin{equation*}
\lll ( 1 ) = 0 \qquad \text{and} \qquad \lll^\dagger ( M_{h} ) = 0 .
\end{equation*}
Hence, $P :=  M_{h}^{- 1 / 2} h \lll^{\dagger} M_{h}^{1 / 2}$ satisfies
\begin{equation*}
P ( e^{- f / h} ) = 0 \qquad \text{and} \qquad  P^{\dagger} ( e^{- f / h} ) = 0 ,
\end{equation*}
where 
$f ( x , v ) := ( V ( x ) + W ( v ) ) / 2$. Moreover, an immediate computation shows that
\begin{equation}
P = \partial_{v} W \cdot h \partial_{x} - \partial_{x} V \cdot h \partial_{v} + \gamma \Delta_{W / 2} ,
\end{equation}
where $\Delta_{W / 2} =(-h\partial_v+ \partial_{v} W / 2)\circ (h\partial_v+  \partial_{v} W / 2 )$ is the Witten Laplacian in the variable $v$ associated with the function $W ( v ) / 2$. Thus, $P$ has the form \eqref{a3} with 
\begin{equation*}
A(x,v) =
\left( \begin{array}{cc}
0 & 0 \\
0 & \gamma\Id
\end{array} \right) , \quad
b(x,v) =
\left( \begin{array}{c}
\partial_{v} W ( v ) \\
-\partial_{x} V ( x )
\end{array} \right)
\quad \text{and} \quad
c(x,v)=\frac{\gamma}{4} \vert \nabla_{v}W \vert^{2} - h \frac{\gamma}{2}\Delta W  .
\end{equation*}
Of course, $f$ satisfies \eqref{a24} if and only if $V$ and $W$ do it. In that case, Corollary \ref{g10}, Remark \ref{g11} and
\begin{equation*}
B := db^{0} = \left( \begin{array}{cc}
0 & \Hess_{v} W \\
- \Hess_{x} V & 0
\end{array} \right) ,
\end{equation*}
show that \eqref{a12} is satisfied. Under some additional growth assumptions on $V$ and $W$ at infinity, one can verify that \eqref{h1} and \eqref{a14} hold true. At a critical point~$\u$, the matrix $\Lambda ( \u )$ of Lemma \ref{a29} is given by
\begin{equation*}
\Lambda ( \u ) = \left( \begin{array}{cc}
0 & - \Hess_{x} V \\
\Hess_{v} W & \gamma \Hess_{v} W
\end{array} \right) .
\end{equation*}
The setting of \eqref{g2} corresponds to $W ( v ) = v^{2} / 2$. In that case, the saddle points of $f$ are of the form $\s = ( \x , 0 )$, where $\x$ is a saddle point of $V$, and the unique eigenvalue with negative real part of $\Lambda ( \s )$ is
\begin{equation} \label{g3}
\mu ( \s ) = \frac{1}{2} \big( \gamma - \sqrt{\gamma^{2} - 4 \lambda_{1}} \big) ,
\end{equation}
where $\lambda_{1}$ is the unique negative eigenvalue of $\Hess_{x} V ( \x )$. This yields an explicit expression of the prefactor $z(\m)$ in \eqref{e13} and we observe that it has the same form as the one obtained in equation (6.67) of \cite{HeHiSj11_01}. Observe also that $\vert \mu ( \s ) \vert < \vert \lambda_{1} \vert$ if and only if $\gamma > 1 + \lambda_{1}$. Thus, the rate of convergence to the equilibrium  for the Langevin process \eqref{eq.lang}
(whose generator is the Kramers--Fokker--Planck operator, see \eqref{g2})
 is smaller than for the overdamped Langevin process \eqref{g24} (generated by the Witten Laplacian, see the lines above \eqref{g25}) if and only if this condition holds. Note that this discussion is more generally valid if $W$ admits a unique critical point at $v = \v$ with $\Hess_{v} W ( \v ) = \Id$. In particular, it is easy to choose $W$ such that~$P$ is not PT-symmetric, which provides a setting covered by Theorem \ref{a66} but which can not be treated using \cite{HeHiSj11_01}.
 
More sophisticated  Langevin equations have been  considered in the literature, as the generalized Langevin equation in \cite{PaStVa} (see also references therein). Our results permit to compute
in the low temperature regime
the low-lying eigenvalues  of the different generators treated in
  \cite{PaStVa}.

The rest of this paper is organized as follows.
In Section~\ref{b35}, we derive algebraic and geometric results 
from our assumptions. This leads to the proofs of  Lemmas \ref{a21} and \ref{a29}, and of 
the rough localization of the spectrum of $P$ stated in Proposition \ref{a22}.
Section~\ref{s20} is devoted to the construction of new ans\"atze
of the eigenmodes of $P$ near the saddle points of~$f$. These ans\"atze
are then used in Section~\ref{s21}  to construct global quasimodes.
Afterwards, in Section~\ref{s3}, we prove our main results, Theorem~\ref{a66} and Corollaries~\ref{g30} and~\ref{g33}.
The aim of Section~\ref{b28} is then to show that 
Theorem~\ref{a66} can be generalized to an arbitrary Morse function~$f$, that is
without assuming  \eqref{a28}. A vague statement is given in Theorem~\ref{b34},
while precise ones are given in Theorems~\ref{b23} and~\ref{b20}.
Let us also recall here that
we are not working with symmetry assumptions such as supersymmetry or PT-symmetry.
This prevents us from concluding as in \cite{HeKlNi04_01} (and in many other works later such as, e.g., in 
\cite{HeHiSj11_01,Mi19}) 
after reduction of the problem to the computation of the eigenvalues of a square matrix of size $n_{0}$.
To handle this computation, we then
use crucially the Schur complement method, as in \cite{LePMi20},
and refine  \cite[Theorem A.4]{LePMi20} in Theorem~\ref{b10}.
We believe that Theorem~\ref{b10} has his own interest and may be used in other contexts. The authors are supported by the ANR project QuAMProcs 19-CE40-0010-01.

\section{Preliminary results} \label{b35}

In this section, we prove some preparatory geometric results and use them to show the rough 
localization of the spectrum of $P$ stated in Proposition  \ref{a22}.  For $\u$ a critical point of $f$, we use the shortcuts
\begin{equation*}
A^{0} = A^{0} ( \u ) , \qquad B = d b^{0} ( \u ) \qquad \text{and} \qquad H = \Hess f ( \u ) .
\end{equation*}

\subsection{Analysis of the critical set}

The aim of this section is to prove Lemma \ref{a21} and to discuss the assumptions of Section \ref{s1}. We start  with a microlocal observation.

\begin{lemma}\sl \label{g4}
$i)$ Assume \eqref{a12} and let $( \u , 0 ) $ belong to $\ccc$. Then,
\begin{equation} \label{g51}
\text{the real symmetric matrix } \int_{- T}^{T} e^{- t B} A^{0} e^{- t B^{t}} d t \text{ is positive definite}.
\end{equation}
$ii)$ When \eqref{g51} holds true, we have the identity $\ker ( A^{0} ) \cap \ker ( B^{t} - z ) = \{ 0 \}$ for every $z \in \C$.
\end{lemma}

\begin{proof}
Since the Hamilton vector field $H_{p_{1}^{0}}$ at $( \u , \xi )$ equals $( 0 , - B^{t} \xi )$ when $\xi \in \R^{d}$, it follows that, for every $\xi \in \R^{d}$ (see \eqref{a9} and \eqref{a10}),
\begin{align}
\< \widetilde{p} \>_{T} ( \u , \xi ) = \frac{1}{2 T} \int_{- T}^{T} \widetilde{p} \big( \u , e^{- t B^{t}} \xi \big) \, d t & = \frac{1}{2 T} \int_{- T}^{T} \frac{p^{0}_{2} \big( \u , e^{- t B^{t}} \xi \big)}{\< e^{- t B^{t}} \xi \>^{2}} \, d t   \nonumber  \\
&\leq \frac{1}{2 T} \int_{- T}^{T} \big\< A^{0} e^{- t B^{t}} \xi , e^{- t B^t} \xi \big\> \, d t ,   \label{g7}
\end{align}
where we used that $( \u , 0 ) \in \ccc $ and thus $c ( \u ) = 0$ to obtain the second equality. Since the relation \eqref{a12} implies $\< \widetilde{p} \>_{T} ( \u , \xi ) \geq \frac{1}{C} \vert \xi \vert^{2}$ for some constant $C>0$ and  every $\xi$ small enough, the first part of Lemma \ref{g4} is then an immediate consequence of \eqref{g7}.

To prove the second part of Lemma \ref{g4}, let $\eta \in \C^{d}$ belong to $\ker ( A^{0} ) \cap \ker ( B^{t} - z )$ for some $z\in\C$. Then $e^{- t B^t} \eta= e^{- t z} \eta$, and thus $\eta \in \ker \big( \int_{- T}^{T} e^{- t B} A^{0} e^{- t B^{t}} \, d t \big) = \{ 0 \}$.
\end{proof}

For a not necessarily Morse function $f$ (that is without assuming \eqref{a24}), $\uuu$ (resp. $\widetilde{\uuu}$) denote the set of critical (resp. non-degenerate critical) points of $f$.

\begin{lemma}\sl \label{g5}
$i)$ If \eqref{a17} holds true, then
\begin{equation} \label{g8}
\widetilde{\mathcal U} \times \{ 0 \} \subset \ccc .
\end{equation}
$ii)$ If \eqref{a17} and \eqref{g51} hold true, then
\begin{equation} \label{g9}
\ccc \subset \uuu \times \{ 0 \} .
\end{equation}
\end{lemma}

\begin{proof}
Suppose that $\u \in \widetilde{\uuu}$ and assume without loss of generality that $\u = 0$. Thanks to \eqref{a19}, we have $c^{0} ( 0 ) = 0$. Moreover, one has near the origin
\begin{equation*}
\nabla f ( x ) = H x + \ooo ( x^{2} ) \qquad \text{and} \qquad b^{0} ( x ) = b^{0} ( 0 ) + \ooo ( x ) ,
\end{equation*}
which, combined with \eqref{a20}, yields
\begin{equation*}
\forall x \in \R^{d} , \qquad \< b^{0} ( 0 ) , H x \> = 0 ,
\end{equation*}
and thus $H b^{0} ( 0 ) = 0$. Since $\u=0$ is assumed to be non-degenerate, $H$ is invertible and hence $b^{0} ( 0 ) = 0$. This proves \eqref{g8}.

We now show \eqref{g9}. Let $( \u , 0 )$ belong to $\ccc$ and assume without loss of generality that  $\u = 0$. Then $b^{0} ( 0 ) = 0$ and $c^{0} ( 0 ) = 0$. Denoting $\eta : = \nabla f ( 0 )$, it follows from \eqref{a19} that $\< A^{0} \eta , \eta \> = 0$. 
Since $A^{0}$ is positive semidefinite, this implies $A^{0} \eta = 0$.  On the other hand, \eqref{a20} yields
\begin{equation*}
\forall x \in \R^{d} , \qquad 0 = \< b^{0} ( 0 ) + B x + \ooo ( x^{2} ) , \eta + \ooo(x) \> = \< B x + \ooo ( x^{2} ) , \eta + \ooo ( x ) \> ,
\end{equation*}
and hence $B^{t} \eta = 0$. It follows that $\eta \in \ker ( A^{0} ) \cap \ker ( B^{t} )$ and then $\eta = 0$ thanks to 
Lemma \ref{g4} $ii)$. This completes the proof of \eqref{g9}.
\end{proof}

\begin{proof}[Proof of Lemma \ref{a21}]
We can deduce Lemma \ref{a21} from Lemma \ref{g5}. Indeed, the sets $\widetilde{\mathcal U}$ and $\uuu$ coincide when $f$ is a Morse function. Then, Lemma \ref{g5} $i)$ provides the first statement of Lemma \ref{a21}. On the other hand, \eqref{a12} and \eqref{a17} imply \eqref{g51} from Lemma \ref{g4} $i)$ and then \eqref{g9}. Thus, Lemma \ref{g5} gives $\ccc = \uuu \times \{ 0 \}$ in that case, finishing the proof of Lemma~\ref{a21}.
\end{proof}

\begin{lemma}\sl \label{a40}
Assume \eqref{a17} and let $\u \in \uuu$ with $b ( \u ) = 0$. Then, the matrix $B^{t} H$ is antisymmetric. If moreover the critical point $\u$ is non-degenerate (i.e. $u \in \widetilde{\uuu}$), the matrix $J : =H^{- 1} B^{t}$ is also antisymmetric.
\end{lemma}

\begin{proof}
As before, we assume for simplicity that $\u = 0$. Performing a Taylor expansion of the identity $b ( x ) \cdot \nabla f ( x ) = 0$ (see \eqref{a20}), we deduce $B x \cdot H x = 0$ for every $x \in \R^{d}$. Therefore, $x \cdot B^{t} H x = 0$ for every $x \in \R^{d}$, which implies that the matrix $\widetilde{J} = B^{t} H$ is antisymmetric.

Moreover, if in addition $H$ is invertible, the matrix $J = H^{-1}B^{t} = H^{-1} \widetilde{J} H^{- 1}$ is antisymmetric since $H$ is symmetric.
\end{proof}

\begin{corollary}\sl \label{g10}
Let us assume \eqref{a17}. Then, the condition \eqref{a12} is satisfied if and only if, for every $( \u , 0 ) \in \ccc$,
\begin{equation*}
H \text{ is invertible and the symmetric matrix } \int_{- T}^{T}  e^{- t B} A^{0} e^{- t B^{t}} \, d t \text{ is positive definite.}
\end{equation*}
\end{corollary}

\begin{remark}\sl \label{g11}
Mimicking the Kalman criterion for controllability of finite-dimensional systems (see  \cite[Theorem 1.16]{Co07_01}), one can show that $\int_{- T}^{T} e^{- t B} A^{0} e^{- t B^{t}} \, d t$ is positive definite if and only if
\begin{equation*}
\bigplus_{n = 0}^{d - 1} \Im ( B^{n} A^{0} ) = \R^{d} \qquad \text{or} \qquad \bigcap_{n = 0}^{d - 1} \ker \big( A^{0} ( B^{t} )^{n} \big) = \{ 0 \} .
\end{equation*}
\end{remark}

\begin{proof}[Proof of Corollary~\ref{g10}]
We first prove a formula for $\< \widetilde{p} \>_{T}$ near $( \u , 0 ) \in \ccc$ with $\u \in \uuu$. Without loss of generality, we can assume that $\u = 0$. Denoting $( x , \xi )$ by $\rho$, \eqref{a9} and \eqref{a19} yield near $(0,0)$
\begin{equation*}
\begin{aligned}
\widetilde{p} ( \rho ) &= p^{0}_{2} ( x , \xi ) + c^{0} ( x ) + \ooo ( \rho^{3} )  \\
&= A^{0} \xi \cdot \xi + A^{0} H x \cdot H x + \ooo ( \rho^{3} ) .
\end{aligned}
\end{equation*}
On the other hand, since $0 = ( 0 , 0 )$ is a critical point of $H_{p_{1}^{0}}$, then $d \exp ( t H_{p_{1}^{0}} ) ( 0 ) = \exp ( t F_{p_{1}^{0}} )$, where $F_{p_{1}^{0}}$ is the linearization of $H_{p_{1}^{0}}$ at $0$. Using $p_{1}^{0} ( x , \xi ) = B x \cdot \xi + \ooo ( \rho^{3} )$, it holds $F_{p_{1}^{0}} ( x , \xi ) = ( B x , - B^{t} \xi )$. Consequently, one has
near $(0,0)$
\begin{align*}
\< \widetilde{p} \>_{T} ( \rho ) &= \frac{1}{2 T} \int_{- T}^{T} \< A^{0} H e^{t B} x , H e^{t B} x \> + \< A^{0} e^{- t B^{t}} \xi , e^{- t B^t} \xi \> + \ooo ( \rho^3 ) \, d t \\
&= \frac{1}{2 T} \int_{- T}^{T} \< A^{0} H e^{t B} x , H e^{t B} x \> + \< A^{0} e^{- t B^{t}} \xi , e^{- t B^t} \xi \>  \, d t + \ooo ( \rho^3 ) ,
\end{align*}
from \eqref{a10}. Since moreover $B^{t} H$ is antisymmetric according to Lemma \ref{a40}, we have $H B = - B^{t} H$ and hence $H e^{t B} = e^{- t B^{t}}H $. This leads to
\begin{equation} \label{g12}
\< \widetilde{p} \>_{T} ( \rho )= \frac{1}{2 T} \int_{- T}^{T} \< A^{0} e^{- t B^{t}}  H x , e^{- t B^{t}}  H  x \> + \< A^{0} e^{- t B^{t}} \xi , e^{- t B^t} \xi \> \, d t + \ooo ( \rho^3 ) .
\end{equation}

Let us now prove the corollary and consider $( \u , 0 ) \in \ccc$. If \eqref{a12} is satisfied, Lemma~\ref{g4}~$i)$ and Lemma~\ref{g5}~$ii)$ imply that $\u \in \uuu$. Then, we can apply \eqref{g12} which, combined with \eqref{a12}, shows that the matrices 
\begin{equation*}
H \big( \int_{- T}^{T} e^{- t B} A^{0} e^{- t B^{t}} \, d t \big) H \text{ and } \int_{- T}^{T} e^{- t B} A^{0} e^{- t B^{t}} \, d t \text{ are positive definite}.
\end{equation*}
This latter formula is equivalent to
\begin{equation} \label{g52}
H \text{ is invertible and the symmetric matrix } \int_{- T}^{T}  e^{- t B} A^{0} e^{- t B^{t}} \, d t \text{ is positive definite},
\end{equation}
proving the direct implication of Corollary~\ref{g10}. On the other hand, if \eqref{g52} holds true, 
then~\eqref{g51} holds true and
Lemma~\ref{g5}~$ii)$ gives $\u \in \uuu$. Then, \eqref{g12} and \eqref{g52} imply \eqref{a12}, proving the converse implication.
\end{proof}

\subsection{The spectrum of the matrix $\Lambda$}

The aim of this section is to  provide informations on  the spectrum of the matrix $\Lambda = 2 H A^{0} + B^{t}$ and to prove Lemma \ref{a29}. We start with the following result, where $J=H^{-1} B^{t}$
is antisymmetric according to Lemma~\ref{a40}.

\begin{lemma}\sl \label{a53}
Assume \eqref{a12}, \eqref{a17} and \eqref{a24}. For any $\u\in\uuu$ and  $r \in [ 0 , 1 ]$, the matrix 
\begin{equation*}
\Lambda_{r} : = H  \big( r ( 2 A^{0}  + J ) + ( 1 - r ) \Id \big) ,
\end{equation*}
has no eigenvalue on the imaginary axis $\{ \Re z = 0 \}$.
\end{lemma}

\begin{proof}
Suppose that there exists $v \in \C^{d}$ such that $\Lambda_{r} v = z v$ with $\Re z = 0$. Then, using $\< v , H^{- 1} v \>_{\C^{d}} \in \R$, we get
\begin{equation} \label{a54}
\Re \< \Lambda_{r} v , H^{- 1} v \>_{\C^{d}} = \Re \big( z \< v , H^{- 1} v \>_{\C^{d}} \big) = 0 .
\end{equation}
On the other hand,
\begin{align}
\Re \< \Lambda_{r} v , H^{- 1} v \>_{\C^{d}} &= 2 r \Re \< A^{0} v , v\>_{\C^{d}} + r \Re \< J v , v\>_{\C^{d}} + ( 1 - r ) \Re \< v , v \>_{\C^{d}}   \nonumber \\
&= 2 r \< A^{0} v , v\>_{\C^{d}} + ( 1 - r ) \Vert v \Vert^{2} .  \label{a55}
\end{align}
For $r \in [ 0 , 1 [$, \eqref{a54} and \eqref{a55} imply $v = 0$ and the lemma follows in that case. Assume now that $r = 1$. The relations \eqref{a54} and \eqref{a55} yield $\< A^{0} v , v\> = 0$ and then $A^{0} v = 0$ since~$A^{0}$ is positive semidefinite. Thus, the eigenvalue equation $\Lambda v = z v$ writes $B^{t} v = z v$. 
Hence, $v\in \ker ( A^{0} ) \cap \ker ( B^{t}-z )$ and thus $v=0$ according to Lemma \ref{g4}, which
completes the proof in the case $r=1$. 
\end{proof}

We now give the proof of  Lemma \ref{a29}. Let $k\in\{0,\dots,d\}$ and let $\u\in\uuu^{(k)}$. For $r = 0$, $\Lambda_{0}  = H $ has exactly $k$ eigenvalues in $\{ \Re z < 0 \}$ since $\u$ is a critical point of index $k$. Using that the eigenvalues of $\Lambda_{r} $ are continuous with respect to $r$ and cannot cross the imaginary axis from Lemma \ref{a53}, $\Lambda_{1}  = \Lambda $ has exactly $k$ eigenvalues in  $\{ \Re z < 0 \}$ and no eigenvalue on the imaginary axis. This proves the $i)$ of the lemma. Suppose now that $k=1$ and let $\mu(\u)$ denote the unique eigenvalue of $\Lambda $ in $\{ \Re z < 0 \}$.  Since $\Lambda $ is a real matrix, its set of eigenvalues is stable by complex conjugation and hence $\mu ( \u ) \in \R$. This proves the $ii)$.

\subsection{Rough localization of the spectrum}

In this section, we prove Proposition \ref{a22}. The arguments are close to those of  \cite[Section 2.2.2]{HeHiSj11_01}. Thanks to Lemma \ref{a21}, one has $\ccc = \uuu \times \{ 0 \}$ and hence it suffices to show that, for any $\rho = ( \u , 0 ) \in \uuu \times \{ 0 \}$, the numbers $\mu_{\rho , 0}^{0}$ of Theorem~\ref{a15} satisfy the following property
\begin{equation} \label{e9}
\left\{ \begin{aligned}
&\mu_{\rho , 0}^{0} = 0 \text{ for all } \u \in \uuu^{( 0 )} ,   \\
&\Re \mu_{\rho , 0}^{0} > 0 \text{ for all } \u \in \uuu^{( k )} , k \geq 1 ,
\end{aligned} \right.
\end{equation}
where we recall that $\mu_{\rho , 0}^{0} = \frac{1}{2} \widetilde{\tr} ( p , \rho )$ with $\widetilde{\tr} ( p , \rho ) = - i \sum_{\ell=1}^{d} \lambda_{\rho , \ell} +  2 c^{1} ( \u )$, and $\pm \lambda_{\rho , \ell}$, $\ell\in\{1,\dots,d\}$, denote the eigenvalues of the 
fundamental matrix $F_{p^{0}}$ of $p^{0}$ at $\rho = ( \u , 0 )$ with the convention $\Im \lambda_{\rho , \ell} > 0$ (see the paragraph above Theorem \ref{a15}). From now on, we take $\u \in \uuu^{( k )}$, $k \geq 0$, and we suppose 
again
that $\u = 0$ in order to lighten the notations. Thanks to \eqref{a19}, one has near $(0,0)$
\begin{align*}
p^{0} ( x , \xi ) &= \< A^{0} ( x ) \xi , \xi \> + i b^{0} ( x ) \cdot \xi + c^{0} ( x )   \\
&= \< A^{0} ( x ) \xi , \xi \> + i b^{0} ( x ) \cdot \xi + \< A^{0} ( x ) \nabla f ( x ) , \nabla f ( x ) \>   \\
&=  \< A^{0} \xi , \xi \> + i B x \cdot \xi + \< A^{0} H x , H x \> + \ooo ( ( x , \xi )^{3} ) .
\end{align*}
Then, 
\begin{equation} \label{g13}
F_{p^{0}} = \left( \begin{array}{cc}
i B & 2 A^{0} \\
- 2 H A^{0} H & - i B^{t}
\end{array} \right) .
\end{equation}
Using the identity $B^{t} H = - H B$ which follows from Lemma \ref{a40}, a direct computation shows that $L_{\pm} = \{ ( X , \pm i H X ) ; \ X \in \C^{d} \}$ are vector spaces stable under $F_{p^{0}}$, that $F_{p^{0}}$ restricted to $L_{\pm}$ acts like
\begin{equation*}
F_{\pm} = i ( \pm 2 A^{0} H + B ) ,
\end{equation*}
and that $\C^{2d} = L_{+} \oplus L_{-}$ since $H$ is invertible. In particular, $\sigma ( F_{p^{0}} ) = \sigma ( F_{+} ) \cup \sigma ( F_{-} )$. It follows moreover from Lemma \ref{a29} that $F_{+} = i ( 2 H A^{0} + B^{t} )^{t}$ has $k$ eigenvalues
$\lambda^{+}_{1} , \ldots , \lambda^{+}_{k}$ in $\{ \Im z < 0 \}$ and $d - k$ eigenvalues $\lambda^{+}_{k + 1} , \ldots , \lambda^{+}_{d}$ in $\{ \Im z > 0 \}$. In addition, using again $B^{t} H = - H B$, we get
\begin{equation*}
F_{-} = i ( - 2 A^{0} H + B ) = i H^{- 1} ( - 2 H A^{0} + H B H^{- 1} ) H = - H^{- 1} i ( 2 H A^{0} + B^{t} ) H = - H^{- 1} ( F_{+} )^{t} H .
\end{equation*}
Hence, $\sigma ( F_{-} ) = - \sigma ( F_{+} )$, which proves that 
\begin{equation} \label{e10}
\sigma ( F_{p^{0}} ) \cap \{ \Im z > 0 \} = \{ - \lambda^{+}_{1} , \ldots , - \lambda^{+}_{k} \} \cup \{ \lambda^{+}_{k + 1} , \ldots , \lambda^{+}_{d} \}=\{ \lambda_{\rho , 1},\ldots, \lambda_{\rho , d}\} .
\end{equation}
On the other hand, let us recall that \eqref{a17} implies
$( - h \div \circ A \circ h \nabla + c ) ( e^{- f / h} ) = 0$ (see the lines above \eqref{a19}).
Using the classical expansions of the coefficients and looking at the terms of order $1$  then shows that
\begin{equation*}
2 c^{1} ( \u ) = - 2 \sum_{j , \ell = 1}^{d} A^{0}_{j , \ell} ( \u ) \partial_{x_{j}} \partial_{x_{\ell}} f ( \u ) = - \tr ( 2 A^{0} H ) .
\end{equation*}
Thus, since $\tr ( B ) = \tr ( - H^{- 1} B^{t} H ) = - \tr ( B^{t} ) = - \tr ( B ) = 0$
thanks to $B^{t} H = - H B$,  
\begin{equation} \label{e11}
2 c^{1} ( \u ) = - \tr ( 2 A^{0} H + B ) = i \tr ( F_{+} ) .
\end{equation}
Combining \eqref{e10} and \eqref{e11} yields
\begin{equation*}
\widetilde{\tr} ( p , \rho ) = - i \sum_{\ell=1}^{d} \lambda_{\rho , \ell} 
+2 c^{1} ( \u )
=i \sum_{j = 1}^{k} \lambda^{+}_{j} - i \sum_{j = k + 1}^{d} \lambda^{+}_{j} 
+ i \sum_{j = 1}^{d} \lambda^{+}_{j} = 2 i \sum_{j = 1}^{k} \lambda^{+}_{j} .
\end{equation*}
By definition, this latter quantity vanishes when $k = 0$ and has a positive real part when $k \geq  1$. This proves \eqref{e9}.

\section{Quasimodal constructions near the saddle points}
\label{s20}

In this part, we assume the hypotheses of Theorem \ref{a15}, \eqref{a17} and \eqref{a24}.

\subsection{A first step in the construction of quasimodes}

Given $\s \in \uuu^{( 1 )}$, we look for an approximate solution to the equation $P u = 0$ in a neighborhood $V$ of $\s$ under the form
\begin{equation*}
u ( x ) = v ( x , h ) e^{- f ( x ) / h} ,
\end{equation*}
with a function $v$ of the form
\begin{equation} \label{a31}
v ( x , h ) = \int_{0}^{\ell ( x , h )} \zeta ( s / \tau ) e^{- s^{2} / 2 h} d s ,
\end{equation}
where the function $\ell ( x , h )\in C^{\infty} ( V)$ 
has a classical expansion $\ell ( x , h ) \sim \sum_{j \geq 0} h^{j} \ell_{j} ( x )$ in $C^{\infty} ( V)$
with $\ell_{0}\not\equiv 0$. Here, $\zeta$ denotes a fixed smooth even function equal to $1$ on $[ - 1 , 1 ]$ and supported in $[ - 2 , 2 ]$, and $\tau > 0$ is a small parameter which will be fixed later. The object of this section is to construct the function $\ell$.

\begin{lemma}\sl \label{a34}
One has 
\begin{equation*}
P ( v e^{- f / h} ) = ( w + r ) e^{- ( f + \frac{\ell^{2}}{2} ) / h} ,
\end{equation*}
where
\begin{equation*}
w = h \big( 2 A \nabla \ell \cdot \nabla f + ( A \nabla \ell \cdot \nabla \ell ) \ell + b \cdot \nabla \ell \big) - h^{2} \div ( A \nabla \ell ),
\end{equation*}
the function $r$ and all its derivatives are (locally) bounded, uniformly  with respect to $h$, and $\supp ( r) \subset \{ \vert \ell \vert \geq \tau \}$.
Here, we recall that all the functions above depend on $x$ and $h$. Moreover, $w$ admits a classical expansion $w \sim \sum_{j \geq 1} h^{j} w_{j}$ with 
\begin{equation*}
w_{1} = 2 A^{0} \nabla f \cdot \nabla \ell_{0} + ( A^{0} \nabla \ell_{0} \cdot \nabla \ell_{0} ) \ell_{0} + b^{0} \cdot \nabla \ell_{0} ,
\end{equation*}
and, for all $j \geq 1$,
\begin{equation*}
w_{j + 1} = 2 A^{0} \nabla f \cdot \nabla \ell_{j} + ( A^{0} \nabla \ell_{0} \cdot \nabla \ell_{0} ) \ell_{j} + 2 ( A^{0} \nabla \ell_{0} \cdot \nabla \ell_{j} ) \ell_{0} + b^{0} \cdot \nabla \ell_{j} + R_{j} ( x , \partial^{\alpha}\ell_{k} ) ,
\end{equation*}
where $R_{j}$ is a polynomial of $\partial^{\alpha}\ell_{k} $, $|\alpha|\leq 2$ and $k\in\{0,\dots,j-1\}$,
with smooth coefficients. 
\end{lemma}

\begin{proof}
Throughout the proof, we write $\zeta ( \cdot )$ instead of $\zeta ( \cdot / \tau )$. Recall that $P = P_{2} + P_{1} + P_{0}$ with 
$P_{0} = c $, $P_{1} = \frac{1}{2} (  b \cdot h \nabla + h \div \circ b )$, and $P_{2} = - h \div \circ A \circ h \nabla$. One has of course 
\begin{equation} \label{a32}
[ P_{0} , v ] = 0 .
\end{equation}
Using \eqref{a31}, one gets
\begin{equation} \label{a33}
[ P_{1} , v ] = h ( b \cdot \nabla \ell ) \zeta ( \ell ) e^{- \ell^{2} / 2 h} = h ( ( b \cdot \nabla \ell ) + r_{1} ) e^{- \ell^{2} / 2 h} ,
\end{equation}
with $r_{1} = ( b \cdot \nabla \ell ) ( \zeta ( \ell ) - 1 )$. In particular, $r_{1}$ and all its derivatives are (locally) bounded, uniformly with respect to $h$, and $\supp ( r_{1} ) \subset \{ \vert \ell \vert \geq \tau \}$. On the other hand, since the matrix $A$ is symmetric, one has for any smooth function $\psi$, 
\begin{equation*}
P_{2} ( v \psi ) = - h \div ( A h \nabla ( v \psi ) ) = \psi P_{2} v + v P_{2} \psi - 2 A h \nabla v \cdot h \nabla \psi .
\end{equation*}
Using again \eqref{a31}, this yields
\begin{equation*}
[ P_{2} , v ] = e^{- \ell^{2} / 2 h} \big( - 2 h \zeta ( \ell ) A \nabla \ell \cdot h \nabla - h^{2} \div ( \zeta ( \ell ) A \nabla \ell ) + h \ell \zeta ( \ell ) A \nabla \ell \cdot \nabla \ell \big) ,
\end{equation*}
and hence
\begin{equation*}
\begin{aligned}
{}^{} [ P_{2} , v ] ( e^{- f / h} ) &= \big( h \zeta ( \ell ) ( 2  A \nabla \ell \cdot \nabla f + \ell  A \nabla \ell \cdot \nabla \ell ) - h^{2} \div ( \zeta ( \ell ) A \nabla \ell ) \big) e^{- ( f + \frac{\ell^{2}}{2} ) / h}   \\
&= \big( h ( 2 A \nabla \ell \cdot \nabla f + \ell A \nabla \ell \cdot \nabla \ell ) - h^{2} \div ( A \nabla \ell ) + r_{2} \big) e^{- ( f + \frac{\ell^{2}}{2} ) / h} ,
\end{aligned}
\end{equation*}
with $r_{2} = h ( \zeta ( \ell ) - 1 ) ( 2 A \nabla \ell \cdot \nabla f + \ell A \nabla \ell \cdot \nabla \ell ) - h^{2} \div ( ( \zeta ( \ell ) - 1 ) A \nabla \ell )$. In particular, $r_{2}$ and all its derivatives are (locally) bounded, uniformly with respect to $h$, and $\supp ( r_{2} ) \subset \{ \vert \ell \vert \geq \tau \}$. Combining this identity with \eqref{a32} and \eqref{a33},
and using the relation $P(ve^{- f / h})=[ P , v ] ( e^{- f / h} )$ implied by \eqref{a17}, we obtain the first part of the statement. Moreover, since the coefficients of $P$ and the function $\ell$ have a classical expansion, so do $w$. Plugging the expansions of $A$, $b$, $c$, and $\ell$ into the expression of $w$ and identifying the powers of $h$, we obtain  the formulas for the $w_{j}$.
\end{proof}

In order to construct accurate quasimodes, we have to find smooth functions $\ell_{j}$, $ j \geq 0$, with $\ell_{0}\not\equiv0$ 
and such that the above $w_{j + 1}$ vanish. The equation on $\ell_{0}$ is 
\begin{equation} \label{a35}
2 A^{0} \nabla f \cdot\nabla \ell_{0} + ( A^{0} \nabla \ell_{0} \cdot \nabla \ell_{0} ) \ell_{0} + b^{0} \cdot \nabla \ell_{0} = 0 ,
\end{equation}
and the equations on the $\ell_{j}$, $j \geq 1$, are
\begin{equation} \label{a36}
2 A^{0} \nabla f \cdot \nabla \ell_{j} + ( A^{0} \nabla \ell_{0} \cdot \nabla \ell_{0} ) \ell_{j} + 2 ( A^{0} \nabla \ell_{0} \cdot \nabla \ell_{j} ) \ell_{0} + b^{0} \cdot \nabla \ell_{j} = - R_{j}.
\end{equation}
By analogy with the usual WKB method, we call \eqref{a35} the eikonal equation and \eqref{a36} the transport equations.

\subsection{Resolution of  the eikonal equation \eqref{a35}}
\label{sub.eikonal}

Consider the complexified symbol
\begin{equation}
q ( x , \xi ) = - p^{0} ( x , i \xi ) = \xi \cdot A^{0} ( x ) \xi + b^{0} ( x ) \cdot \xi - c^{0} ( x ) \in \R , 
\end{equation}
and let,
 for some saddle point $\s \in \uuu^{( 1 )}$,
 $\Lambda_{+}$ (resp. $\Lambda_{-}$) be the stable outgoing (resp. incoming) manifold of the $H_{q}$ flow passing through $\rho_{\s} = ( \s , 0 )\in \ccc$. It is proved in \cite[Lemma 8.1]{HeHiSj08-2} that $\Lambda_{\pm}$ are
 Lagrangian manifolds which project nicely on the $x$-space. 
  Hence, there exist smooth functions $\phi_{\pm}$ defined in a neighborhood of $\rho_{\s}$ such that $\phi_{\pm} ( \s ) = 0$ and, near $\rho_{\s}$,
 \begin{equation*}
\Lambda_{\pm} = \{ ( x , \xi ) \in T^{*} \R^{d} ; \ \xi = \nabla \phi_{\pm} ( x ) \} .
\end{equation*}
Moreover, one has $\Lambda_{\pm} \subset \{ ( x , \xi ) \in T^{*} \R^{d} ; \ q ( x , \xi ) = q(\rho_{\s})=0 \}$ and, according to  \cite[Proposition 8.2]{HeHiSj08-2},
$\pm\Hess \phi_{\pm}(\s)>0$.
Summing up, the following properties of the functions $\phi_{\pm}$ will be used in the sequel:
\begin{equation} \label{g14}
\phi_{\pm}(\s)=0, \qquad \nabla \phi_{\pm}(\s)=0 \qquad \text{and} \qquad  \pm\Hess \phi_{\pm}(\s)>0.
\end{equation}

\begin{lemma}\sl \label{a38}
There exists a neighborhood $V$ of $\s$ and a smooth function $\ell_{\s , 0} \in C^{\infty} ( V )$ such that 
\begin{equation*}
\forall x \in V , \qquad \phi_{+} ( x ) = f ( x ) - f ( \s ) + \frac{\ell_{\s , 0}^{2} ( x )}{2}.
\end{equation*}
Moreover, we have $\nabla \ell_{\s,0}(\s)\neq 0$.
\end{lemma}

\begin{proof}
This lemma comes from an observation of \cite[(11.20)]{HeHiSj08-2} and we follow this approach. Let $\Lambda_{f} = \{ ( x , \nabla f ( x ) ) \} \subset T^{*} \R^{d}$ denote the Lagrangian manifold associated with the phase function $f$. The eikonal equations \eqref{a19} and \eqref{a20} imply $q ( x , \nabla f ( x ) ) = 0$ and then $\Lambda_{f}$ is stable by the $H_{q}$ flow. In particular, its tangent space $T_{\rho_{\s}} \Lambda_{f}$ at $\rho_{\s}$ is stable by $F_{q}$, the linearization of $H_{q}$ at $\rho_{\s}$. Moreover, a direct computation and \eqref{g13} show that
\begin{equation*}
F_{q} = \left( \begin{array}{cc}
B & 2 A^{0} \\
2 H A^{0} H & - B^{t}
\end{array} \right) = - i \left( \begin{array}{cc}
1 & 0 \\
0 & -i
\end{array} \right) F_{p^{0}} \left( \begin{array}{cc}
1 & 0 \\
0 & -i
\end{array} \right)^{*} .
\end{equation*}
 In particular, the discussion below \eqref{g13} implies that $F_{q}$ has no eigenvalue on the imaginary axis. Let $k_{\pm}$ be the number of eigenvalues of $F_{q}$ restricted to $T_{\rho_{\s}} \Lambda_{f}$ with positive/negative real part. Then, we have $k_{+} + k_{-} = d$.

Let $K_{\pm}$ be the stable outgoing/incoming submanifold of $\Lambda_{f}$ given by the Hamiltonian vector field $H_{q}$ restricted to $\Lambda_{f}$. Then $K_{\pm}$ has dimension $k_{\pm}$ and $K_{\pm}$ projects nicely on the $x$-space
\begin{equation} \label{a2}
K_{\pm} = \Lambda_{\pm} \cap \Lambda_{f} \qquad \text{and} \qquad T_{\rho_{s}} K_{\pm} = T_{\rho_{s}} \Lambda_{\pm} \cap T_{\rho_{s}} \Lambda_{f} .
\end{equation}
Using that $\nabla \phi_{\pm} = \nabla f$ on $\pi_{x} ( K_{\pm} )$, we get
\begin{equation} \label{a1}
\forall x \in \pi_{x} ( K_{\pm} ) , \qquad \phi_{\pm} ( x ) = f ( x ) - f ( \s ) .
\end{equation}
Since $\s$ is a saddle point of $f$, its Hessian has signature $( d - 1 , 1)$. Thus, \eqref{g14} and \eqref{a1} imply that $k_{+} = d - 1$ and $k_{-} = 1$. Eventually, using again \eqref{a1}, we have on $T_{\s} \pi_{x} ( K_{-} )$
\begin{equation*}
\Hess ( \phi_{+} - f ) = \Hess \phi_{+} - \Hess \phi_{-} > 0 .
\end{equation*}
Thus, in a neighborhood of $\s$, $g : = \phi_{+} - f + f ( \s )$ is a nonnegative function which vanishes 
at order~$2$ on $\pi_{x} ( K_{+} )$.

Let us now construct a square root of $g$. After a local change of coordinates $x \rightarrow ( y , z ) \in \R^{d - 1} \times \R$ mapping $\s$ to $0$, we can assume that $\pi_{x} ( K_{+} ) = \{ ( y , z ) ; \ z = 0 \}$. Near $0$, $g ( y , 0 ) = 0$ from \eqref{a1}, $\partial_{z} g ( y , 0 ) = 0$ from \eqref{a2}, and $\partial^{2}_{z , z} g ( y , 0 ) > 0$ from the last sentence of the previous paragraph. Then, the Taylor formula gives
\begin{equation*}
g ( y , z ) = z^{2}\int_{0}^{1} ( 1 - t ) \partial^{2}_{z , z} g ( y , t z ) \, d t  ,
\end{equation*}
which leads to
\begin{equation*}
\label{eq.ell}
\phi_{+} = f - f ( \s ) + \frac{\ell_{\s , 0}^{2}}{2} \qquad \text{with} \qquad \ell_{\s , 0} ( y , z ) = z \Big( 2 \int_{0}^{1} ( 1 - t ) \partial^{2}_{z , z} g ( y , t z ) \, d t \Big)^{1 / 2}.
\end{equation*}
Since the quantity under the square root is positive when evaluated in $z = 0$, then $\ell_{\s , 0}$ is a smooth function in a vicinity of $\s$ and $\nabla \ell_{\s,0}(\s)\neq0$.
\end{proof}

\begin{lemma}\sl \label{a42}
Let $\s \in \uuu^{( 1 )}$. The function $\ell_{\s , 0}$ defined by Lemma \ref{a38} solves \eqref{a35} in a neighborhood of $\s$.
Moreover, the vector $\eta ( \s ) : = \nabla \ell_{\s , 0} ( \s )$ is an eigenvector of the matrix $\Lambda ( \s ) = 2 H ( \s ) A^{0} ( \s ) + B^{t} ( \s )$ associated with its negative eigenvalue $\mu ( \s )$. Finally,
\begin{equation} \label{a43}
\mu ( \s ) = - A^{0} ( \s ) \eta ( \s ) \cdot \eta ( \s ) ,
\end{equation}
and
\begin{equation} \label{a44}
\det \Hess \Big( f + \frac{1}{2} \ell_{\s , 0}^{2} \Big) ( \s ) = - \det H( \s ) .
\end{equation}
In particular, $A^{0} ( \s ) \eta ( \s ) \cdot \eta ( \s )>0$.
\end{lemma}

\begin{proof}
Let us drop the index $\s$ and write $\ell_{0}$ instead of $\ell_{\s , 0}$. By definition, $\phi_{+}$ is solution of the eikonal equation
\begin{equation} \label{a45}
 A^{0} \nabla \phi_{+} \cdot \nabla \phi_{+} + b^{0} \cdot \nabla \phi_{+} - c^{0}= 0 .
\end{equation}
Since $\phi_{+} = f -f(\s)+ \frac{1}{2} \ell_{0}^{2}$, this implies
\begin{equation*}
 A^{0} \nabla f \cdot \nabla f + 2 \ell_{0} A^{0} \nabla f \cdot \nabla \ell_{0} + \ell_{0}^{2} A^{0} \nabla \ell_{0} \cdot \nabla \ell_{0} + b^{0} \cdot \nabla f + \ell_{0} b^{0} \cdot \nabla \ell_{0} - c^{0} = 0 .
\end{equation*}
Since $f$ is also solution of \eqref{a45} by \eqref{a19} and \eqref{a20}, it follows that
\begin{equation*}
 2 \ell_{0} A^{0} \nabla f \cdot \nabla \ell_{0} + \ell_{0}^{2} A^{0} \nabla \ell_{0} \cdot \nabla \ell_{0} + \ell_{0} b^{0} \cdot \nabla \ell_{0} = 0 ,
\end{equation*}
which gives \eqref{a35} by dividing by $\ell_{0}$ (this is allowed since 
$\nabla \ell_{0}(\s)\neq0$ implies
$\ell_{0}\neq 0 $ a.e. around~$\s$). Moreover, a Taylor expansion of \eqref{a35} at $\s$ gives,
for every $x$ around $\s$,
\begin{equation*}
2 A^{0} ( \s ) H ( \s ) (x-\s) \cdot \eta ( \s ) + ( A^{0} ( \s ) \eta ( \s ) \cdot \eta ( \s ) ) ( \eta ( \s ) \cdot (x-\s) ) + d b^{0} ( \s ) (x-\s) \cdot \eta ( \s ) = 0 ,
\end{equation*}
or equivalently
\begin{equation*}
(x-\s) \cdot \Big( \big( 2 H ( \s ) A^{0} ( \s ) + B^{t} ( \s ) + A^{0} ( \s ) \eta ( \s ) \cdot \eta ( \s ) \big) \eta ( \s ) \Big) = 0 .
\end{equation*}
It follows that
\begin{equation}
\big( 2 H ( \s ) A^{0} ( \s ) + B^{t} ( \s ) + A^{0} ( \s ) \eta ( \s ) \cdot \eta ( \s ) \big) \eta ( \s ) = 0 ,
\end{equation}
which shows that $\eta ( \s )$ is an eigenvector of $\Lambda ( \s ) = 2 H ( \s ) A^{0} ( \s ) + B^{t} ( \s )$ associated with 
the eigenvalue $- A^{0} ( \s ) \eta ( \s ) \cdot \eta ( \s )$. Since $A^{0} ( \s )$ is positive semidefinite, this eigenvalue is nonpositive. Then, Lemma~\ref{a29} implies that $ - A^{0} ( \s ) \eta ( \s ) \cdot \eta ( \s ) = \mu ( \s ) <0$.

It remains to prove \eqref{a44}.
We remove the dependence in $\s$ in the following when it is unambiguous.
 By definition of $\ell_{0}$, one has
\begin{equation*}
\Hess \Big( f + \frac{1}{2} \ell_{0}^{2} \Big) ( \s ) = H + \Pi_{\eta} ,
\end{equation*}
where $\Pi_{\eta} x := \< x , \eta \> \eta$. Hence, \eqref{a44} is equivalent to 
\begin{equation} \label{a46}
\det ( \Id + H^{- 1} \Pi_{\eta} ) = - 1.
\end{equation}
We first observe that $\eta^{\bot}$ is stable by $E : = \Id + H^{- 1} \Pi_{\eta}$ and that $E_{\vert \eta^\bot} = \Id$. On the other hand, one has
\begin{equation} \label{a47}
\< E \eta , \eta \> = \Vert \eta  \Vert^{2} \big( 1 + \< H^{- 1}  \eta  , \eta  \> \big) .
\end{equation}
But, $H  ( 2 A^{0}  + J ) \eta  = \Lambda \eta  = \mu  \eta $
with $J=H^{-1} B^{t}$ gives $\< ( 2 A^{0}  + J  ) \eta  , \eta  \> = \mu  \< H^{- 1} \eta  , \eta  \>$.
Since $J$ is antisymmetric 
by Lemma~\ref{a40} and thanks to \eqref{a43},
this implies
\begin{equation*}
\< H^{- 1}  \eta  , \eta  \> = \frac{2}{\mu} \< A^{0} \eta , \eta \> = - 2 .
\end{equation*}
Plugging this identity into \eqref{a47}, we get $\< E \eta  , \eta  \> = - \Vert \eta  \Vert^{2}$. Choosing a basis $( e_{2} , \ldots , e_{d} )$ of $\eta ^{\bot}$ and computing the matrix of $E$ in the basis $( \eta  , e_{2} , \ldots , e_{d})$, the above discussion implies \eqref{a46}.
\end{proof}

\subsection{Resolution of the transport equations \eqref{a36}} \label{b36}

\begin{lemma}\sl
There exists an open neighbourhood $V$ of $\s$ and some smooth functions $\ell_{\s , j} \in C^{\infty} ( V )$ such that, for all $j \geq 1$, $\ell_{\s , j}$ solves \eqref{a36}.
\end{lemma}

\begin{proof}
Since, for all $j\geq 1$, $R_{j}$ only depends on the $\ell_{\s , k}$ with $0\leq k\leq j-1$, we can solve the equations
\eqref{a36} by induction. It thus suffices to show that there exists an open neighbourhood $V$ of $\s$ such that, for any smooth function $f$, there exists $u \in C^{\infty} ( V )$ satisfying
\begin{equation} \label{a48}
\lll u = f ,
\end{equation}
where $\lll$ is the transport operator defined by
\begin{equation} \label{a49}
\lll u = 2 A^{0} \nabla f \cdot \nabla u + ( A^{0} \nabla \ell_{\s , 0} \cdot \nabla \ell_{\s , 0} ) u + 2 \ell_{\s , 0} ( A^{0} \nabla \ell_{\s , 0} \cdot \nabla u ) + b^{0} \cdot \nabla u .
\end{equation}
Assume for simplicity that $\s = 0$.
We first look for a formal solution in powers of $x$. Given~$m \in \N$, we denote by $\ppp_{hom}^{m}$ the set of homogeneous polynomials of degree $m$ and we write $f ( x ) \simeq \sum_{m \in \N} f_{m}$ with $f_{m} \in \ppp_{hom}^{m}$. We recall that
$\nabla \ell_{\s , 0} ( \s ) = \eta ( \s )$ is an eigenvector of $\Lambda ( \s ) = 2 H ( \s ) A^{0} ( \s ) + B^{t} ( \s )$
associated with its sole negative eigenvalue $\mu(\s)$ (see  Lemma~\ref{a29}). Then, 
$\lll$ decomposes into
\begin{equation*}
\lll = \lll_{0} + \lll_{>} ,
\end{equation*}
where $\lll_{>} ( p ) = \ooo ( x^{m + 1} )$ for all $p \in \ppp_{hom}^{m}$ and $\lll_{0} : \ppp_{hom}^{m} \rightarrow \ppp_{hom}^{m}$ is given by
\begin{equation*}
\lll_{0} = 2 A^{0} ( \s ) H ( \s ) x \cdot \nabla + A^{0} ( \s ) \eta ( \s ) \cdot \eta ( \s ) + 2 \< \eta ( \s ) , x \> A^{0} ( \s ) \eta ( \s ) \cdot \nabla + B ( \s ) x \cdot \nabla ,
\end{equation*}
that we can rewrite, since $\mu(\s)=-A^{0} ( \s ) \eta ( \s ) \cdot \eta ( \s )$ by Lemma~\ref{a42},
as
\begin{equation} \label{a50}
\lll_{0} = \big( 2 A^{0} ( \s ) H ( \s ) + B ( \s ) + 2 A^{0} ( \s ) \Pi_{\eta} \big) x \cdot \nabla - \mu ( \s ) ,
\end{equation}
where $\Pi_{\eta} y := \< y , \eta ( \s ) \> \eta ( \s )$. We shall prove that $\lll_{0}$ is invertible on $\ppp_{hom}^{m}$ for any $m \geq 0$. Let us denote $\Upsilon = 2 A^{0} ( \s ) H ( \s ) + B ( \s ) + 2 A^{0} ( \s ) \Pi_{\eta} = \Lambda^t ( \s ) + 2 A^{0} ( \s ) \Pi_{\eta}$. One has
\begin{equation*}
\Upsilon^{t} \eta(\s) = \Lambda ( \s ) \eta(\s) + 2 \Pi_{\eta} A^{0} ( \s ) \eta(\s) = \mu ( \s ) \eta(\s) + 2 \< A^{0} ( \s ) \eta(\s) , \eta(\s) \> \eta(\s) = - \mu ( \s ) \eta(\s) .
\end{equation*}
Choosing some vectors $e_{2} , \ldots , e_{d}$ such that $\bbb = ( \eta(\s) , e_{2} , \ldots , e_{d})$ is a basis of $\C^{d}$ and the matrix $M$ of $\Lambda ( \s )$ in $\bbb$ is upper triangular, it follows that the matrix $M^{\prime}$ of $\Upsilon^{t}$ in the basis $\bbb$ is also upper triangular, with the same diagonal entries as $M$, except that the first diagonal entry
$\mu ( \s )$ is replaced by $- \mu ( \s )$ (actually, only the first line of $M$ and $M^{\prime}$ differ).
Since $\mu(\s)$ is the only eigenvalue of $\Lambda ( \s )$ with nonpositive real part, the spectrum of $\Upsilon^{t}$ is contained in $\{ \Re z > 0\}$. Thanks to Lemma \ref{a30} in the appendix, this implies that the spectrum of $\Upsilon x \cdot \nabla : \ppp_{hom}^{m} \rightarrow \ppp_{hom}^{m}$ is contained in $\{ \Re z > 0 \}$ for every $m>0$ and hence $\lll_{0}=\Upsilon x \cdot \nabla-\mu(\s)$ is invertible on $\ppp_{hom}^{m}$ for every $m\geq 0$ (note that $\lll_{0}=-\mu(\s)$ on $\ppp_{hom}^{0}$). Using this property, we can solve the transport equation $\lll u = f$ following the method of \cite[Chapter 3]{DiSj99_01}. We recall briefly the successive steps. We first find a formal solution $\widetilde{u}$ to the equation 
\begin{equation} \label{a51}
\lll \widetilde{u} = \widetilde{f} ,
\end{equation}
where $\widetilde{f}$ denotes the formal power expansion of $f$: $\widetilde{f} \simeq \sum_{k} \widetilde{f}_{k}$ with $\widetilde{f}_{k} \in \ppp^{k}_{hom}$. We look for $\widetilde{u}$ under the form $\widetilde{u} \simeq \sum_{k} \widetilde{u}_{k}$ with $\widetilde{u}_{k} \in \ppp_{hom}^{k}$. Since $\lll_{0}$ is invertible, there exists $\widetilde{u}_{0} \in \ppp_{hom}^{0}$ solving $\lll_{0} \widetilde{u}_{0} = \widetilde{f}_{0}$. Then we choose $\widetilde{u}_{1} \in \ppp_{hom}^{1}$ solution of 
\begin{equation*}
\lll_{0} \widetilde{u}_{1} = \widetilde{f}_{1} + r_{1} ,
\end{equation*}
where $r_{1}$ denotes the homogeneous part of degree $1$ of $- \lll_{>}( \widetilde{u}_{0} )$. Iterating this procedure, we obtain a formal solution to \eqref{a51}. Using this formal solution and a Borel procedure, we construct a smooth function $\overline{u}$ such that $\overline{u}$ and $\widetilde{u}$ have the same Taylor expansion at the origin. As a consequence $\lll \overline{u} = f + \ooo ( x^{\infty} )$. The last step consists in showing that for every $g=\ooo ( x^{\infty} )$,
there exists a solution $v=\ooo ( x^{\infty} )$ to $\lll v = g $. This can be done by using the 
characteristic method and the fact that the spectrum of $\Upsilon$ is contained in $\{ \Re z > 0\}$; 
 we refer to \cite[proof of Proposition 3.5]{DiSj99_01} for details.
Then, taking $\overline{u}_{\infty}$ satisfying $\lll \overline{u}_{\infty} = f-\lll \overline{u} = \ooo ( x^{\infty} )$,
$u:=\overline{u}+\overline{u}_{\infty}$ is a true solution of \eqref{a48}
and \cite[Proposition 3.5]{DiSj99_01} shows that the neighborhood~$V$
of $\s$ where $u$ is defined can be chosen independently of $f$. 
\end{proof}

Using the preceding result and a Borel procedure, we get the following.

\begin{proposition}\sl \label{a52}
For any $\s \in \uuu^{( 1 )}$, there exists a smooth function $x \mapsto \ell_{\s} ( x , h )$ defined in a neighborhood $V_{\s}$ of $\s$ such that the following hold true

$i)$ $\ell_{\s}$ admits a classical expansion $\ell_{\s} \sim \sum_{k} h^{k} \ell_{\s , k}$,

$ii)$ $( 2 A \nabla \ell_{\s} \cdot \nabla f + ( A \nabla \ell_{\s} \cdot \nabla \ell_{\s} ) \ell_{\s} + b \cdot \nabla \ell_{\s} ) - h \div ( A \nabla \ell_{\s} ) = \ooo ( h^{\infty} )$, uniformly with respect to~$x$ in $V_{\s}$,

$iii)$ $\ell_{\s , 0} ( x ) = ( x - \s ) \cdot \eta ( \s ) + \ooo ( \vert x - \s \vert^{2} )$.
\end{proposition}

Note that the function $x\mapsto - \ell_{\s} ( x , h )$ also satisfies  Proposition \ref{a52}. More precisely, $i)$ and~$ii)$ hold true without modification while  $\eta ( \s )$  has to be replaced by $- \eta ( \s ) = - \nabla \ell_{\s , 0} ( \s )$ in~$iii)$. At this point, we do not specify which function ($\ell_{\s}$ or $- \ell_{\s}$) will be used later.

\section{Global construction of quasimodes}
\label{s21}

In this section, which is an adaptation of \cite[Section 4A]{LePMi20}, we assume the hypotheses of Theorem \ref{a15}, \eqref{a17},
\eqref{h1}, and \eqref{a24}. We send back the reader to the notations following Definition \ref{a23} and introduce some additional topological objects. Given $\m \in \uuu^{( 0 )}\setminus\{\underline\m\}$, one has $\bsigma ( \m ) = \sigma_{i}$ for some $i \geq 2$. Moreover, since $\sigma_{i - 1} > \sigma_{i}$, there exists a unique connected component of  $\{ f < \sigma_{i - 1} \}$ that contains $\m$ (observe that this component is not necessarily critical). We denote that component by $E_{-} ( \m )$, and by
\begin{equation} \label{b30}
E_{-} : \uuu^{( 0 )}\setminus\{\underline\m\} \longrightarrow {\mathcal P}( \R^{d} )
\end{equation}
the corresponding application. It follows from  \cite[Remark 2.2]{Mi19} that, for any $\m \in \uuu^{( 0 )}\setminus\{\underline\m\}$, there exists a unique $\widehat{\m} \in E_{-} ( \m ) \cap \uuu^{( 0 )}$
such that $\bsigma ( \widehat{\m} ) > \bsigma ( \m )$.  In particular,
$\m\in  E_{-} ( \m ) \subset E(\widehat{\m})$ and thus,
\begin{equation} \label{b31}
\forall \m \in \uuu^{( 0 )}\setminus\{\underline\m\} ,
\qquad f ( \widehat{\m} ) \leq f ( \m ) .
\end{equation}
We denote  by $\widehat{E} ( \m )$ the connected component of $\{ f < \sigma ( \m ) \}$ containing 
$\widehat{\m}$. It holds additionally $\widehat{E} ( \m ) \subset E_{-} ( \m )$ and $\widehat{E} ( \m )$ is a critical component (see Definition \ref{a23}). We denote by $\widehat{E} : \uuu^{( 0 )}\setminus\{\underline\m\} \rightarrow \Cr$ and $\widehat{\m} : \uuu^{( 0 )}\setminus\{\underline\m\} \rightarrow \uuu^{( 0 )}$ the corresponding applications.

Let us consider some arbitrary $\m \in \uuu^{( 0 )} \setminus \{ \underline{\m} \}$. For every $\s \in {\bf j} ( \m )$, one has $f ( \s ) = \bsigma ( \m )$. For any  $\tau , \delta > 0$, we define the sets $\bbb_{\s , \tau , \delta}$ and $\ccc_{\s , \tau , \delta}$ by
\begin{equation*}
\bbb_{\s , \tau , \delta} : = \{ f \leq \bsigma ( \m ) + \delta \} \cap \big\{ x \in \R^{d} ; \ \vert \eta ( \s ) \cdot ( x - \s ) \vert \leq \tau \big\} ,
\end{equation*}
and
\begin{equation} \label{a56}
\ccc_{\s , \tau , \delta} \text{ is the connected component of $\bbb_{\s , \tau , \delta}$ containing } \s ,
\end{equation}
where $\eta ( \s )$ has been defined in Lemma \ref{a42}. We recall that $\eta ( \s )$ is an eigenvector of the matrix $\Lambda ( \s ) = 2 H ( \s ) A^{0} ( \s ) + B^{t} ( \s )$ associated with its only negative eigenvalue $\mu ( \s )$, which has multiplicity one (see Lemma \ref{a29}). Moreover, one has from  \eqref{a43}  the normalization condition $A^{0} ( \s ) \eta ( \s ) \cdot \eta ( \s ) = - \mu ( \s )$. Observe that this normalization condition is not the same as in \cite{LePMi20}, where it is imposed $\Vert \eta(\s) \Vert = 1$. Let us also define
\begin{equation} \label{a57}
E_{\m , \tau , \delta} := \big( E_{-} ( \m ) \cap \{ f < \bsigma ( \m ) + \delta \} \big) \setminus \bigcup_{\s \in {\bf j} ( \m )} \ccc_{\s , \tau , \delta} ,
\end{equation}
where $E_-(\m)$ is defined by \eqref{b30}.

According to the geometry of the Morse function $f$ around $\partial E ( \m )$ and to the lemmas  of Section \ref{sub.eikonal}, we have the following result.

\begin{lemma}\sl \label{a59}
for any $\m \in \uuu^{( 0 )} \setminus \{ \underline{\m} \}$ and $\s \in {\bf j} ( \m )$, there exists a neighborhood $V$ of $\s$ such that
\begin{equation*}
\forall x \in V \setminus \{ \s \} , \qquad x - \s \in \eta ( \s )^{\perp} \ \Longrightarrow \ f ( x ) > f ( \s ) .
\end{equation*}
It follows that, for $\tau_{0},\delta_{0} > 0$ sufficiently small and every $\tau\in ] 0,\tau_{0}]$,
$\delta\in ] 0,\delta_{0}]$,
there exists a connected component of $E_{\m , 3 \tau , 3 \delta}$ containing $\uuu^{(0)}\cap E(\m)$ and disjoint from $\uuu^{(0)}\cap (E_{-}(\m)\setminus E(\m))$.
We will denote by $E_{\m , 3 \tau , 3 \delta}^{+}$ this component and
by $E^{-}_{\m , 3 \tau , 3 \delta}$
its complement in $E_{\m , 3 \tau , 3 \delta}$.
\end{lemma}

\begin{remark}\sl
The  above set $E^{-}_{\m , 3 \tau , 3 \delta}$ contains $\widehat\m$
but is not connected in general. However,
when $\m$ satisfies
${\bf j} ( \m ) \cap {\bf j} ( \m^{\prime} ) = \emptyset$ for every $\m^{\prime}\in\uuu^{(0)}\setminus\{\m\}$,
one has, owing to \cite[Remark 1.7 and Section 4A]{LePMi20},  ${\bf j} ( \m ) = \partial \widehat{E} ( \m ) \cap \partial E ( \m )$.
In such a case, the set $E^{-}_{\m , 3 \tau , 3 \delta}$ is connected. 
\end{remark}

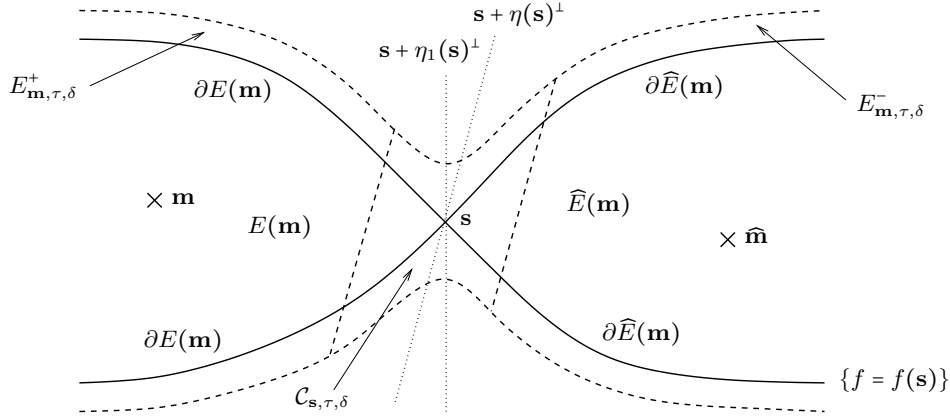
\begin{figure}
\begin{center}
\begin{picture}(0,0)%
\includegraphics{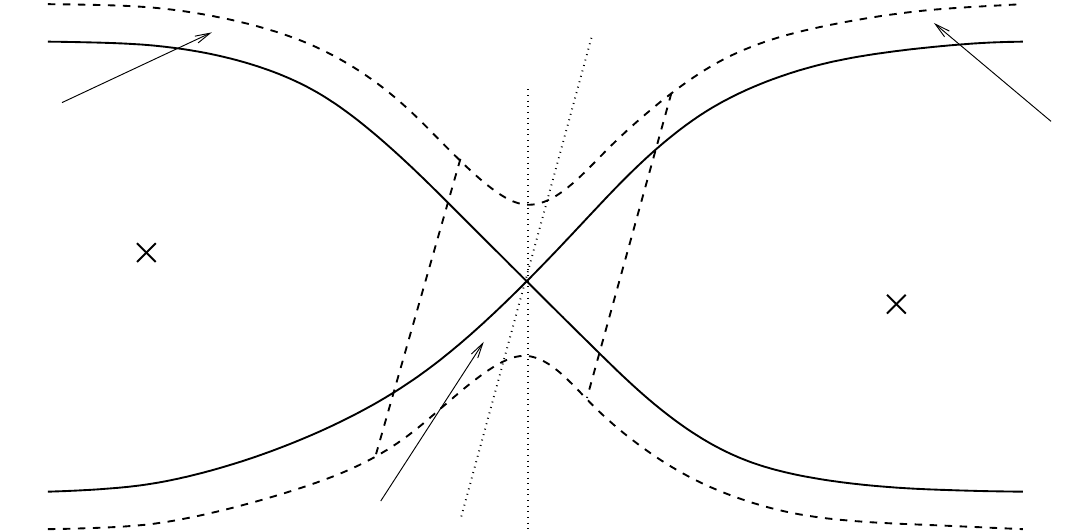}%
\end{picture}%
\setlength{\unitlength}{1184sp}%
\begingroup\makeatletter\ifx\SetFigFont\undefined%
\gdef\SetFigFont#1#2#3#4#5{%
  \reset@font\fontsize{#1}{#2pt}%
  \fontfamily{#3}\fontseries{#4}\fontshape{#5}%
  \selectfont}%
\fi\endgroup%
\begin{picture}(17055,8466)(-2564,-9394)
\put(12376,-5836){\makebox(0,0)[b]{\smash{{\SetFigFont{9}{10.8}{\rmdefault}{\mddefault}{\updefault}$\widehat{\m}$}}}}
\put(6301,-5461){\makebox(0,0)[b]{\smash{{\SetFigFont{9}{10.8}{\rmdefault}{\mddefault}{\updefault}$\s$}}}}
\put(5551,-1936){\makebox(0,0)[b]{\smash{{\SetFigFont{9}{10.8}{\rmdefault}{\mddefault}{\updefault}$\s + \eta_{1} ( \s )^{\perp}$}}}}
\put(2401,-5611){\makebox(0,0)[b]{\smash{{\SetFigFont{9}{10.8}{\rmdefault}{\mddefault}{\updefault}$E ( \m )$}}}}
\put(9151,-7861){\makebox(0,0)[lb]{\smash{{\SetFigFont{9}{10.8}{\rmdefault}{\mddefault}{\updefault}$\partial \widehat{E} ( \m )$}}}}
\put(10051,-2686){\makebox(0,0)[lb]{\smash{{\SetFigFont{9}{10.8}{\rmdefault}{\mddefault}{\updefault}$\partial \widehat{E} ( \m )$}}}}
\put(9076,-5161){\makebox(0,0)[b]{\smash{{\SetFigFont{9}{10.8}{\rmdefault}{\mddefault}{\updefault}$\widehat{E} ( \m )$}}}}
\put(-2549,-2686){\makebox(0,0)[b]{\smash{{\SetFigFont{9}{10.8}{\rmdefault}{\mddefault}{\updefault}$E_{\m , \tau , \delta}^{+}$}}}}
\put(601,-2761){\makebox(0,0)[lb]{\smash{{\SetFigFont{9}{10.8}{\rmdefault}{\mddefault}{\updefault}$\partial E ( \m )$}}}}
\put(-449,-8011){\makebox(0,0)[lb]{\smash{{\SetFigFont{9}{10.8}{\rmdefault}{\mddefault}{\updefault}$\partial E ( \m )$}}}}
\put(14101,-8836){\makebox(0,0)[lb]{\smash{{\SetFigFont{9}{10.8}{\rmdefault}{\mddefault}{\updefault}$\{ f = f ( \s ) \}$}}}}
\put(14476,-3061){\makebox(0,0)[lb]{\smash{{\SetFigFont{9}{10.8}{\rmdefault}{\mddefault}{\updefault}$E_{\m , \tau , \delta}^{-}$}}}}
\put(3301,-9211){\makebox(0,0)[b]{\smash{{\SetFigFont{9}{10.8}{\rmdefault}{\mddefault}{\updefault}$\ccc_{\s , \tau , \delta}$}}}}
\put(376,-5011){\makebox(0,0)[b]{\smash{{\SetFigFont{9}{10.8}{\rmdefault}{\mddefault}{\updefault}$\m$}}}}
\put(7426,-1186){\makebox(0,0)[b]{\smash{{\SetFigFont{9}{10.8}{\rmdefault}{\mddefault}{\updefault}$\s + \eta ( \s )^{\perp}$}}}}
\end{picture}%
\end{center}
\caption{Representation of the Morse function $f$ near a point $\s \in {\bf j} ( \m )\cap\partial \widehat E(\m)$
when the latter set is nonempty. Here, $\eta_{1} ( \s )$ denotes an eigenvector of $\Hess f ( \s )$ associated with its negative eigenvalue.}
\label{f1}
\end{figure}

\begin{proof}[Proof of Lemma \ref{a59}]
Without loss of generality, we can assume that $\s = 0$ and $f ( \s ) = 0$. Thanks to Lemma \ref{a38} and Proposition \ref{a52}, one has
\begin{equation*}
\phi_{+} ( x ) = f ( x ) + \frac{\ell_{\s , 0} ( x )^{2}}{2} = f ( x ) + \frac{\< x , \eta(\s)\>^{2}}{2} + \ooo ( x^{3} ) ,
\end{equation*}
near $\s=0$. This implies that, for all $x \in \eta(\s)^{\bot}$, we have
\begin{equation*}
f ( x ) = \phi_{+} ( x ) + \ooo ( x^{3} ) .
\end{equation*}
Since $\Hess \phi_{+} ( \s )$ is positive definite by \eqref{g14}, the conclusion follows.
\end{proof}

Let us now define, for $h>0$ and $\tau_{0} , \delta_{0} > 0$ small enough, the function $v_{\m , h}$ on the sublevel set $E_{-} ( \m ) \cap \{ f < \bsigma ( \m ) + 3 \delta_{0} \}$ (see \eqref{b30}) as follows. On the disjoint open sets $E^{+}_{\m , 3 \tau_{0} , 3 \delta_{0}}$ and $E^{-}_{\m , 3 \tau_{0} , 3 \delta_{0}}$ introduced in Lemma \ref{a59}, we set
\begin{equation} \label{a60}
v_{\m , h} ( x ) : =
\left\{\begin{aligned}
&+ 1 &&\text{ for } x \in E^{+}_{\m , 3 \tau_{0} , 3 \delta_{0}} ,        \\
&- 1 &&\text{ for } x \in E^{-}_{\m , 3 \tau_{0} , 3 \delta_{0}} .
\end{aligned} \right.
\end{equation}
In addition, for every $\s \in {\bf j} ( \m )$ and $x \in \ccc_{\s , 3 \tau_{0} , 3 \delta_{0}}$ (see \eqref{a56}), we set
\begin{equation} \label{a61} 
v_{\m , h} ( x ) : = C^{- 1}_{\s , h} \int_{0}^{\ell_{\s} ( x , h )} \zeta ( r / \tau_{0} ) e^{- \frac{r^{2}}{2 h}} \, d r ,    
\end{equation}
where the function $\ell_{\s}$ is given by Proposition \ref{a52} and its sign (see the discussion below Proposition \ref{a52}) is chosen so that there exists a neighborhood $V$ of $\s$ such that $E ( \m ) \cap V$ is included in the half-plane $\{ \eta ( \s ) \cdot ( x - \s ) > 0 \}$ (see Lemma \ref{a59} and Figures \ref{f1} and \ref{f2}), $\zeta \in C^{\infty} ( \R ; [ 0 , 1 ] )$ is even and satisfies $\zeta = 1$ on $[ - 1 , 1 ]$, $\zeta ( r ) = 0$ for $\vert r \vert \geq 2$, and
\begin{equation*}
C_{\s , h} : = \frac{1}{2} \int_{- \infty}^{+ \infty} \zeta ( r / \tau_{0} ) e^{- \frac{r^{2}}{2 h}} \, d r .
\end{equation*}
In particular, we have
\begin{equation} \label{g19}
\exists \gamma > 0 , \qquad C^{- 1}_{\s , h} = \sqrt{\frac{2}{\pi h}} \big( 1 + \ooo ( e^{- \frac{\gamma}{h}} ) \big) .
\end{equation}
Note also that, for every $\tau_{0}>0$ and then $\delta_{0} > 0$ small enough, thanks to Proposition \ref{a52} and to the definitions \eqref{a60} and \eqref{a61}, and since the sets $E^{+}_{\m , 3 \tau_{0} , 3 \delta_{0} }$, $E^{-}_{\m , 3 \tau_{0} , 3 \delta_{0}}$ and $\ccc_{\s , 3 \tau_{0} , 3 \delta_{0}}$'s, $\s \in {\bf j} ( \m )$, are mutually disjoint (see Lemma \ref{a59}), $v_{\m , h}$ is well defined and is $C^{\infty}$ on $E_{-} ( \m ) \cap \{ f < \bsigma ( \m ) + 3 \delta_{0} \}$
for $h>0$ small enough.

Consider now a smooth function $\theta_{\m}$ such that
\begin{equation} \label{a62}
\theta_{\m} ( x ) : =
\left\{ \begin{aligned}
&1 &&\text{ for } x \in \Big\{ f \leq \bsigma ( \m ) + \frac{3}{2} \delta_{0} \Big\} \cap E_{-} ( \m ) ,    \\
&0 &&\text{ for } x \in \R^{d} \setminus \big( \{ f < \bsigma ( \m ) + \frac74 \delta_{0} \} \cap E_{-} ( \m ) \big) .    \\
\end{aligned} \right.
\end{equation}
The function $\theta_{\m} v_{\m , h}$ belongs to $C_{c}^{\infty} ( \R^{d} ; [ - 1 , 1 ] )$ and
\begin{equation*}
\supp ( \theta_{\m} v_{\m , h} ) \subset E_{-} ( \m ) \cap \{ f < \bsigma ( \m ) + 2 \delta_{0} \} .
\end{equation*}

\begin{figure}
 \begin{center}
 \hspace{4cm}
\scalebox{0.8}{
\begingroup%
  \makeatletter%
  \providecommand\color[2][]{%
    \errmessage{(Inkscape) Color is used for the text in Inkscape, but the package 'color.sty' is not loaded}%
    \renewcommand\color[2][]{}%
  }%
  \providecommand\transparent[1]{%
    \errmessage{(Inkscape) Transparency is used (non-zero) for the text in Inkscape, but the package 'transparent.sty' is not loaded}%
    \renewcommand\transparent[1]{}%
  }%
  \providecommand\rotatebox[2]{#2}%
  \ifx\svgwidth\undefined%
    \setlength{\unitlength}{382.20486069bp}%
    \ifx\svgscale\undefined%
      \relax%
    \else%
      \setlength{\unitlength}{\unitlength * \real{\svgscale}}%
    \fi%
  \else%
    \setlength{\unitlength}{\svgwidth}%
  \fi%
  \global\let\svgwidth\undefined%
  \global\let\svgscale\undefined%
  \makeatother%
  \begin{picture}(1,0.54872199)%
    \put(0,0){\includegraphics[width=\unitlength,page=1]{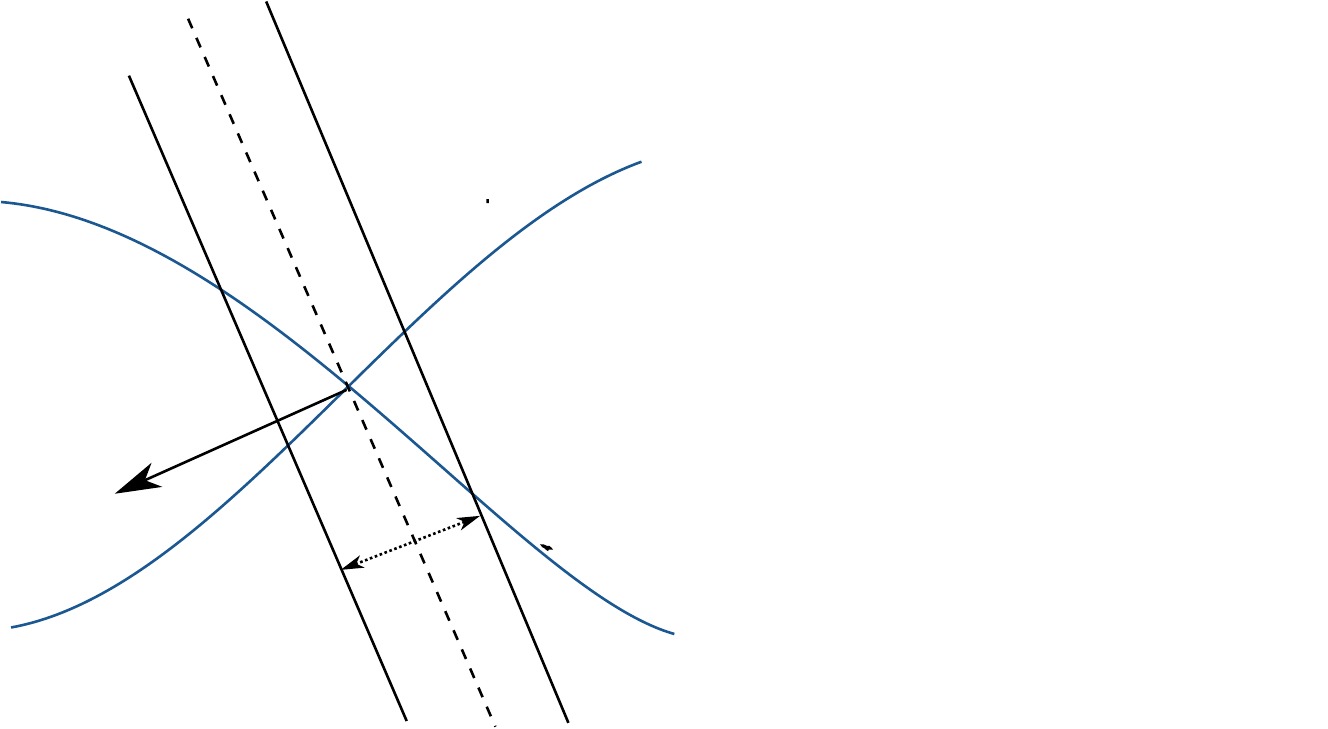}}%
    \put(0.29209216,0.09852605){\color[rgb]{0,0,0}\makebox(0,0)[lb]{\smash{$\ooo(\tau_0)$}}}%
    \put(0.06094284,0.20539241){\color[rgb]{0,0,0}\makebox(0,0)[lb]{\smash{$\eta(\s)$}}}%
    \put(0,0){\includegraphics[width=\unitlength,page=2]{cutoffloc2.pdf}}%
    \put(0.2630373,0.27629371){\color[rgb]{0,0,0}\makebox(0,0)[lb]{\smash{$\s$}}}%
    \put(0.38557519,0.26118089){\color[rgb]{0,0,0}\makebox(0,0)[lb]{\smash{$v_{\m,h}=-1$}}}%
    \put(0.0277826,0.26323132){\color[rgb]{0,0,0}\makebox(0,0)[lb]{\smash{$v_{\m,h}=1$}}}%
    \put(0.47485892,0.46956104){\color[rgb]{0,0,0}\makebox(0,0)[lb]{\smash{$\textcolor{mauve}{\supp(\theta_\m)}$}}}%
    \put(0.26803035,0.23393679){\color[rgb]{0,0,0}\makebox(0,0)[lb]{\smash{}}}%
    \put(0.47093252,0.38584055){\color[rgb]{0,0,0}\makebox(0,0)[lb]{\smash{\textcolor{bleu}{$\{V=\bsigma(\m)\}$}}}}%
    \put(0.5357438,0.18487925){\color[rgb]{0,0,0}\makebox(0,0)[lb]{\smash{}}}%
    \put(0.00425117,0.30155497){\color[rgb]{0,0,0}\makebox(0,0)[lb]{\smash{}}}%
    \put(0.24278387,0.35383609){\color[rgb]{0,0,0}\makebox(0,0)[lb]{\smash{}}}%
    \put(0,0){\includegraphics[width=\unitlength,page=3]{cutoffloc2.pdf}}%
    \put(-0.30623584,0.26077893){\color[rgb]{0,0,0}\makebox(0,0)[lb]{\smash{}}}%
    \put(-0.02336851,0.62347388){\color[rgb]{0,0,0}\makebox(0,0)[lt]{\begin{minipage}{1.02908175\unitlength}\raggedright \end{minipage}}}%
    \put(-0.02163312,0.61826773){\color[rgb]{0,0,0}\makebox(0,0)[lb]{\smash{}}}%
  \end{picture}%
\endgroup%
}
  \end{center}
  \caption{The support of the function $v_{\m,h}$}
  \label{f2}
  \end{figure}

\begin{defin}\sl \label{a63}
For $\tau_{0}>0$ and then $\delta_{0} , h > 0$ small enough, let us define the functions
\begin{equation} \label{a64}
\left\{ \begin{aligned}
\psi_{\underline{\m }, h} ( x ) &:=e^{- \frac{f ( x ) - f ( \underline{\m} )} {h}} ,  \\
\psi_{\m , h} ( x ) &: = \theta_{\m} ( x ) \big( v_{\m , h} ( x ) + 1 \big) e^{- \frac{f ( x ) - f ( \m )} {h}} \qquad \text{for } \m \in \uuu^{( 0 )}\setminus\{\underline{\m}\} .
\end{aligned} \right.
\end{equation}
We then define, for any $\m \in \uuu^{( 0 )}$, the quasimode $\varphi_{\m , h}$ by 
\begin{equation*}
\varphi_{\m , h} : = \frac{\psi_{\m , h}}{\Vert \psi_{\m , h} \Vert_{L^{2}}} .
\end{equation*}
\end{defin}

Note that, for every $\tau_{0},\delta_{0},h > 0 $ so that the above definition makes sense, $P \varphi_{\underline{\m} , h} = 0$ and, for every $\m \in \uuu^{( 0 )} \setminus \{ \underline{\m} \}$, the quasimodes $\psi_{\m , h}$ and $\varphi_{\m , h}$ belong to $C_{c}^{\infty} ( \R^{d} ; \R^{+} )$ with supports included in $E_{-} ( \m ) \cap \{ f < \bsigma ( \m ) + 2 \delta_{0} \}$. We have more precisely the following lemma resulting from the previous construction.

\begin{lemma}\sl \label{a65}
For every $\m \in \uuu^{( 0 )}$ and $\varepsilon > 0$, there exist $\tau_{0}>0$ and then $\delta_{0} > 0$ small enough such that,  for every $h>0$ small enough, one has

$i)$ The support of $\psi_{\m , h}$ satisfies
\begin{equation*}
\supp \psi_{\m , h} \subset \overline{E ( \m )} + D ( 0 , \varepsilon ) .
\end{equation*}

$ii)$ When $\m \neq \underline{\m}$, there exists a neighborhood $V_{\tau_{0} , \delta_{0}}$ of $\overline{E ( \m )}$ such that
\begin{equation*}
V_{\tau_{0} , \delta_{0}} \setminus \bigcup_{\s \in {\bf j} ( \m )} \ccc_{\s , 3 \tau_{0} , 3 \delta_{0}} \subset \{ \theta_{\m} v_{\m , h} = 1 \} .
\end{equation*}

$iii)$ When $\m \neq \underline{\m}$, it holds
\begin{equation*}
\forall x \in \supp \nabla ( \theta_{\m} ( v_{\m} + 1 ) ) , \qquad \Big( f ( x ) < \bsigma ( \m ) + \frac{3}{2} \delta_{0} \ \Longrightarrow \ x \in \bigcup_{\s \in {\bf j} ( \m )} \ccc_{\s , 3 \tau_{0} , 3 \delta_{0}} \Big) .
\end{equation*}

\noindent
Let moreover $\m^{\prime}$ belong to $\uuu^{( 0 )}$ with $\m \neq \m^{\prime}$. For every $\tau_{0}>0$ and then $\delta_{0} > 0$ small enough, one has, for $h >0$ small enough,

$iv)$ if $\bsigma ( \m ) = \bsigma ( \m^{\prime} )$ and ${\bf j}(\m)\cap {\bf j}(\m^{\prime} )= \emptyset$, then $\supp ( \psi_{\m , h} ) \cap \supp ( \psi_{\m^{\prime} , h} ) = \emptyset$,

$v)$ if $\bsigma ( \m ) > \bsigma ( \m^{\prime} )$, then 
\begin{itemize}
\item[$\star$] either $\supp ( \psi_{\m , h} ) \cap \supp ( \psi_{\m^{\prime} , h} ) = \emptyset$,
\item[$\star$] or $\psi_{\m , h} = 2 e^{- ( f - f ( \m ) ) / h}$ on $\supp ( \psi_{\m^{\prime} , h} )$.
\end{itemize}
\end{lemma}

\begin{proof}
Points $i)$, $ii)$ and $iii)$ of Lemma \ref{a65} follow from the construction of the quasimodes $\varphi_{\m , h}$ in Definition \ref{a63} for $\m \in \uuu^{( 0 )}$, see indeed \eqref{a60}, \eqref{a61} and \eqref{a62}. Let us then prove the two last points of Lemma \ref{a65}.

When $\bsigma ( \m ) = \bsigma ( \m^{\prime} )$ and $\m \neq \m^{\prime}$, note first that $\m$ and $\m^{\prime}$ differ from $\underline{\m}$ since $\bsigma ( \m ) = + \infty$ if and only if $\m = \underline{\m}$.
Moreover, $\bsigma ( \m ) = \bsigma ( \m^{\prime} )$ and $\m \neq \m^{\prime}$
imply $E(\m)\cap E(\m^{\prime})=\emptyset$. Indeed, the relation $E(\m)\cap E(\m^{\prime})\neq\emptyset$ would imply
that $E(\m)$ and $E(\m^{\prime})$ are the same connected component of $\{f<\bsigma(\m)=\bsigma(\m^{\prime})\}$, in contradiction with the construction of $E$. When in addition ${\bf j}(\m)\cap {\bf j}(\m^{\prime} )= \emptyset$, we then have
\begin{equation*}
\overline{E ( \m )} \cap \overline{E ( \m^{\prime} )} = \partial E ( \m ) \cap \partial E ( \m^{\prime} ) = {\bf j} ( \m ) \cap {\bf j} ( \m^{\prime} )= \emptyset .
\end{equation*}
Combined with $i)$ of Lemma \ref{a65} with $\varepsilon > 0$ sufficiently small, this implies $\supp ( \psi_{\m , h} ) \cap \supp ( \psi_{\m^{\prime} , h} ) = \emptyset$.

When $\bsigma(\m) > \bsigma ( \m^{\prime} )$ and $\m^{\prime} \notin E ( \m )$, we have $\overline{E ( \m )} \cap \overline{E ( \m^{\prime} )} = \emptyset$, and again, according to the first item of Lemma \ref{a65}, it holds $\supp ( \psi_{\m , h} ) \cap \supp ( \psi_{\m^{\prime} , h} ) =\emptyset$ for every $\tau_{0}>0$ and then $\delta_{0} > 0$ small enough. Lastly, when $\bsigma(\m) > \bsigma ( \m^{\prime} )$ and $\m^{\prime} \in E ( \m )$, it holds $\overline{E ( \m^{\prime} )} \subset E_{-} ( \m^{\prime} ) \subset E ( \m )$ and then, according to the second item of Lemma \ref{a65}, $\psi_{\m , h} = 2 e^{- ( f - f ( \m ) ) / h}$ on $\supp ( \psi_{\m^{\prime} , h} )$ for every $\tau_{0}>0$ and then $\delta_{0} > 0$ small enough.
\end{proof}

\section{Proof of the main results} \label{s3}

We will use the following notation, here and in the sequel. For two families of numbers $a = ( a_{h} )_{h\in ] 0,h_{0}]}$ and $b = ( b_{h} )_{h\in ] 0,h_{0}]}$, we say that $a \in \eee_{cl} ( b )$ if there exists a family $( c_{h} )_{h\in ] 0,h_{0}]}$ such that, for every $h\in ] 0,h_{0}]$,
\begin{equation*}
a_{h} = b_{h} c_{h} \text{ and } c_{h} \text{ admits a classical expansion } c_{h} \sim \sum_{j \geq 0} c_{j} h^{j} \text{ with } c_{0} = 1 .
\end{equation*}
We also write $D_{\u} = \big\vert \det \Hess ( f ) ( \u ) \big\vert^{1 / 2}$ for any $\u \in \uuu$.

\subsection{Computation of interaction coefficients}

\begin{proposition}\sl \label{a67}
Under the assumptions of Theorem \ref{a66} but with only the first part of \eqref{a28}, we have the following estimates for  $\tau_{0} > 0$ and then $\delta_{0} > 0$ small enough: there exists $C > 0$ such that, for every $\m , \m^{\prime} \in \uuu^{( 0 )}$ and $h > 0$ small enough,
\begin{equation*}
\begin{aligned}
&i) &&\< \varphi_{\m , h} , \varphi_{\m^{\prime} , h} \> = \delta_{\m , \m^{\prime}} + \ooo ( e^{- C / h} ) ,  \\
&ii) &&\< P \varphi_{\m , h} , \varphi_{\m , h} \> \in \eee_{cl} \Big( h e^{- 2 S ( \m ) / h} \sum_{\s \in {\bf j} ( \m )} \frac{\vert \mu ( \s ) \vert}{2 \pi} \frac{D_{\m}}{D_{\s}} \Big) ,  \\
&iii) &&\Vert P \varphi_{\m , h} \Vert^{2} = \ooo ( h^{\infty} ) \< P \varphi_{\m , h} , \varphi_{\m , h} \> ,  \\
&iv) &&\Vert P^{*} \varphi_{\m , h} \Vert^{2} = \ooo ( h ) \< P \varphi_{\m , h} , \varphi_{\m , h} \> .
\end{aligned}
\end{equation*}
\end{proposition}

The results of this proposition are very close to those of \cite[Propositions 4.4 to 4.6]{LePMi20}. The difference is that we have a classical expansion in $ii)$ and that the multiplicative error in $iii)$ is of order $\ooo ( h^{\infty} )$ instead of $\ooo ( h^{2} )$. This is due to the fact that our quasimodal constructions are sharper.

\begin{remark}\sl \label{g16}
When, in comparison with the assumptions of Proposition \ref{a67}, the first part of \eqref{a28} is not satisfied neither, items $iii)$ and $iv)$ of Proposition \ref{a67} still hold, while item $ii)$ becomes
\begin{equation*}
\< P \varphi_{\m , h} , \varphi_{\m , h} \> \in \eee_{cl} \bigg( h e^{- 2 S ( \m ) / h} \sum_{\s \in {\bf j} ( \m )} \frac{\vert \mu ( \s ) \vert}{2 \pi} \frac{\big(\sum_{\m^{\prime} \in \argmin_{E(\m)} f} D^{-1}_{\m^{\prime} }\big)^{-1}}{D_{\s}} \bigg) .
\end{equation*}
However, the quasi-orthonormality of the family $(\varphi_{\m , h})_{\m \in \uuu^{(0)}}$ stated in $i)$ is not satisfied anymore when the first part of \eqref{a28} does not hold. Indeed, if $\m \neq \m^{\prime} \in \uuu^{( 0 )}$ satisfy $f ( \m ) = f ( \m^{\prime} )$ and $\m^{\prime} \in E ( \m )$, it holds $\bsigma ( \m ) > \bsigma ( \m^{\prime} )$ and it follows from $v)$ in Lemma\ref{a65} that $\psi_{\m , h} = 2 e^{- ( f - f ( \m^{\prime} ) ) / h}$ on $\supp ( \psi_{\m^{\prime} , h} )$. Working with similar arguments as in the proof below, one can then show that there exists $c \in ] 0 , 1 [$ such that $\< \varphi_{\m , h} , \varphi_{\m^{\prime} , h} \> \sim c$ when $h \to 0^{+}$.
\end{remark}

\begin{proof}
The proof of Proposition \ref{a67} follows very closely the proofs of Propositions 4.4 to 4.6 in \cite{LePMi20}. We just sketch it briefly and drop the index $h$ in order to lighten the notation.

First, we recall that, for any $\m \in \uuu^{( 0 )}$, $\varphi_{\m} = \psi_{\m} / \Vert \psi_{\m} \Vert$ with $\psi_{\m}$ given by \eqref{a64}. Moreover, according to the first part of \eqref{a28}, $f$ uniquely attains its absolute minimum  at $\m$ on $\overline{E(\m)}$, and then on $\supp ( \psi_{\m} )$ for every $\tau_{0},\delta_{0}> 0$ small enough (see indeed $i)$ in Lemma \ref{a65}). By a standard Laplace method, we then easily get, for any $\m \in \uuu^{( 0 )}$,
\begin{equation} \label{a68}
\Vert \psi_{\m} \Vert \in \eee_{cl} \Big( 2 (\pi h)^{d / 4} D_{\m}^{- 1 / 2} \Big) .
\end{equation}

Let us now prove $i)$. First, by definition, we have $\< \varphi_{\m} , \varphi_{\m} \> = 1$ for every $\m \in \uuu^{( 0 )}$. Moreover, for every $\m \neq \m^{\prime} \in \uuu^{( 0 )}$, we are in one of the three following cases.
\begin{itemize}
\item[a)] The case when $\bsigma ( \m ) = \bsigma ( \m^{\prime} )$ and ${\bf j} ( \m ) \cap {\bf j} ( \m^{\prime} ) = \emptyset$, in which case we deduce from $iv)$ in Lemma \ref{a65} that $\supp ( \psi_{\m } ) \cap \supp ( \psi_{\m^{\prime} } ) = \emptyset$ and then that $\< \varphi_{\m} , \varphi_{\m^{\prime}} \> = 0$.

\item[b)] The case when $\bsigma ( \m ) = \bsigma ( \m^{\prime} )$ and ${\bf j} ( \m ) \cap {\bf j} ( \m^{\prime} ) \neq \emptyset$ (note that this case does not occur when the second part of \eqref{a28} is satisfied). In this case, we have $\overline{E ( \m )} \cap \overline{E ( \m^{\prime} )}= {\bf j} ( \m ) \cap {\bf j} ( \m^{\prime} ) \neq \emptyset$ and we deduce from the construction of the $\varphi_{\m}$, $\m \in \uuu^{( 0 )}$, that $\varphi_{\m} \varphi_{\m^{\prime}}$ is supported in $\bigcup_{\s \in {\bf j} ( \m ) \cap {\bf j} ( \m^{\prime} )} \ccc_{\s , 3 \tau_{0} , 3 \delta_{0}}$ (see indeed \eqref{a60}, \eqref{a61}, \eqref{a62} and Definition \ref{a63}). Since $\vert \psi_{\m} \vert \leq 2 e^{- ( f - f ( \m ) / h )}$, it follows that
\begin{equation*}
\big\vert \< \varphi_{\m} , \varphi_{\m^{\prime}} \> \big\vert \leq \sum_{\s \in {\bf j} ( \m ) \cap {\bf j} ( \m^{\prime} )} \frac{4}{\Vert \psi_{\m} \Vert \Vert \psi_{\m^{\prime}} \Vert} \big\< e^{- ( f - f ( \m ) ) / h} , e^{- ( f - f ( \m^{\prime} ) ) / h} \big\>_{L^{2} ( \ccc_{\s , 3 \tau_{0} , 3 \delta_{0}} )} .
\end{equation*}
Using \eqref{a68} and $f ( {\bf j} ( \m ) ) = f ( {\bf j} ( \m^{\prime} ) ) > \max ( f(\m) , f ( \m^{\prime} ) )$, the latter relation implies $i)$ for $\tau_{0} > 0$ and then $\delta_{0} > 0$ small enough.

\item[c)] The case when, up to switching $\m$ and $\m^{\prime}$, $\bsigma ( \m ) > \bsigma ( \m^{\prime} )$. Here, according to $v)$ in Lemma \ref{a65}, either $\supp ( \psi_{\m} ) \cap \supp ( \psi_{\m^{\prime}} ) = \emptyset$, in which case $\< \varphi_{\m} , \varphi_{\m^{\prime}} \> = 0$, or $\psi_{\m } = 2 e^{- ( f - f ( \m ) ) / h}$ on $\supp ( \psi_{\m^{\prime} } )$, in which case, since $f \geq f ( \m^{\prime} )$ on $\supp ( \psi_{\m^{\prime} } )$, the Cauchy--Schwarz inequality gives
\begin{equation*}
\< \varphi_{\m} , \varphi_{\m^{\prime}} \> = \frac{2}{\Vert \psi_{\m} \Vert} \< e^{- ( f - f ( \m ) ) / h} , \varphi_{\m^{\prime}} \>_{L^{2}(\supp ( \psi_{\m^{\prime} } ))} = \frac{1}{\Vert \psi_{\m} \Vert} \ooo \big( e^{- ( f(\m^{\prime}) - f ( \m ) ) / h} \big) .
\end{equation*}
The relation $i)$ follows easily, using \eqref{a68} and the relation $f ( \m^{\prime} ) > f ( \m )$ implied by the first part of \eqref{a28}.
\end{itemize}

In order to prove the remaining points $ii)$ to $iv)$ of Proposition \ref{a67}, let us write $\psi_{\m} = \widetilde{v}_{\m} \widetilde{\psi}_{\m}$ with $\widetilde{v}_{\m} = \theta_{\m} ( 1 + v_{\m} )$ and $\widetilde{\psi}_{\m} = e^{- ( f - f ( \m ) ) / h}$. Using this writing and the decomposition $P = P_{2} + P_{1} + P_{0}$ with $P_{0} = c $, $P_{1} = \frac{1}{2} (  b \cdot h \nabla + h \div \circ b )$, and $P_{2} = - h \div \circ A \circ h \nabla$, we get
\begin{equation*}
\begin{aligned}
\< P \psi_{\m} , \psi_{\m} \> &= \< ( P_{2} + P_{0} ) \psi_{\m} , \psi_{\m} \> = \< A h \nabla ( \widetilde{v}_{\m} \widetilde{\psi}_{\m} ) , h \nabla ( \widetilde{v}_{\m} \widetilde{\psi}_{\m}) \> + \< P_{0} \psi_{\m} , \psi_{\m} \>  \\
&= h^{2} \< \widetilde{v}_{\m} A \nabla \widetilde{\psi}_{\m} , \widetilde{v}_{\m} \nabla \widetilde{\psi}_{\m} \> + h^{2} \< \widetilde{\psi}_{\m} A \nabla\widetilde{v}_{\m} , \widetilde{\psi}_{\m} \nabla \widetilde{v}_{\m} \>  \\
&\qquad \qquad \qquad \qquad + 2 h^{2} \< \widetilde{\psi}_{\m} A \nabla \widetilde{v}_{\m} , \widetilde{v}_{\m} \nabla \widetilde{\psi}_{\m} \> + \< P_{0} \psi_{\m} , \psi_{\m} \>    \\
&= - h^{2} \< \div ( \widetilde{v}_{\m}^{2} A \nabla \widetilde{\psi}_{\m} ) , \widetilde{\psi}_{\m} \> + h^{2} \< \widetilde{\psi}_{\m} A \nabla \widetilde{v}_{\m} , \widetilde{\psi}_{\m} \nabla \widetilde{v}_{\m} \>   \\
&\qquad \qquad \qquad \qquad + 2 h^{2} \< \widetilde{\psi}_{\m} A \nabla \widetilde{v}_{\m} , \widetilde{v}_{\m} \nabla \widetilde{\psi}_{\m} \> + \< P_{0} \psi_{\m} , \psi_{\m} \>  \\
&= \< \widetilde{v}_{\m}^{2} P_{2} \widetilde{\psi}_{\m} , \widetilde{\psi}_{\m} \> + h^{2} \< \widetilde{\psi}_{\m} A \nabla \widetilde{v}_{\m} , \widetilde{\psi}_{\m} \nabla \widetilde{v}_{\m} \> + \< \widetilde{v}_{\m}^{2} P_{0} \widetilde{\psi}_{\m} , \widetilde{\psi}_{\m} \> ,
\end{aligned}
\end{equation*}
and, since $( P_{2} + P_{0} ) ( \widetilde{\psi}_{\m} ) = 0$, it implies
\begin{equation*} \label{a69}
\< P \psi_{\m} , \psi_{\m} \> = h^{2} \< \widetilde{\psi}_{\m} A \nabla \widetilde{v}_{\m} , \widetilde{\psi}_{\m} \nabla \widetilde{v}_{\m} \> .
\end{equation*}
Since $f - f ( \m ) > S ( \m ) + \delta_{0}$ on $\supp ( \nabla \theta_{\m} )$ (see \eqref{a62}), it follows that
\begin{equation*}
\< P \psi_{\m} , \psi_{\m} \> = h^{2} \int \theta_{\m}^{2} A \nabla v_{\m} \cdot \nabla v_{\m} e^{- 2 ( f - f ( \m ) ) / h} d x + \ooo \big( e^{- 2 ( S ( \m ) + \delta_{0} ) / h} \big) .
\end{equation*}
On the other hand, thanks to \eqref{a61}, we have on $\bigcup_{\s \in {\bf j} ( \m )}\ccc_{\s , 3 \tau_{0} , 3 \delta_{0}}$
\begin{equation} \label{v69}
\nabla v_{\m} = \sum_{\s \in {\bf j} ( \m )} \frac{1}{C_{\s , h}} \zeta ( \ell_{\s} / \tau_{0} ) e^{- \ell_{\s}^{2} / 2 h} \nabla \ell_{\s} ,
\end{equation}
which yields, since $\nabla v_{\m} = 0$ on $E^{+}_{\m , 3 \tau_{0} , 3 \delta_{0}} \cup E^{-}_{\m , 3 \tau_{0} , 3 \delta_{0}}$ by \eqref{a60},
\begin{align*}
\< P \psi_{\m} , \psi_{\m} \> = h^{2} \sum_{\s \in {\bf j} ( \m )} \frac{1}{C_{\s , h}^{2}} \int_{\ccc_{\s , 3 \tau_{0} , 3 \delta_{0}}} \theta_{\m}^{2} \zeta ( \ell_{\s} / \tau_{0} )^{2} A \nabla \ell_{\s} \cdot \nabla \ell_{\s} & e^{- 2 \big( f + \frac{\ell_{\s}^{2}}{2} - f ( \m ) \big) / h} d x   \\
&\qquad + \ooo \big( e^{- 2 ( S ( \m ) + \delta_{0} ) / h} \big) .
\end{align*}
The first term of the right hand side can now be computed as in \cite[Proof of Proposition 4.5]{LePMi20}, the only difference here being that $A$ and $\ell_{\s}$ depend on $h$ and admit a classical expansion with respect to $h$.  More precisely, since
\begin{equation*}
f + \frac{\ell_{\s}^{2}}{2} - f ( \m ) = f + \frac{\ell_{\s , 0}^{2}}{2} - f ( \m ) + \ell_{\s , 0} \ell_{\s , 1} h + \ooo ( h^{2} ) ,
\end{equation*}
where, according to Lemma \ref{a38} and to \eqref{g14}, the function $f + \frac{\ell_{\s , 0}^{2}}{2}$ satisfies
\begin{equation} \label{g15}
\Big( f + \frac{\ell_{\s , 0}^{2}}{2} \Big) ( \s ) = f ( \s ) = \bsigma ( \m ) , \quad \nabla \Big( f + \frac{\ell_{\s , 0}^{2}}{2} \Big) ( \s ) = 0 \quad \text{and} \quad \Hess \Big( f + \frac{1}{2} \ell_{\s , 0}^{2} \Big) ( \s ) > 0 ,
\end{equation}
we can apply the Laplace method with the phase function $- 2 \big( f + \frac{\ell_{\s , 0}^{2}} 2 - f ( \m ) \big)$. Using  $\ell_{\s , 0} ( \s ) = 0$ and \eqref{g19}, this yields
\begin{equation*}
\< P \psi_{\m} , \psi_{\m} \> \in \eee_{cl} \Big( \frac{2 h}{\pi} ( \pi h )^{\frac{d}{2}} \sum_{\s \in {\bf j} ( \m )} \big( A^{0} \nabla \ell_{\s , 0} \cdot \nabla\ell_{\s , 0} \big) ( \s ) \Big( \det \Hess \Big( f + \frac{1}{2} \ell_{\s , 0}^{2} \Big) ( \s ) \Big)^{- 1 / 2} e^{- 2 S ( \m ) / h} \Big) .
\end{equation*}
Moreover, thanks to Lemma \ref{a42}, one has 
\begin{equation*}
\big( A^{0} \nabla \ell_{\s , 0} \cdot \nabla \ell_{\s , 0} \big) ( \s ) = \vert \mu ( \s ) \vert \qquad \text{and} \qquad \det \Hess \Big( f + \frac{1}{2} \ell_{\s , 0}^{2} \Big) ( \s ) = \vert \det \Hess ( f ) ( \s ) \vert ,
\end{equation*}
and hence
\begin{equation} \label{g18}
\< P \psi_{\m} , \psi_{\m} \> \in \eee_{cl} \Big( \frac{2 h}{\pi} ( \pi h )^{\frac{d}{2}} e^{- 2 S ( \m ) / h} \sum_{\s \in {\bf j} ( \m )} \vert \mu ( \s ) \vert D_{\s}^{- 1} \Big) .
\end{equation}
Combining \eqref{g18} with \eqref{a68} proves $ii)$.

Let us now prove $iii)$. Since $P ( e^{- f / h} ) = 0$, one has 
\begin{equation*}
P \psi_{\m} = P \big( \theta_{\m} ( v_{\m} + 1 ) e^{- ( f - f ( \m ) ) / h} \big) = [ P , \theta_{\m} ] ( v_{\m} + 1 ) e^{- ( f - f ( \m ) ) / h} + \theta_{\m} P \big( v_{\m} e^{- ( f - f ( \m ) ) / h} \big) ,
\end{equation*}
and, since $f - f ( \m ) > S ( \m ) + \delta_{0}$ on $\supp ( \nabla \theta_{\m} )$ and on $\supp ( 1 - \theta_{\m} ) \cap \supp \theta_{\m}$, this implies
\begin{align} 
\Vert P \psi_{\m} \Vert^{2} &= \big\Vert \theta_{\m} P \big( v_{\m} e^{- ( f - f ( \m ) ) / h} \big) \big\Vert^{2} + \ooo \big( e^{- 2 ( S ( \m ) + \delta_{0} ) / h} \big)   \nonumber \\
&= \big\Vert P \big( v_{\m} e^{- ( f - f ( \m ) ) / h} \big) \big\Vert_{L^{2} ( \supp \theta_{\m} )}^{2} + \ooo \big( e^{- 2 ( S ( \m ) + \delta_{0} ) / h} \big) .   \label{a74}
\end{align}
On the other hand, on $\supp \theta_{\m}$, $P ( v_{\m} e^{- ( f - f ( \m ) ) / h} )$ is supported in $\bigcup_{\s \in {\bf j} ( \m )} \ccc_{\s , 3 \tau_{0} , 3 \delta_{0}}$ by \eqref{a60} and, on any $\ccc_{\s , 3 \tau_{0} , 3 \delta_{0}}$, one has (see \eqref{a61})
\begin{equation*}
P \big( v_{\m} e^{- ( f - f ( \m ) ) / h} \big) = C^{- 1}_{\s , h}( w + r ) e^{- ( f - f ( \m ) +  \frac{\ell^{2}_{\s}}{2} ) / h} ,
\end{equation*}
where $w$ and $r$ are given by Lemma \ref{a34}. Since $f + \frac{\ell^{2}_{\s}}{2} = f + \frac{\ell^{2}_{\s , 0}}{2} + \ooo ( h )$, the Hessian of $f + \frac{\ell^{2}_{\s,0}}{2}$ at $\s$ is positive definite (see \eqref{g15}), and $r$ is supported away from $\s$, one has, for some $\delta > 0$,
\begin{equation*}
\big\Vert r e^{- ( f - f ( \m ) +  \frac{\ell^{2}_{\s}}{2} ) / h} \big\Vert_{L^{2} ( \ccc_{\s , 3 \tau_{0} , 3 \delta_{0}} )}^{2} = \ooo \big( e^{- 2 ( S ( \m ) + \delta ) / h} \big) = \ooo ( h^{\infty} ) \< P \psi_{\m} , \psi_{\m} \> ,
\end{equation*}
where we used \eqref{g18} to obtain the last equality. Moreover, thanks to Proposition \ref{a52} (and to Lemma \ref{a34}), one also has 
\begin{equation*}
\big\Vert w e^{- ( f - f ( \m ) +  \frac{\ell^{2}_{\s}}{2} ) / h} \big\Vert_{L^{2} ( \ccc_{\s , 3 \tau_{0} , 3 \delta_{0}} )}^{2} = \ooo ( h^{\infty} ) e^{- 2 S ( \m ) / h} = \ooo ( h^{\infty} ) \< P \psi_{\m} , \psi_{\m} \> .
\end{equation*}
These two estimates and \eqref{a74} show that $\Vert P \psi_{\m} \Vert^{2} = \ooo ( h^{\infty} ) \< P \psi_{\m} , \psi_{\m} \>$,
which proves $iii)$.

The proof of $iv)$ is similar and left to the reader.
\end{proof}

\subsection{Proof of Theorem \ref{a66}}

Until the end of this section, the local minima $\m_{1} , \ldots , \m_{n_{0}}$ of $f$ are labeled so that $( S ( \m_{j} ) )_{j \in \{ 1 , \dots , n_{0} \}}$ is non-decreasing (see \eqref{a27}). That is 
\begin{equation}\label{b32}
 \text {  for all } j \in \{ 0, \dots , n_{0}-1 \} , \ S ( \m_{j + 1} ) \geq S ( \m_{j} )  \text{ and } S_{\m_{n_0}}=+\infty.
\end{equation}
For $j \in \{ 1 , \ldots , n_{0} \}$, we will also denote for shortness
\begin{equation}\label{h2}
S_{j} : = S ( \m_{j} ) , \quad \varphi_{j} : = \varphi_{\m_{j} , h} \quad \text{and} \quad \widetilde{\lambda}_{j} : = \< P \varphi_{j} , \varphi_{j} \> .
\end{equation}
From Proposition \ref{a67}, one knows that, for all $j \in \{ 1 , \ldots , n_{0} \}$, we have
\begin{equation} \label{a70}
\widetilde{\lambda}_{j} \in \eee_{cl} \Big(  h e^{- 2 S_{j} / h} \sum_{\s \in {\bf j} ( \m_{j} )} \frac{\vert \mu ( \s ) \vert}{2 \pi} \frac{D_{\m_j}}{D_{\s}} \Big)
\end{equation}
and, for $j,k\in \{ 1 , \ldots , n_{0} \} $,
\begin{equation} \label{a71}
\< \varphi_{j} , \varphi_{k} \>=\delta_{j,k} + \ooo(e^{-\frac Ch}) , \quad \Vert P \varphi_{j} \Vert = \ooo \big( h^{\infty} \sqrt{\widetilde{\lambda}_{j}} \big) \quad \text{and} \quad \Vert P^{*} \varphi_{j} \Vert = \ooo \big( \sqrt{h \widetilde{\lambda}_j} \big).
\end{equation}
Using in addition $iv)$ and $v)$ in Lemma \ref{a65}
together with the second part of \eqref{a28} (see \cite[Proof of Lemma 4.7]{LePMi20} for details), we also have
\begin{equation} \label{g35}
\< P \varphi_{j} , \varphi_{k} \> = \delta_{j , k} \widetilde{\lambda}_{j}.
\end{equation}
We then introduce the spectral projector 
\begin{equation} \label{g34}
\Pi_{h} = \frac{1}{2 i \pi} \int_{\partial D ( 0 , \varepsilon_{*} h / 2 )} ( z - P )^{- 1} d z ,
\end{equation}
with $\varepsilon_{*}$ given by Proposition \ref{a22}. 
Working as in the proof of \cite[Lemma 4.9]{LePMi20}, one deduces from the two last estimates of \eqref{a71}
and from the resolvent estimate \eqref{a18} of Theorem \ref{a15} that 
\begin{equation} \label{a5}
( 1 - \Pi_{h} ) \varphi_{j} = \ooo \big( h^{\infty} \sqrt{\widetilde{\lambda}_{j}} \big) \quad \text{and} \quad ( 1 - \Pi_{h}^{*} ) \varphi_{j} = \ooo \big( \sqrt{h \widetilde{\lambda}_{j}} \big) .
\end{equation}

The estimates \eqref{a71}, \eqref{g35}, and \eqref{a5} easily imply the following proposition, whose proof is the same as the proof of Proposition 4.10 in
\cite{LePMi20}.

\begin{proposition}\sl \label{a72}
For every $j \in \{ 1 , \dots , n_{0} \}$ and $h > 0$ small enough, let us define $v_{j} : = \Pi_{h} \varphi_{j}$. Then, there exists $c > 0$ such that, for all $j , k \in \{ 1 , \ldots , n_{0} \}$, one has
\begin{equation} \label{a73}
\< v_{j} , v_{k} \> = \delta_{j , k} + \ooo ( e^{- c / h} ) 
\end{equation}
and
\begin{equation} \label{a75}
\< P v_{j} , v_{k} \> = \delta_{j , k} \widetilde{\lambda}_{j} + \ooo \big( h^{\infty} \sqrt{\widetilde{\lambda}_{j} \widetilde{\lambda}_{k}} \big) .
\end{equation}
In particular, it follows from \eqref{a73} that for every $h > 0$ small enough, the family $( v_{1} , \dots , v_{n_{0}} )$ is a basis of $\Ran \Pi_{h}$.
\end{proposition}

The end of the proof follows line by line the proof of Theorem 1.9 in \cite[pp. 39--42]{LePMi20}. First, we orthonormalize the basis $( v_{n_0-j+1} )_{j\in\{1,\dots,n_{0}\}}$ into a basis $( e_{n_0-j+1} )_{j\in\{1,\dots,n_{0}\}}$ by the Gram--Schmidt process.
Thanks to \eqref{a73}, we obtain an orthonormal basis of  $\Ran \Pi_{h}$ such that, for every $j\in\{1,\dots,n_{0}\}$,
\begin{equation*}
e_{n_{0} - j + 1} = v_{n_{0} - j + 1} + \ooo ( e^{- C / h} ),
\end{equation*}
(see \cite[Lemma 4.11]{LePMi20}). We then show, using \eqref{a75} and the above labeling of the basis $( e_{n_0-j+1} )_{j\in\{1,\dots,n_{0}\}}$
starting from $e_{n_{0}}$, that, for every $j,k\in \{1,\dots,n_{0}\}$,
\begin{equation} \label{g36}
\< P e_{n_{0}-j+1} , e_{n_{0}-k+1} \> = \delta_{j , k} \widetilde{\lambda}_{n_{0}-j+1} + \ooo \big( h^{\infty} \sqrt{\widetilde{\lambda}_{n_{0}-j+1} \widetilde{\lambda}_{n_{0}-k+1}} \big) ,
\end{equation}
(see \cite[Proposition 4.12]{LePMi20}). We can then compute the eigenvalues of the matrix $M_{h}$ of $P_{\vert \Ran \Pi_{h}}$ in this basis, using the so-called graded structure of $M_{h}$ (see  \cite[pp. 41--42 and Theorem A.4]{LePMi20} or Theorem \ref{b10} of the next section). Since $\widetilde{\lambda}_{j}$ (or more precisely $e^{ 2 S_{j} / h}\widetilde{\lambda}_{j}$, see \eqref{a70}) admits a classical expansion, and thanks to the $h^{\infty}$ multiplicative errors in \eqref{g36}, we obtain the announced result.

\subsection{Proof of Corollaries \ref{g30} and \ref{g33}}

Recall that $\Pi_{h}$, the spectral projector of $P$ associated with the $n_{0}$ exponentially small eigenvalues of $P$, has been defined in \eqref{g34}. Since $\partial D ( 0 , \varepsilon_{*} h / 2 )$ is at a distance of order  $h$ from the spectrum of $P$ (see Proposition \ref{a22}), \eqref{a18} implies that there exists $C > 0$ such that, for every $h$ small enough, 
\begin{equation} \label{g37}
\Vert \Pi_{h} \Vert \leq C .
\end{equation}
Then, \eqref{a18} and \eqref{g37} give that $( P - z )^{- 1} ( 1 - \Pi_{h} )$ is of order $h^{- 1}$ in $\{ \Re z < 2 \varepsilon_{*} h / 3 \} \setminus D ( 0 , \varepsilon_{*} h / 2 )$. Since this operator-valued function is holomorphic in $\{ \Re z < 2 \varepsilon_{*} h / 3 \}$, the maximum principle yields
\begin{equation} \label{g38}
\big\Vert ( P - z )^{- 1} ( 1 - \Pi_{h} ) \big\Vert \leq C h^{- 1} ,
\end{equation}
for $\Re z < 2 \varepsilon_{*} h / 3$ and $h$ small enough.

The solution of \eqref{g32} can be written
\begin{equation} \label{g42}
u ( t ) = e^{- t P / h} u_{0} = e^{- t P / h} \Pi_{h} u_{0} + e^{- t P / h} ( 1 - \Pi_{h} ) u_{0} .
\end{equation}
Let $Q:\Im ( 1 - \Pi_{h} )\to \Im ( 1 - \Pi_{h} )$ be the operator $P$ restricted to the Hilbert space $\Im ( 1 - \Pi_{h} )$. Since $P$ is maximal accretive, so is $Q$ and thus $\Vert e^{- t Q / h} \Vert \leq 1$. Moreover, \eqref{g38} shows that $\Vert ( Q - z )^{- 1} \Vert \leq C h^{- 1}$ for $\Re z < 2 \varepsilon_{*} h / 3$. To estimate the last term in \eqref{g42}, we use a Gearhardt--Pr\"{u}ss type inequality with an explicit bound. More precisely,  \cite[Theorem 1.4]{HeSj21_01} (see  \cite[Proposition 2.1]{HeSj10_01} for more details) gives, for some $C>0$ and all $t \geq 0$,
\begin{equation*}
\big\Vert e^{- t Q / h} \big\Vert \leq \Big( 1 + 2 \frac{\varepsilon_{*} h}{2} \sup_{\Re z = \varepsilon_{*} h / 2} \big\Vert ( Q - z )^{- 1} \big\Vert \Big) e^{- t \frac{\varepsilon_{*} h}{2} / h} \leq C e^{- t \varepsilon_{*} / 2} .
\end{equation*}
 Combined with \eqref{g37}, it implies the existence of $C>0$ such that, for all $t \geq 0$,
\begin{equation} \label{g43}
\big\Vert e^{- t P / h} ( 1 - \Pi_{h} ) u_{0} \Big\Vert \leq C e^{- \varepsilon t} \Vert u_{0} \Vert ,
\end{equation}
where $\varepsilon = \varepsilon_{*} / 2$. On the other hand, $P$ restricted to $\Im \Pi_{h}$ is a matrix of size~$n_{0}$ whose eigenvalues are the $\lambda ( \m , h )$'s. Then, \eqref{g42}, \eqref{g43}, and the usual formula for the exponential of a matrix applied to $e^{- t P / h} \Pi_{h} u_{0}$ provide \eqref{g31}. Moreover, using \eqref{g43} instead of the argument  of \cite[page 43]{LePMi20}, the proof of \eqref{g49} is similar to the one of  \cite[Theorem 1.11]{LePMi20},
except we have here to apply Theorem~\ref{b10} instead of \cite[Theorem A.4]{LePMi20},  and we omit the details. Summing up, we have just shown Corollary \ref{g30}.

Let us now prove Corollary \ref{g33}. For $R > 1$, we define the balls
\begin{equation*}
\forall k \in \{ 1 , \ldots , K - 1 \} , \qquad D_{k} = D \big( ( R + R^{- 1} ) h e^{- 2 S_{k} / h} , R h e^{- 2 S_{k} / h} \big) ,
\end{equation*}
and $D_{K} = D ( 0 , R^{- 1} h e^{- 2 S_{K - 1} / h} )$. For $R$ fixed large enough and every $h$ small enough, each exponentially small eigenvalue of $P$ belongs to exactly one of the disjoint sets $D_{k}$ from Theorem~\ref{a66}. Moreover, $\partial D_{k}$ is at distance of order $h e^{- 2 S_{k} / h}$ (resp. $h e^{- 2 S_{K - 1} / h}$) from the spectrum of $P$ for $k \in \{ 1 , \ldots , K - 1 \}$ (resp. $k = K$). Using \eqref{b13} to estimate the resolvent of $P$ on the image of $\Pi_{h}$ and \eqref{g38} to control the contribution on the image of $1 - \Pi_{h}$, we get
\begin{equation} \label{g45}
\forall z \in \partial D_{k} , \qquad \big\Vert ( P - z )^{- 1} \big\Vert \leq \left\{ \begin{aligned}
&C h^{- 1} e^{2 S_{k} / h}  &&\text{ for } k \in \{ 1 , \ldots , K - 1 \} , \\
&C h^{- 1} e^{2 S_{K - 1} / h}  &&\text{ for } k = K .
\end{aligned} \right.
\end{equation}
In particular, the spectral projector associated with the eigenvalues of order $h e^{- 2 S_{k} / h}$,
\begin{equation} \label{g40}
\widetilde{\Pi}_{k} = \frac{1}{2 i \pi} \int_{\partial D_{k}} ( z - P )^{- 1} d z ,
\end{equation}
is well-defined and satisfies $\Vert \widetilde{\Pi}_{k} \Vert \leq C$.

We can now decompose
\begin{equation} \label{g44}
e^{- t P / h} \Pi_{h} = \sum_{1 \leq k \leq K} e^{- t P / h} \widetilde{\Pi}_{k} .
\end{equation}
For $k \in \{ 1 , \ldots , K - 1 \}$ and $0 \leq t \leq t_{k}^{-}$, \eqref{g45} and $t_{k}^{-} e^{- 2 S_{k} / h} = \ooo ( h^{\infty} )$ imply
\begin{align}
e^{- t P / h} \widetilde{\Pi}_{k} &= \frac{1}{2 i \pi} \int_{\partial D_{k}} e^{- t z / h} ( z - P )^{- 1} d z  \nonumber  \\
&= \frac{1}{2 i \pi} \int_{\partial D_{k}} ( z - P )^{- 1} d z + \frac{1}{2 i \pi} \int_{\partial D_{k}} \big( e^{- t z / h} - 1 \big) ( z - P )^{- 1} d z    \nonumber  \\
&= \widetilde{\Pi}_{k} + \int_{\partial D_{k}} \ooo ( t \vert z \vert / h ) \big\Vert ( P - z )^{- 1} \big\Vert \, d z  \nonumber  \\
&= \widetilde{\Pi}_{k} + \ooo \big( t h e^{- 2 S_{k} / h} / h \big) \nonumber  \\
&= \widetilde{\Pi}_{k} + \ooo ( h^{\infty} ) . \label{g46}
\end{align}
On the contrary, for $t_{k}^{+} \leq t$, \eqref{g45} and $e^{- t_{k}^{+} e^{- 2 S_{k} / h} / R} = \ooo ( h^{\infty} )$ give
\begin{align}
e^{- t P / h} \widetilde{\Pi}_{k} &= \frac{1}{2 i \pi} \int_{\partial D_{k}} e^{- t z / h} ( z - P )^{- 1} d z  \nonumber  \\
&= \ooo \Big( \int_{\partial D_{k}} e^{- t \Re z / h} \big\Vert ( P - z )^{- 1} \big\Vert \, d z \Big) \nonumber  \\
&= \ooo \big( e^{- t R^{- 1} h e^{- 2 S_{k} / h} / h} \big) \int_{\partial D_{k}} \big\Vert ( P - z )^{- 1} \big\Vert \, d z  \nonumber  \\
&= \ooo ( h^{\infty} ) . \label{g47}
\end{align}
Lastly, $e^{- t P / h} \widetilde{\Pi}_{K} = \widetilde{\Pi}_{K}$ since $\widetilde{\Pi}_{K}$ is the rank-one spectral projector associated with the eigenvalue $0$. On the other hand, since $e^{- \varepsilon t_{0}^{+}} = \ooo ( h^{\infty} )$, \eqref{g43} becomes
\begin{equation} \label{g48}
\big\Vert e^{- t P / h} ( 1 - \Pi_{h} ) \Big\Vert = \ooo ( h^{\infty} ) ,
\end{equation}
for $t \geq t_{0}^{+}$. Summing up, Corollary \ref{g33} is a direct consequence of the formulas \eqref{g42} and \eqref{g44} with the relation $\Pi^{\leq}_{k} = \sum_{k \leq j \leq K} \widetilde{\Pi}_{j}$ and the estimates \eqref{g46}, \eqref{g47}, and \eqref{g48}.

\section{Generalization} \label{b28}

In this part, we briefly explain how one can drop the assumption  \eqref{a28} and treat the general case in the spirit of \cite{Mi19}. This requires to introduce some additional material of \cite{Mi19}.

\begin{defin}\sl \label{b33}
Let $\m \in \uuu^{( 0 )} \setminus \{ \underline \m \}$. We say that $\m$ is of type I when $f ( \widehat{\m} ) < f ( \m )$ and that $\m$ is of type II when $f ( \widehat{\m}  ) = f ( \m )$. We will also denote
\begin{align*}
\uuu^{( 0 ) , I} &:= \{ \m \in \uuu^{( 0 )}\setminus\{\underline\m\} ; \ \m \text{ is of type I} \} ,  \\
\uuu^{( 0 ) , II} &:= \{ \m \in \uuu^{( 0 )}\setminus\{\underline\m\} ; \ \m \text{ is of type II} \} .
\end{align*}
We have clearly the following disjoint union $\uuu^{( 0 )}\setminus\{\underline\m\} = \uuu^{( 0 ) , I} \bigsqcup \uuu^{( 0 ) , II }$.
\end{defin}

Given $\sigma \in \Sigma$, let $\Omega_{\sigma} := \{ E ( \m ) ; \ \m \in \bsigma^{- 1} ( \sigma ) \} \bigcup \{ \widehat E ( \m ) ; \ \m \in \bsigma^{- 1} ( \sigma ) \cap \uuu^{( 0 ) , II} \}$.

\begin{defin}\sl
\label{de.equiv}
We define an equivalence relation $\rrr$ on $\uuu^{( 0 )}$ by $\m \rrr \m^{\prime}$ if and only if
\begin{equation}
\left\{
\begin{aligned}
&\bsigma ( \m ) = \bsigma ( \m^{\prime} ) = \sigma ,    \\
&\exists \omega_{1} , \ldots , \omega_{K} \in \Omega_\sigma \ \ \text{s.t.}\ \   \m \in \omega_{1} , \ \m^{\prime} \in \omega_{K}, \text{ and } \forall k = 1 , \ldots , K - 1 , \ \overline{\omega_{k}} \cap \overline{\omega_{k + 1}} \neq \emptyset .
\end{aligned}
\right.
\end{equation}
\end{defin}

We denote by $\Cl ( \m )$ the equivalence class of $\m$ for the relation $\rrr$. Observe that $\Cl ( \underline{\m} ) = \{ \underline{\m} \}$ since $\underline\m$ is the only $\m \in \uuu^{( 0 )}$  such that $\bsigma ( \m ) = +\infty$. Let $\aaa$ denote the (finite) set made of the equivalence classes for $\rrr$ and, for $\alpha \in \aaa$, let $\uuu^{( 0 )}_{\alpha}$ be the set of the elements of the class~$\alpha$. We have then evidently
\begin{equation*}
\uuu^{( 0 )} = \bigsqcup_{\alpha \in \aaa} \uuu^{( 0 )}_{\alpha} = \{ \underline{\m} \} \bigsqcup \bigsqcup_{\alpha \in \aaa \setminus \{ \Cl ( \underline{\m} )\}} \uuu^{( 0 )}_{\alpha} .
\end{equation*}
In the theorem below, we sum up in a rather vague way  the description of the small eigenvalues of $P$. Precise statements are given in Theorems~\ref{b23} and~\ref{b20}.

\begin{theorem}\sl \label{b34}
Suppose that the assumptions of Theorem \ref{a15} are satisfied and that \eqref{a17}, \eqref{h1}, and \eqref{a24} hold true. Let $\varepsilon_{*}$ be given by Proposition  \ref{a22}. Then, for $h > 0$ small enough, there exists a bijection, taking into account multiplicities,
\begin{equation*}
\beta : \{ 0 \} \cup \bigcup_{\alpha \in \aaa \setminus \{ \Cl ( \underline{\m} )\}} \bigcup_{j = 1}^{p} h e^{- 2 S_{j} / h} \sigma ( M_{\alpha,j} ) \longrightarrow \sigma ( P ) \cap \{ \Re z <\varepsilon_{*} h \} ,
\end{equation*}
with $\beta ( z ) = z + \ooo ( h^{\infty} \vert z \vert )$, for some  symmetric positive definite matrices $M_{\alpha , j}$ having a classical expansion in powers of $h$ with explicit invertible leading term and for some labeling $( S_{j} )_{j\in\{1,\dots,p\}}$ of $S ( \uuu^{( 0 )}\setminus\{\underline\m\} )$.
\end{theorem}

Here, the set $h e^{- 2 S_{j} / h} \sigma ( M_{\alpha,j} )$ is empty whenever $S_{j} \notin S(\uuu^{( 0 )}_{\alpha})$. The general strategy to prove Theorem \ref{b34} (leading to the explicit expression of the matrices $M_{\alpha , j}$) is to combine the quasimodal constructions near the saddle points developed in the preceding section together with the topological classification of the saddle points introduced in \cite{Mi19}. In the latter work, the construction of the quasimode $\varphi_{\m}$ depends on the fact that $\m$ is of type I or II. In order to lighten the presentation, we assume from now that every point $\m$ is of type I and will prove Theorem \ref{b34} under this assumption, i.e. when $\uuu^{( 0 ) , II} = \emptyset$. In that case, the leading term of $M_{\alpha , j}$ is given in Theorem \ref{b20} below, which makes explicit the statement of Theorem \ref{b34}
when $\uuu^{( 0 ) , II} = \emptyset$. The reader may check that the construction below can be adapted to the case of type II points as in \cite{Mi19}.

\begin{remark}\sl \label{g17}
Note that  $\uuu^{( 0 ) , II} = \emptyset$ if and only if $f ( \widehat{\m} ) < f ( \m )$ for every  $\m \in \uuu^{( 0 )} \setminus \{ \underline \m \}$ (see Definition \ref{b33}). It follows that $\uuu^{( 0 ) , II} = \emptyset$ if and only if the first part of \eqref{a28} is satisfied. Indeed, if $\m$  is the unique global minimum of  $f_{\vert E ( \m )}$ for every $\m \in \uuu^{( 0 )}$,
then, for every $\m \in \uuu^{( 0 )} \setminus \{ \underline \m \}$, the relation $\m \in E ( \widehat{\m} )$ implies $f ( \widehat{\m} ) <f(\m)$. Conversely, by contraposition, assume  the existence of $\m \neq \m^{\prime} \in \uuu^{( 0 )}$ such that $\m^{\prime} \in E ( \m )$ and $f ( \m^{\prime} ) = f ( \m )$. It then holds $\bsigma ( \m ) > \bsigma ( \m^{\prime} )$ and thus $\widehat{\m^{\prime}} \in E ( \m )$, which implies $f ( \m ) \leq f ( \widehat{\m^{\prime}} ) \leq f ( \m^{\prime} )$ (see \eqref{b31}), and then $f ( \m^{\prime} ) = f ( \widehat{\m^{\prime}} )$, i.e. $\m^{\prime} \in \uuu^{( 0 ) , II} \neq \emptyset$.

In particular, the statement of Proposition \ref{a67} is valid when $\uuu^{( 0 ) , II} = \emptyset$. This will be used in the sequel.
\end{remark}

\subsection{Quasimodal constructions for type I minima}

Let $( \psi_{\m} )_{\m \in \uuu^{( 0 )}}$ denote the family of quasimodes of Definition \ref{a63}. We recall that, when $\m\neq \underline\m$, 
\begin{equation*} 
\psi_{\m} ( x ) = \theta_{\m} ( x ) \big( v_{\m } ( x ) + 1 \big) e^{- ( f ( x ) - f ( \m ) ) / h} ,
\end{equation*}
with $\theta_{\m}$ and $v_{\m}$ defined by \eqref{a60}, \eqref{a61}, and \eqref{a62} (here and throughout we dropped the subscript $h$ to lighten the notation). In particular, near any point $\s \in {\bf j} ( \m )$, one has
\begin{equation*}
v_{\m} ( x ) = C^{- 1}_{\s , h} \int_{0}^{\ell_{\s} ( x , h )} \zeta ( r / \tau_{0} ) e^{- \frac{r^{2}}{2 h}} \, d r ,
\end{equation*}
where the function $\ell_{\s}$ is the function defined by Proposition \ref{a52} such that there exists a neighbourhood $V$ of $\s$ such that $E ( \m ) \cap V \subset\{ \eta ( \s ) \cdot ( x - \s ) > 0 \}$ (see the lines below \eqref{a61}). This choice of sign depends of course of $\m$ and in order to avoid any confusion, we shall denote by $\ell_{\s , \m}$ the function $\ell_{\s}$ above.

\begin{lemma}\sl \label{g20}
Let $\m \neq \m^{\prime} \in \uuu^{( 0 )}$ satisfy ${\bf j} ( \m ) \cap {\bf j} ( \m^{\prime} )\neq\emptyset$  and take $\s \in {\bf j} ( \m ) \cap {\bf j} ( \m^{\prime} )$. Then, the functions $\ell_{\s , \m}$ and $\ell_{\s , \m^{\prime}}$ can be chosen so that $\ell_{\s , \m^{\prime}} = - \ell_{\s , \m}$ near $\s$.
\end{lemma}

\begin{proof}
Note that the function $\ell_{\s,\n}$ only makes sense when $\n\in \{\m,\m^{\prime}\}$. We assume that $\ell_{\s , \m}$ is given by Proposition \ref{a52} with the sign condition for $\m$ below \eqref{a61}. As explained at the end of Section \ref{b36}, $- \ell_{\s , \m}$ also satisfies Proposition \ref{a52} with the opposite sign condition. Thus, this function satisfies the sign condition for $\m^{\prime}$ and can be chosen as the function $\ell_{\s , \m^{\prime}}$.
\end{proof}

\begin{proposition}\sl \label{b19}
Assume that the hypotheses of Theorem \ref{b34} hold true
and that $\uuu^{( 0 ) , II} = \emptyset$. Then, the conclusions of Proposition \ref{a67} are satisfied. Moreover, for every $\m , \m^{\prime} \in \uuu^{( 0 )}$, one has
\begin{equation} \label{b27}
\< P \varphi_{\m} , \varphi_{\m^{\prime}} \> \in \eee_{cl} \bigg( ( - 1 )^{\delta_{\m , \m^{\prime}} - 1} h e^{- (S ( \m ) + S ( \m^{\prime} ) ) / h} \sum_{\s \in {\bf j} ( \m ) \cap {\bf j} ( \m^{\prime} )} \frac{\vert \mu ( \s ) \vert}{2 \pi} \frac{\sqrt{D_{\m} D_{\m^{\prime}}}}{D_{\s}} \bigg) .
\end{equation}
Finally, denoting $P^{\sharp} = P_{2} + P_{0}$, there exists $\alpha > 0$ such that
\begin{equation} \label{b22}
\< P \varphi_{\m} , \varphi_{\m^{\prime}} \> = \< P^{\sharp} \varphi_{\m} , \varphi_{\m^{\prime}} \> \big( 1 + \ooo ( e^{- \alpha / h} ) \big) .
\end{equation}
\end{proposition}

\begin{remark}\sl
$i)$ Note that the relation \eqref{b27} implies that 
\begin{equation} \label{b21}
\< P \varphi_{\m} , \varphi_{\m^{\prime}} \> = 0 ,
\end{equation}
for every $\m,\m^{\prime}\in\uuu^{(0)}$ with ${\bf j} ( \m ) \cap {\bf j} ( \m^{\prime} ) = \emptyset$.
This is in particular the case when $\m^{\prime} \notin \Cl ( \m )$ (see Definition \ref{de.equiv}).

$ii)$ When $\uuu^{( 0 ) , II} \neq \emptyset$, the family of quasimodes $(\varphi_{\m , h})_{\m \in \uuu^{( 0 )}}$ is not quasi-orthonormal and thus does not satisfy the $i)$ of Proposition \ref{a67} (see indeed Remarks \ref{g16} and \ref{g17}).
\end{remark}

\begin{proof}[Proof of Proposition \ref{b19}]
Thanks to Remark \ref{g17} and $\uuu^{( 0 ) , II} = \emptyset$, the assumptions of Proposition \ref{a67} are satisfied.

Let us now prove \eqref{b27} and take $\m , \m^{\prime} \in \uuu^{( 0 )}$. When $\m^{\prime} = \m$, this is exactly $i)$ of Proposition \ref{a67} and we can hence assume that $\m^{\prime} \neq \m$.

Assume moreover that ${\bf j} ( \m ) \cap {\bf j} ( \m^{\prime} ) = \emptyset$. If $\bsigma ( \m ) > \bsigma ( \m^{\prime} )$, then $\varphi_{\m} = C_{\m,h} e^{- ( f - f ( \m ) ) / h}$ on $\supp \varphi_{\m^{\prime}}$ for some constant $C_{\m,h}$. Hence $P \varphi_{\m} = 0$ on $\supp \varphi_{\m^{\prime}}$, which implies $\< P \varphi_{\m} , \varphi_{\m^{\prime}} \> = 0$. If $\bsigma ( \m ) < \bsigma ( \m^{\prime} )$, the same argument works since $P^{*} ( e^{- f / h} ) = 0$. Assume now that $\bsigma ( \m ) = \bsigma ( \m^{\prime} )$. According to $iv)$ in Lemma \ref{a65}, it follows
that $\supp ( \varphi_{\m } ) \cap \supp ( \varphi_{\m^{\prime} } ) = \emptyset$
and then that  $\< P \varphi_{\m} , \varphi_{\m^{\prime}} \> = 0$. Summing up, if ${\bf j} ( \m ) \cap {\bf j} ( \m^{\prime} ) = \emptyset$, we always have $\< P \varphi_{\m} , \varphi_{\m^{\prime}} \> = 0$, which is precisely \eqref{b27} in that case.

Assume eventually that ${\bf j} ( \m ) \cap {\bf j} ( \m^{\prime} ) \neq \emptyset$ and then $\bsigma ( \m ) = \bsigma ( \m^{\prime} )$. As in the proof of Proposition \ref{a67}, we denote $\psi_{\m} = \widetilde{v}_{\m} \widetilde{\psi}_{\m}$ with $\widetilde{v}_{\m} = \theta_{\m} ( 1 + v_{\m} )$ and $\widetilde{\psi}_{\m} = e^{- ( f - f ( \m ) ) / h}$. Using the identity $\widetilde{\psi}_{\m} \nabla \widetilde{\psi}_{\m^{\prime}} = \widetilde{\psi}_{\m^{\prime}} \nabla \widetilde{\psi}_{\m}$ and working as in the proof of Proposition \ref{a67}, we get 
\begin{equation} \label{b37}
\< ( P_{2} + P_{0} ) \psi_{\m} , \psi_{\m^{\prime}} \> = h^{2} \big\< \widetilde{\psi}_{\m} A \nabla \widetilde{v}_{\m} , \widetilde{\psi}_{\m^{\prime}} A \nabla \widetilde{v}_{\m^{\prime}} \big\> .
\end{equation}
Observe now that
$\theta_{\m}=\theta_{\m^{\prime}}$ (see \eqref{a62}) and that
\begin{equation} \label{g26}
\supp \theta_{\m} ( v_{\m} + 1 ) ( v_{\m^{\prime}} + 1 ) \subset \bigcup_{\s \in{\bf j} ( \m ) \cap {\bf j} ( \m^{\prime} )} \ccc_{\s , 3 \tau_{0} , 3 \delta_{0}} ,
\end{equation}
from \eqref{a60}. Hence, since
$\ell_{\s , \m^{\prime}} = - \ell_{\s , \m}$ (see Lemma \ref{g20}) and then $v_{\m^{\prime}} = - v_{\m}$ on the support of $\theta_{\m} (v_{\m} +1)(v_{\m^{\prime}}+1)$ (see \eqref{a61}), we have
\begin{align*}
\< P_{1} \psi_{\m} , \psi_{\m^{\prime}} \> &= \big\< P_{1} ( \theta_{\m} (v_{\m}+1) \widetilde{\psi}_{\m}) , \theta_{\m} ( - v_{\m} + 1 ) \widetilde{\psi}_{\m^{\prime}} \big\> .
\end{align*}
Since in addition $P_{1}$ is formally anti-adjoint, $P_{1} ( e^{-f/h} ) = 0$, and $f - \delta_{0} > \bsigma ( \m ) = \bsigma ( \m^{\prime} )$ on the compact $\supp ( \nabla \theta_{\m} )$ (see \eqref{a62}),
\begin{align}
\< P_{1} \psi_{\m} , \psi_{\m^{\prime}} \> &= -  \big\< P_{1} ( \theta_{\m} v_{\m} \widetilde{\psi}_{\m}) , \theta_{\m} v_{\m} \widetilde{\psi}_{\m^{\prime}} \big\> + \ooo \big( e^{- ( S ( \m ) + S ( \m^{\prime} ) + 2\delta_{0} ) / h} \big)   \nonumber \\
&= -e^{( f ( \m^{\prime} ) - f ( \m ) ) / h}\big\< P_{1} ( \theta_{\m} v_{\m} \widetilde{\psi}_{\m}) , \theta_{\m} v_{\m} \widetilde{\psi}_{\m} \big\> +\ooo \big( e^{- ( S ( \m ) + S ( \m^{\prime} ) + 2\delta_{0} ) / h} \big)
 \nonumber\\
&= \ooo \big( e^{- ( S ( \m ) + S ( \m^{\prime} ) + 2\delta_{0} ) / h} \big). \label{b38}
\end{align}
Combining \eqref{b37} and \eqref{b38}, we get
\begin{equation} \label{b39}
\< P \psi_{\m} , \psi_{\m^{\prime}} \> = h^{2} \big\< \widetilde{\psi}_{\m} A \nabla \widetilde{v}_{\m} , \widetilde{\psi}_{\m^{\prime}} \nabla \widetilde{v}_{\m^{\prime}} \big\> + \ooo \big( e^{- ( S ( \m ) + S ( \m^{\prime} ) + 2\delta_{0} ) / h} \big) .
\end{equation}
Using \eqref{v69}, \eqref{g26}, and $f - \delta_{0} > \bsigma ( \m ) = \bsigma ( \m^{\prime} )$ on $\supp ( \nabla \theta_{\m} )$, this implies
\begin{align*}
\< P \psi_{\m} , \psi_{\m^{\prime}} \> ={}& h^{2} \sum_{\s \in {\bf j} ( \m )\cap {\bf j} ( \m^{\prime} )} \frac{1}{C^{2}_{\s , h} } \int_{\ccc_{\s , 3 \tau_{0} , 3 \delta_{0}}} \theta^{2}_{\m} \zeta ( \ell_{\s , \m} / \tau_{0} ) \zeta ( \ell_{\s , \m^{\prime}} / \tau_{0} )   \\
&\times A \nabla \ell_{\s , \m} \cdot \nabla \ell_{\s, \m^{\prime}} e^{- \big( 2 f + \frac{\ell_{\s , \m}^{2}}{2} + \frac{\ell_{\s , \m^{\prime}}^{2}}{2} - f ( \m ) - f ( \m^{\prime} ) \big) / h} d x + \ooo \big( e^{-( S ( \m ) + S ( \m^{\prime} ) + 2\delta_{0} ) / h} \big) .
\end{align*}
Since  $\ell_{\s , \m^{\prime}} = - \ell_{\s , \m}$ on $\ccc_{\s , 3 \tau_{0} , 3 \delta_{0}}$ and $\zeta$ is even, it follows that
\begin{align*}
\< P \psi_{\m} , \psi_{\m^{\prime}} \> ={}& h^{2} \sum_{\s \in {\bf j} ( \m ) \cap {\bf j} ( \m^{\prime} )} \frac{- 1}{C_{\s , h}^{2}} \int \theta^{2}_{\m} \zeta ( \ell_{\s , \m} / \tau_{0} )^{2} A \nabla \ell_{\s , \m} \cdot \nabla \ell_{\s , \m} e^{- \big( 2 f + \ell_{\s , \m}^{2} - f ( \m ) - f ( \m^{\prime} ) \big) / h} d x \\
&+ \ooo \big( e^{-( S ( \m ) + S ( \m^{\prime} ) + 2\delta_{0} ) / h} \big) .
\end{align*}
Since ${\bf j} ( \m ) \cap {\bf j} ( \m^{\prime} ) \neq \emptyset$, this quantity can be computed by the Laplace method as in the proof of Proposition \ref{a67}. We obtain 
\begin{equation} \label{b40}
\< P \psi_{\m} , \psi_{\m^{\prime}} \> \in \eee_{cl} \Big( \frac{-2 h}{\pi} ( \pi h )^{\frac{d}{2}} e^{- (S ( \m ) +S(\m^{\prime}))/ h} \sum_{\s \in {\bf j} ( \m )\cap {\bf j} ( \m^{\prime} )} \vert \mu ( \s ) \vert D_{\s}^{- 1} \Big).
\end{equation}
The relation \eqref{b27} follows, using also \eqref{a68}.

Eventually, \eqref{b22} follows from \eqref{b37}, \eqref{b38}, \eqref{b39}, and \eqref{b40} when ${\bf j} ( \m ) \cap {\bf j} ( \m^{\prime} ) \neq \emptyset$, and, when ${\bf j} ( \m ) \cap {\bf j} ( \m^{\prime} ) = \emptyset$,  from 
$\< P \varphi_{\m} , \varphi_{\m^{\prime}} \> = 0= \< P^{\sharp} \varphi_{\m} , \varphi_{\m^{\prime}} \> $,
 where the last equality can be proved
as was the first one in the third paragraph, using  $( P_{2} + P_{0} ) ( e^{- f / h} ) =( P_{2} + P_{0} )^{*}( e^{- f / h} )= 0$.
\end{proof}

\subsection{Graded structure of the interaction matrix}

Suppose from now on that the minima $\m \in \uuu^{( 0 )}$ of $f$ are labeled as in \eqref{b32}. Let $( \widetilde{e}_{j} )_{j = 1 , \ldots , n_{0} }$ denote the basis of $\Ran \Pi_{h}$ obtained from the sequence $( \Pi_{h} \varphi_{\m_{n_{0} - j + 1}} )_{j = 1 , \ldots , n_{0}}$ by 
the Gram--Schmidt process, and let $e_{j} = \widetilde{e}_{n_{0} - j + 1}$ for $j = 1 , \ldots , n_{0}$.
We recall that, taking the $\varepsilon_{*}$ of Proposition \ref{a22},
\begin{equation*}
\Pi_{h} = \frac{1}{2 i \pi} \int_{\partial D ( 0 , \varepsilon_{*} h / 2 )} ( z - P )^{- 1} d z
\end{equation*}
denotes the spectral projector associated with the  
$n_0$ eigenvalues of $P$ 
of order $\ooo(h^{1+\alpha})$, $\alpha>0$.

Let $\Upsilon = ( \upsilon_{i , j} )_{i , j = 1 , \ldots , n_{0} - 1}$ denote the matrix with coefficients
\begin{equation} \label{b26}
\upsilon_{i , j} = \< P e_{i} , e_{j} \> .
\end{equation}
Introduce also the matrix $\Upsilon^{\sharp} = ( \upsilon^{\sharp}_{i , j})_{i , j = 1 , \ldots , n_{0} - 1}$ defined by 
\begin{equation} \label{b24}
\upsilon^{\sharp}_{i , j} = \< P^{\sharp} \varphi_{\m_{i}} , \varphi_{\m_{j}} \> =
\<(P_{2} + P_{0})\varphi_{\m_{i}} , \varphi_{\m_{j}} \> .
\end{equation}
Note that, in these definitions, we do not consider the contribution of the vectors $e_{n_{0}}$ and $\varphi_{\m_{n_{0}}}$ which are collinear to $e^{- f / h}$ and then belong to the kernels of $P$, $P^{*}$, and $P^{\sharp}$. If we had added these latter vectors in the definitions of $\Upsilon$ and of $\Upsilon^{\sharp}$, the last line and column of these matrices would have consisted of zeros. In particular, $\sigma ( P_{\vert \Ran \Pi_{h}} ) = \sigma ( \Upsilon ) \cup \{ 0 \}$.

In order to compute the spectrum of the matrix $\Upsilon$ and then  the spectrum of $P_{\vert \Ran \Pi_{h}}$, we recall some tools from \cite[Section 5B]{Mi19}. In the sequel, we denote by $\Sr^{+} ( E )$ the set of symmetric positive definite matrices on a vector space $E$. We will denote by $\Sr^{+}_{cl} ( E )$ the set of $h$-depending matrices $M ( h ) \in \Sr^{+} ( E )$ admitting a classical expansion $M ( h ) \sim \sum_{j} h^{j} M_{j}$ with $M_{0} \in \Sr^{+} ( E )$. We will sometimes forget $E$ and write for short $\Sr^{+}$, $\Sr^{+}_{cl}$.

\begin{defin}\sl \label{b14}
Let $\Er = ( E_{j} )_{j = 1 , \ldots , p}$ be a sequence of finite dimensional vector spaces, $E = \oplus_{j = 1 , \ldots , p} E_{j}$, and let $\tau = ( \tau_{2} , \ldots , \tau_{p} ) \in ( \R_{+}^{*} )^{p - 1}$. Suppose that $\tau \mapsto \mmm ( \tau )$ is a map from $( \R_{+}^{*} )^{p - 1}$ into the set of matrices $\Mr(E)$. 
\begin{itemize}
\item[$\star$] We say that $\mmm ( \tau )$ is an $( \Er , \tau )$-graded symmetric matrix if there exists $M \in \Sr^{+} ( E )$ independent of $\tau$ such that 
\begin{equation} \label{b15}
\mmm ( \tau ) = \Omega ( \tau ) M \Omega ( \tau ) ,
\end{equation}
with $\Omega ( \tau ) = \diag ( \varepsilon_{1} ( \tau ) \Id_{E_{1}} , \ldots , \varepsilon_{p} ( \tau ) \Id_{E_{p}} )$, $\varepsilon_{1} ( \tau ) = 1$ and $\varepsilon_{j} ( \tau ) = \prod_{k = 2}^{j} \tau_{k}$ for $j \geq 2$. 

\item[$\star$] We say that a family of $( \Er , \tau )$-graded symmetric matrices $\mmm_{h} ( \tau )$, $h \in ] 0 , h_{0} ]$ is classical if one has $\mmm_{h} ( \tau ) = \Omega ( \tau ) M_{h} \Omega ( \tau )$ for some matrix $M_{h} \in \Sr^{+}_{cl} ( E )$.

\item[$\star$] We say that $\mmm_{h} ( \tau )$ is an $( \Er , \tau )$-graded almost symmetric matrix if there exists $M_{h} \in \Mr ( E )$ such that 
\begin{equation} \label{b16}
\mmm_{h} ( \tau ) = \Omega ( \tau ) M_{h} \Omega ( \tau ), \quad M_{h} + M_{h}^{*}\in \Sr^{+} ( E ) \quad \text{and} \quad M_{h} - M_{h}^{*} = \ooo ( h^\infty ) .
\end{equation}

\item[$\star$] We say that a family of $( \Er , \tau )$-graded almost symmetric matrices $\mmm_{h} ( \tau )$, $h \in ] 0 , h_{0} ]$ is classical if the matrix $M_{h}$ in \eqref{b16} satisfies $M_{h} + M_{h}^{*} \in \Sr^{+}_{cl} ( E )$.
\end{itemize}
Throughout, we denote by $GS ( \Er , \tau )$ (resp. $GS{cl} ( \Er , \tau )$) the set of (resp. classical) $( \Er , \tau )$-graded symmetric matrices, and by $GAS ( \Er , \tau )$ (resp. $GAS_{cl} ( \Er , \tau )$) the set of (resp. classical) $( \Er , \tau )$-graded almost symmetric matrices.

For any $\mmm \in \Mr ( E )$, we denote $\mmm^{s} = \frac{1}{2} ( M + M^{*} )$ its real part. Obviously, the real part of a matrix in $GAS ( \Er , \tau )$ (resp. $GAS_{cl} ( \Er , \tau )$) belongs to $GS ( \Er , \tau )$ (resp. $GS_{cl} ( \Er , \tau )$).
\end{defin}

Let $\{ S_{1} \leq \cdots \leq S_{p} \}$ denote the set $\{ S ( \m_{j} ) ; \ j = 1 , \ldots , n_{0} - 1 \}$ and, for all $k = 1 , \ldots , p$, let $E_{k}$ denote the vector space generated by $\{ e_{r} ; \ S ( \m_{r} ) = S_{k} \}$. We also set $\tau_{k} = e^{- ( S_{k} - S_{k - 1} ) / h}$ for any $k = 2 , \ldots , p$.

\begin{proposition}\sl \label{b17}
Let $\Er = ( E_{1} , \ldots , E_{p} )$ and $\tau = ( \tau_{2} , \ldots , \tau_{p} )$ be as above, then $h^{- 1} e^{2 S_{1} / h} \Upsilon^{\sharp}$ belongs to $GS_{cl} ( \Er , \tau )$ and $h^{- 1} e^{2 S_{1} / h} \Upsilon$ belongs to $GAS_{cl} ( \Er , \tau )$. Moreover, one has
\begin{equation} \label{b18}
e^{2 S_{1} / h} \Upsilon = e^{2 S_{1} / h} \Upsilon^{\sharp} + \Omega ( \tau ) \ooo ( h^{\infty} ) \Omega ( \tau ) .
\end{equation}
\end{proposition}

\begin{proof}
Mimicking the proof of \cite[Proposition 4.12]{LePMi20}, we get
\begin{equation}
\upsilon_{i , j} = \< P \varphi_{\m_{i}} , \varphi_{\m_{j}} \> + \ooo \big( h^{\infty} \< P \varphi_{\m_{i}} , \varphi_{\m_{i}} \>^{1 / 2} \< P \varphi_{\m_{j}} , \varphi_{\m_{j}} \>^{1 / 2} \big) .
\end{equation}
Then, one deduces from Proposition \ref{b19} the existence of some $\alpha>0$ such that 
\begin{align*}
e^{2 S_{1} / h} \upsilon_{i , j} &= e^{2 S_{1} / h} \big( \upsilon_{i , j}^{\sharp} \big( 1 + \ooo ( e^{- \alpha / h} ) \big) + \ooo \big( h^{\infty} \< P \varphi_{\m_{i}} , \varphi_{\m_{i}} \>^{1 / 2} \< P \varphi_{\m_{j}} , \varphi_{\m_{j}} \>^{1 / 2} \big) \big)    \\
&= e^{2 S_{1} / h} \upsilon^{\sharp}_{i , j} + \ooo ( e^{( 2 S_{1} - S_{i} - S_{j} ) / h} h^{\infty} ) .
\end{align*}
This establishes \eqref{b18}.

It just remains to show that $h^{- 1} e^{2 S_{1} / h} \Upsilon^{\sharp} \in GS_{cl} ( \Er , \tau )$. 
Indeed,  the fact that $h^{- 1} e^{2 S_{1} / h} \Upsilon \in GAS_{cl} ( \Er , \tau )$ will then follow from \eqref{b18}.

Using \eqref{b27}, \eqref{b22} and the fact that $P^{\sharp}$ is symmetric, we deduce that $h^{- 1} e^{2 S_{1} / h} \Upsilon^{\sharp}$ is a graded matrix, say $h^{- 1} e^{2 S_{1} / h} \Upsilon^{\sharp} = \Omega ( \tau ) M_{h} \Omega ( \tau )$, where $M_{h}$ is a symmetric matrix having a classical expansion with leading term
\begin{equation*}
( M_{0} )_{i , j} = ( - 1 )^{\delta_{\m_{i} , \m_{j}} - 1} \sum_{\s \in {\bf j} ( \m_{i} ) \cap {\bf j} ( \m_{j} )} \frac{\vert \mu ( \s ) \vert}{2 \pi} \frac{\sqrt{D_{\m_{i}} D_{\m_{j}}}}{D_{\s}} .
\end{equation*}
In order to show that $M_{0}$ is positive definite, it is sufficient to show that it has the form $M_{0} = L^{*} L$, where $L$ is an injective matrix from $\R^{n_{0} - 1}$ to $\R^{\sharp \vvv^{( 1 )}}$. To this end, let us
define, for every  $\s_{k} \in \vvv^{( 1 )}$, $G ( \s_{k} ) := \{ \m \in \uuu^{( 0 )}\setminus\{\underline\m\} ; \ \s_{k} \in {\bf j} ( \m ) \}$. 
For any $\s_{k} \in \vvv^{( 1 )}$,
the set $G ( \s_{k} )$ is non-empty. Moreover,
from the structure of a Morse function near a (separating) saddle point,
it has at most two elements and
has only  one element $\m$ if and only if $\s_{k}\in \partial \widehat E(\m)$. If there is only one minimum $\m_{i} \in \uuu^{( 0 )}\setminus\{\underline\m\}$ in $G ( \s_{k} )$, we set
\begin{equation*}
L_{k , i} = \sqrt{\frac{\vert \mu ( \s ) \vert}{2 \pi} \frac{D_{\m_{i}}}{D_{\s_{k}}}}.
\end{equation*}
If there are two minima $\m_{i} \neq \m_{j} \in \uuu^{( 0 )}\setminus\{\underline\m\}$ in $G ( \s_{k} )$, we define
\begin{equation*}
L_{k , i} = \sqrt{\frac{\vert \mu ( \s ) \vert}{2 \pi} \frac{D_{\m_{i}}}{D_{\s_{k}}}} \qquad \text{and} \qquad L_{k , j} = - \sqrt{\frac{\vert \mu ( \s ) \vert}{2 \pi} \frac{D_{\m_{j}}}{D_{\s_{k}}}} .
\end{equation*}
We do not specify which index ($i$ or $j$) receives a ``$-$'', since this choice is irrelevant in the sequel. The other coefficients of $L$ are set to zero
\begin{equation*}
L_{k , i} = 0 \text{ whenever } \m_{i} \notin G ( \s_{k} ), \text{ i.e. 
whenever } \s_{k} \notin{\bf j} ( \m_{i} ) .
\end{equation*}
A direct computation gives $M_{0} = L^{*} L$. 
Moreover, it follows from  \cite[Lemma 5.1]{Mi19} that $L$ is injective.
Let us briefly explain the argument of \cite[Lemma 5.1]{Mi19}.
Let $X = (X_{\m})_{\m\in  \uuu^{( 0 )}\setminus\{\underline\m\}}\in \R^{n_{0} - 1}$ be such that $LX=0$.
For any $\s_{k}\in \vvv^{( 1 )}$ such that $G(\s_{k})$ contains one unique element $\m(\s_{k})\in \uuu^{( 0 )}\setminus\{\underline\m\}$,
it then holds $X_{\m(\s_{k})}=0$. It  follows from the structure of $L$  that
$X_\m=0$ for every $\m\in\Cl ( \m(\s_{k}) )$. But, for every $\m\in \uuu^{( 0 )}\setminus\{\underline\m\}$, there
exists $\s_{k} \in \vvv^{( 1 )}$ and  $\m(\s_{k})\in\Cl ( \m)$ 
such that $G(\s_{k})=\{\m(\s_{k})\}$. It thus holds
$X_\m=0$ for every $\m\in \uuu^{( 0 )}\setminus\{\underline\m\}$ and $L$ is injective.
Summing up, $M_{0}$ is positive definite and thus $h^{- 1} e^{2 S_{1} / h} \Upsilon^{\sharp} \in GS_{cl} ( \Er , \tau )$, which concludes the proof of Proposition \ref{b17}. 
\end{proof}

\subsection{The spectrum of GAS matrices}

The results stated here are variants of those of \cite[Section 5]{Mi19} and of \cite[Appendix]{LePMi20}. For $p \in \N^{*}$, a finite dimensional vector space $E = E_{1} \oplus \cdots \oplus E_{p}$, and $j \in \{ 1 , \dots , p \}$, let us write a general matrix $M \in \Mr ( E )$ by blocks
\begin{equation} \label{b9}
M = \left( \begin{array}{cc}
A & B \\
C & D
\end{array} \right) : ( E_{1} \oplus \cdots \oplus E_{j - 1} ) \oplus ( E_{j} \oplus \cdots \oplus E_{p} ) \longrightarrow ( E_{1} \oplus \cdots \oplus E_{j - 1} ) \oplus ( E_{j} \oplus \cdots \oplus E_{p} ) .
\end{equation}
If $A \in \Mr ( E_{1} \oplus \cdots \oplus E_{j-1} )$ is invertible, the Schur matrix of $M$ (with respect to the vector space $E_{1} \oplus \cdots \oplus E_{j - 1}$) is the matrix on $E_{j} \oplus \cdots \oplus E_{p}$ defined by
\begin{equation*}
\rrr_{j} ( M ) = D - C A^{- 1} B ,
\end{equation*}
where by convention $\rrr_{1} ( M ) = M$. By the Schur complement method, $M$ is invertible if and only if $\rrr_{j} ( M )$ is invertible. Moreover, the map $\rrr_{j}$ sends $\Sr^{+} ( E )$ into $\Sr^{+} ( E_{j} \oplus \cdots \oplus E_{p} )$ and
$\S r^{+}_{cl} ( E )$ into $\Sr^{+}_{cl} ( E_{j} \oplus \cdots \oplus E_{p} )$. We will also denote by $\jjj : \Mr ( \oplus_{k = j , \ldots , p } E_{k} ) \to \Mr ( E_{j})$ the restriction map to the first vector space $E_{j}$ of $\oplus_{k = j , \ldots , p } E_{k}$. More precisely, with the notations of \eqref{b9}, we will write $\jjj ( M ) = A$ when $j=1$. Of course, the map $\jjj$ depends on $j \in \{ 1 , \dots , p \}$, but we omit this dependence since the set on which $\jjj$ is acting will be obvious in the sequel.

Let $E$ be a finite dimensional vector space and $M_{h} \in \Sr^{+}_{cl} ( E )$. From a standard result of perturbation theory of symmetric matrices (see  \cite[Theorem 6.1 in Chapter II]{Ka66}), the eigenvalues of $M_{h}$, after an appropriate labeling that we assume in the sequel, have an asymptotic expansion in power of $h$ with a non-zero leading term. This justifies the following definition.

\begin{defin}\sl
For $M_{h} \in \Sr^{+}_{cl} ( E )$, we denote by $\sigma_{\equiv} ( M_{h} )$ the set of asymptotic expansions (that is formal power series in $h$) of the eigenvalues of $M_{h}$. Moreover, $m_{\equiv} ( \Lambda , M_{h} )$ is defined as the multiplicity of $\Lambda \in \sigma_{\equiv} ( M_{h} )$.
\end{defin}

By an abuse of notation, if $\lambda \in \sigma ( M_{h} )$ has an asymptotic expansion $\Lambda \in \sigma_{\equiv} ( M_{h} )$, we will sometimes write $m_{\equiv} ( \lambda , M_{h} )$ instead of $m_{\equiv} ( \Lambda , M_{h} )$. Roughly speaking, $m_{\equiv} ( \lambda , M_{h} )$ can be seen as the multiplicity modulo $\ooo ( h^{\infty} )$ of the eigenvalue $\lambda \in \sigma ( M_{h} )$. Note that if $\lambda , \nu \in \sigma ( M_{h} )$ do not have the same asymptotic expansion, there exists $K_{0} > 0$ such that $\vert \lambda - \nu \vert \geq h^{K_{0}}$ for $h > 0$ small enough.

\begin{theorem}\sl \label{b10}
Let $\mmm_{h} = \Omega ( \tau ) M_{h} \Omega  ( \tau ) \in GAS_{cl} ( \Er , \tau )$ and assume that $\tau_{j} = \tau_{j} ( h ) = \ooo ( h^{\infty} )$ for every $j=2,\dots,p$. Then, we have for $h>0$ small enough,
\begin{equation} \label{b11}
\sigma ( \mmm_{h} ) \subset \bigcup_{j = 1}^{p} \varepsilon_{j} ( \tau )^{2} \big( \sigma \big( \jjj \circ \rrr_{j } ( M_{h}^{s} ) \big) + D ( 0 , \ooo ( h^{\infty} ) ) \big) ,
\end{equation}
where we recall that $M_{h}^{s}=\frac{1}{2} ( M_{h} + M_{h}^{*} )$ denotes the real part of $M_{h}$.
Moreover, for all $j = 1 , \ldots , p$, $K >  0$ large enough and $\lambda \in \sigma ( \jjj \circ \rrr_{j } ( M_{h}^{s} ))$, one has 
\begin{equation} \label{b12}
n ( \mmm_{h} , D_{j , \lambda}^{K} ) = m_{\equiv} \big( \lambda , \jjj \circ \rrr_{j } ( M_{h}^{s} ) \big),
\end{equation}
which does not depend on $K$,
where $n ( \mmm_{h} , D_{j , \lambda}^{K} )$ is the number of eigenvalues of $\mmm_{h}$ in $D_{j , \lambda}^{K} = D ( \varepsilon_{j}^{2} \lambda , \varepsilon_{j}^{2} h^{K} )$ counted with their multiplicity. Eventually, for all $K > 0$ large enough, there exists $C > 0$ such that
\begin{equation} \label{b13}
\forall z \in \C \setminus \bigcup_{j , \lambda} D_{j , \lambda}^{K} , \qquad \Vert ( \mmm_{h} - z )^{- 1} \Vert \leq C \dist ( z , \sigma ( \mmm_{h} ) )^{- 1} ,
\end{equation}
for $h > 0$ small enough.
\end{theorem}

\begin{proof}
For $j = 1 , \ldots , p$, we assume that the spectral parameter $z$ belongs to the ring $\ccc_{j} = \{ z \in \C ; \ \varepsilon_{j}^{2} / R < \vert z \vert < R \varepsilon_{j}^{2} \}$,
where $R>1$ is large enough so that $\bigcup_{\lambda \in \sigma ( \jjj \circ \rrr_{j} ( M_{h}^{s} ) )} \overline{D_{j , \lambda}^{K}}\subset \ccc_{j}$. We write $\mmm_{h}$ as
\begin{equation*}
\mmm_{h} = \Omega ( \tau ) M_{h} \Omega ( \tau ) = 
\left( \begin{array}{cc}
L_{+} &0 \\
0 &L_{-}
\end{array} \right)
\left( \begin{array}{cc}
A &B \\
C &D
\end{array} \right)
\left( \begin{array}{cc}
L_{+} &0 \\
0 &L_{-}
\end{array} \right)  \\
= \left( \begin{array}{cc}
L_{+} A L_{+} &L_{+} B L_{-} \\
L_{-} C L_{+} &L_{-} D L_{-}
\end{array} \right) ,
\end{equation*}
with $L_{+} = \diag ( \varepsilon_{1} \Id_{E_{1}} , \ldots , \varepsilon_{j - 1} \Id_{E_{j - 1}} )$ and $L_{-} = \diag ( \varepsilon_{j} \Id_{E_{j}} , \ldots , \varepsilon_{p} \Id_{E_{p}} )$. For $z \in \ccc_{j}$, we have $\eee : = L_{+} A L_{+} - z = L_{+} ( A - L_{+}^{- 1} z L_{+}^{- 1} ) L_{+}$ with $\Vert L_{+}^{- 1} z L_{+}^{- 1} \Vert = \ooo ( \varepsilon_{j}^{2} \varepsilon_{j - 1}^{- 2} ) = \ooo ( \tau_{j}^{2} ) = \ooo ( h^{\infty} )$. Since $A \in \Sr^{+}_{cl} ( E_{1} \oplus \cdots \oplus E_{j - 1} )$ modulo $\ooo ( h^{\infty} )$, the matrix $\eee$ is invertible and
\begin{equation} \label{b1}
\eee^{- 1} = L_{+}^{- 1} ( A^{- 1} + \ooo ( h^{\infty} ) ) L_{+}^{- 1} = \ooo ( \varepsilon_{j - 1}^{- 2} ) ,
\end{equation}
uniformly for $z \in \ccc_{j}$. By the Schur complement method, $\mmm_{h} - z$ is invertible if and only if
\begin{equation} \label{b6}
\fff = L_{-} D L_{-} - z - L_{-} C L_{+} \eee^{- 1} L_{+} B L_{-} \ \ \ \text{is invertible.}
\end{equation}
 In that case,
\begin{equation} \label{b2}
( \mmm_{h} - z )^{- 1} = \left( \begin{array}{cc}
\eee^{- 1} + \eee^{- 1} L_{+} B L_{-} \fff^{- 1} L_{-} C L_{+} \eee^{- 1} &- \eee^{- 1} L_{+} B L_{-} \fff^{- 1} \\
- \fff^{- 1} L_{-} C L_{+} \eee^{- 1} &\fff^{- 1}
\end{array} \right) .
\end{equation}
From \eqref{b1}, the matrix $\fff$ can be written, uniformly for $z \in \ccc_{j}$,
\begin{equation*}
\fff = L_{-} ( D - C A^{- 1} B ) L_{-} - z + \ooo ( \varepsilon_{j}^{2} h^{\infty} ) = \varepsilon_{j}^{2} \jjj ( D - C A^{- 1} B ) \oplus 0 - z + \ooo ( \varepsilon_{j}^{2} h^{\infty} ) ,
\end{equation*}
with $\jjj ( D - C A^{- 1} B ) : E_{j} \rightarrow E_{j}$ and the shortcut
\begin{equation*}
T \oplus 0 =
\left( \begin{array}{cc}
T & 0 \\
0 & 0
\end{array} \right) : E_{j} \oplus \cdots \oplus E_{p} \rightarrow E_{j} \oplus \cdots \oplus E_{p} .
\end{equation*}
 Thus,
\begin{align} 
\nonumber
\fff &= \varepsilon_{j}^{2} \jjj \circ \rrr_{j} ( M_{h} ) \oplus 0 - z + \ooo ( \varepsilon_{j}^{2} h^{\infty} )\\
\label{b7}
 &= \varepsilon_{j}^{2} \jjj \circ \rrr_{j} ( M_{h}^{s} ) \oplus 0 - z + \ooo ( \varepsilon_{j}^{2} h^{\infty} ) .
\end{align}

We obtain an upper bound on the resolvent of $\mmm_{h}$ away from its expected spectrum. Let $K_{0} > 0$ be such that $\vert \lambda - \nu \vert \geq h^{K_{0}}$ for all $\lambda , \nu \in \sigma ( \jjj \circ \rrr_{j} ( M_{h}^{s} ) )$ having different asymptotic expansions and let $K > K_{0}$. Let us define $\widetilde{\ccc}_{j} = \ccc_{j} \setminus \bigcup_{\lambda \in \sigma ( \jjj \circ \rrr_{j} ( M_{h}^{s} ) )} D_{j , \lambda}^{K}$. Since $\jjj \circ \rrr_{j} ( M_{h}^{s} )$ is symmetric, we have $( \varepsilon_{j}^{2} \jjj \circ \rrr_{j} ( M_{h}^{s} ) - z )^{- 1} = \ooo ( \varepsilon_{j}^{- 2} h^{- K} )$ for $z \in \widetilde{\ccc}_{j}$,
\begin{align*}
\fff &= \big( \varepsilon_{j}^{2} \jjj \circ \rrr_{j} ( M_{h}^{s} ) \oplus 0 - z \big) \big( 1 + \big( \varepsilon_{j}^{2} \jjj \circ \rrr_{j} ( M_{h}^{s} ) \oplus 0 - z \big)^{- 1} \ooo ( \varepsilon_{j}^{2} h^{\infty} ) \big) \\
&= \big( \varepsilon_{j}^{2} \jjj \circ \rrr_{j} ( M_{h}^{s} ) \oplus 0 - z \big) ( 1 + \ooo ( h^{\infty} ) ) ,
\end{align*}
and then $\fff$ is invertible and
\begin{equation} \label{b3}
\fff^{- 1} = \big( \varepsilon_{j}^{2} \jjj \circ \rrr_{j} ( M_{h}^{s} ) \oplus 0 - z \big)^{- 1} + \ooo ( \varepsilon_{j}^{- 2} h^{\infty} ) = \ooo ( \varepsilon_{j}^{- 2} h^{- K} ).
\end{equation}
Combining the first equality in \eqref{b1}, \eqref{b2}, \eqref{b3} and using $\varepsilon_{k} = \ooo ( \varepsilon_{\ell} h^{\infty} )$ for $k > \ell$, we get, for $z \in \widetilde{\ccc}_{j}$,
\begin{equation} \label{b4}
( \mmm_{h} - z )^{- 1} = \left( \begin{array}{ccc}
0 & 0 & 0 \\
0 & \big( \varepsilon_{j}^{2} \jjj \circ \rrr_{j} ( M_{h}^{s} ) - z \big)^{- 1} & 0 \\
0 & 0 & - z^{- 1}
\end{array} \right)
+ \ooo ( \varepsilon_{j}^{- 2} h^{\infty} ) ,
\end{equation}
which implies that there exists $C > 0$ such that for $z \in \widetilde{\ccc}_{j}$
\begin{equation} \label{b5}
\Vert ( \mmm_{h} - z )^{- 1} \Vert \leq C \dist \big( z , \varepsilon_{j}^{2} \sigma \big( \jjj \circ \rrr_{j} ( M_{h}^{s} ) \big) \big)^{- 1} .
\end{equation}

We now compute the eigenvalues of $\mmm_{h}$. For $\lambda \in \sigma ( \jjj \circ \rrr_{j} ( M_{h}^{s} ) )$, we introduce the spectral projectors
\begin{equation*}
\Pi_{\lambda} = - \frac{1}{2 i \pi} \int_{\partial D_{j , \lambda}^{K}} ( \mmm_{h} - z )^{- 1} d z  \quad \text{and} \quad \pi_{\lambda} = - \frac{1}{2 i \pi} \int_{\partial D ( \lambda , h^{K} )} ( \jjj \circ \rrr_{j} ( M_{h}^{s} ) - z )^{- 1} d z .
\end{equation*}
Using that the circumference of $\partial D_{j , \lambda}^{K}$ is $2 \pi \varepsilon_{j}^{2} h^{K}$, \eqref{b4} implies
\begin{equation*}
\Pi_{\lambda} = \left( \begin{array}{ccc}
0 & 0 & 0 \\
0 & \pi_{\lambda} & 0 \\
0 & 0 & 0
\end{array} \right)
+ \ooo ( h^{\infty} ) .
\end{equation*}
Since $E$ is a finite dimensional space and the rank of a projector is an integer equal to its trace, we deduce that
\begin{equation*}
\rk ( \Pi_{\lambda} ) = \tr ( \Pi_{\lambda} ) = \tr ( \pi_{\lambda} ) + \ooo ( h^{\infty} ) = \rk ( \pi_{\lambda} ).
\end{equation*}
Thus, the number of eigenvalues of $\mmm_{h}$ in $D_{j , \lambda}^{K}$ counted with their multiplicity is equal to $m_{\equiv} ( \lambda , \jjj \circ \rrr^{j - 1} ( M_{h}^{s} ) ) $. Since
\begin{equation*}
\sum_{j = 1}^{p}\sum_{\lambda \in \sigma_{\equiv} ( \jjj \circ \rrr_{j} ( M_{h}^{s} ) )} m_{\equiv} \big( \lambda , \jjj \circ \rrr_{j} ( M_{h}^{s} ) \big) = \sum_{j = 1}^{p} \dim E_{j} = \dim E ,
\end{equation*}
the relations \eqref{b11} and \eqref{b12} of Theorem \ref{b10} follow.

To finish the proof of Theorem \ref{b10}, we have to obtain the resolvent estimate \eqref{b13}. Since it follows from \eqref{b5} in the $\widetilde{\ccc}_{j}$, $j=1,\dots, p$ , it remains to prove it in $\ddd_{1} = \{ z \in \C ; \ R \varepsilon_{1}^{2} \leq \vert z \vert \}$, $\ddd_{j} = \{ z \in \C ; \ R \varepsilon_{j}^{2} \leq \vert z \vert \leq \varepsilon_{j - 1}^{2} / R \}$ for $2 \leq j \leq p$ and $\ddd_{p + 1} = \{ z \in \C ; \ \vert z \vert \leq \varepsilon_{p}^{2} / R \}$. We only show it in $\ddd_{j}$, when $2 \leq j \leq p$, since the two remaining situations can be treated in the same way. Mimicking the proof of \eqref{b1}, we have $\eee^{- 1} = L_{+}^{- 1} ( A^{- 1} + \ooo ( R^{- 1} ) ) L_{+}^{- 1}$ for $R$ large enough and $z \in \ddd_{j}$. Then, $\fff$ defined in \eqref{b6} satisfies, instead of \eqref{b7},
\begin{equation*}
\fff = \varepsilon_{j}^{2} \jjj \circ \rrr_{j} ( M_{h}^{s} ) \oplus 0 - z + \ooo ( \varepsilon_{j}^{2} R^{- 1} ) .
\end{equation*}
Since $\vert z \vert \geq R \varepsilon_{j}^{2}$, this implies $\Vert \fff^{- 1} \Vert = \ooo ( \vert z \vert^{- 1} )$ and $\Vert ( \mmm_{h} - z )^{- 1} \Vert = \ooo ( \vert z \vert^{- 1} )$ from \eqref{b2}. The last estimate gives \eqref{b13}.
\end{proof}

\subsection{The spectrum of the interaction matrix}

Combining Proposition \ref{b17} and Theorem \ref{b10}, we obtain the following result.

\begin{theorem}\sl \label{b23}
Suppose that the assumptions of Theorem \ref{a15} are satisfied.
Assume also that \eqref{a17}, \eqref{h1},  and \eqref{a24} hold true and that $\uuu^{(0),II}=\emptyset$. Let $\Upsilon^{\sharp}$ be defined by
\eqref{b24}, then
\begin{equation*}
\sigma ( P ) \cap \{ \Re z < \varepsilon_{*} h \} \subset \{ 0 \} \cup \bigcup_{j = 1}^{p} e^{- 2 S_{j} / h} \Big( \sigma \Big( \jjj \circ \rrr_{j } \big( e^{2 S_{j} / h} \Upsilon^{\sharp} \big) \Big) + D ( 0 , \ooo ( h^{\infty} ) ) \Big) .
\end{equation*}
Moreover, for all $j = 1 , \ldots , p$, $K >  0$ large enough and $\lambda \in \sigma ( \jjj \circ \rrr_{j } ( e^{2 S_{j} / h} \Upsilon^{\sharp} ) )$, one has 
\begin{equation*}
n ( P , D_{j , \lambda}^{K} ) = m_{\equiv} \big( \lambda , \jjj \circ \rrr_{j } \big( e^{2 S_{j} / h} \Upsilon^{\sharp} \big) \big) ,
\end{equation*}
where $n ( P , D_{j , \lambda}^{K} )$ is the number of eigenvalues of $P$ in $D_{j , \lambda}^{K} = D ( e^{- 2 S_{j} / h} \lambda , e^{- 2 S_{j} / h} h^{K} )$ counted with their multiplicity.
\end{theorem}

Note here that the matrix $h^{-1} \jjj \circ \rrr_{j } ( e^{2 S_{j} / h} \Upsilon^{\sharp} )$ belongs to $\Sr^{+}_{cl}$,
which almost gives the statement of Theorem \ref{b34}.
In order to obtain an explicit version of Theorem \ref{b34}, let us use the specific structure of 
the matrices $\Upsilon$ and $\Upsilon^{\sharp}$. For any $\alpha \in  \aaa\setminus \{ \Cl ( \underline{\m} )\}$, let $\Upsilon_{\alpha}$ and $\Upsilon^{\sharp}_{\alpha}$ denote the matrices
\begin{equation} \label{b25}
\Upsilon_{\alpha} = \big( \< P e_{\m} , e_{\m^{\prime}} \> \big)_{\m , \m^{\prime} \in \uuu^{( 0 )}_{\alpha}}\quad\text{and}\quad
\Upsilon^{\sharp}_{\alpha} = \big( \< P^{\sharp} \varphi_{\m} , \varphi_{\m^{\prime}} \> \big)_{\m , \m^{\prime} \in \uuu^{( 0 )}_{\alpha}} .
\end{equation}
It follows from \eqref{b22} and \eqref{b21} that $\< P \varphi_{\m} , \varphi_{\m^{\prime}} \> =\< P^{\sharp} \varphi_{\m} , \varphi_{\m^{\prime}} \> = 0$ for all $\m^{\prime} \notin \Cl ( \m )$ and hence the matrices $\Upsilon$
and $\Upsilon^{\sharp}$ defined by \eqref{b26} and \eqref{b24}
are  permutation-similar to the block diagonal matrices $ \diag ( \Upsilon_{\alpha} ; \ \alpha \in \aaa\setminus \{ \Cl ( \underline{\m} )\})$
and $ \diag ( \Upsilon^{\sharp}_{\alpha} ; \ \alpha \in \aaa\setminus \{ \Cl ( \underline{\m} )\} )$. It implies that
\begin{align*}
\sigma ( P ) \cap \{ \Re z < \varepsilon_{*} h \} &= \{ 0 \} \sqcup \sigma ( \Upsilon )  \\
&=\{ 0 \} \sqcup \bigcup_{\alpha \in \aaa \setminus \{ \Cl ( \underline{\m} ) \}} \sigma(\Upsilon_{\alpha}).
\end{align*}
Applying again Proposition \ref{b17} and Theorem \ref{b10} but  with the blocks $\Upsilon_{\alpha}$ and $\Upsilon^{\sharp}_{\alpha}$, $\alpha \in \aaa \setminus \{ \Cl ( \underline{\m} ) \}$,
we obtain the following  variant of Theorem \ref{b23}, whose formulation is close to that of
 Theorem 5.8 in \cite{Mi19}.

\begin{theorem}\sl \label{b20}
Suppose that the assumptions of Theorem \ref{a15} are satisfied. Assume also that \eqref{a17}, \eqref{h1}, and \eqref{a24} hold true and that $\uuu^{( 0 ) , II} = \emptyset$. Let $\Upsilon^{\sharp}_{\alpha}$ be defined by \eqref{b25}, then
\begin{equation*}
\sigma ( P ) \cap \{ \Re z < \varepsilon_{*} h \} \subset \{ 0 \} \cup \bigcup_{\alpha \in \aaa \setminus \{ \Cl ( \underline{\m} ) \}} \bigcup_{j = 1}^{p} e^{- 2 S_{j} / h} \Big( \sigma \Big( \jjj \circ \rrr_{j } \big( e^{2 S_{j} / h} \Upsilon_{\alpha}^{\sharp} \big) \Big) + D ( 0 , \ooo ( h^{\infty} ) ) \Big) .
\end{equation*}
Moreover, for all $\alpha \in \aaa$, $j = 1 , \ldots , p$, $K >  0$ large enough and $\lambda \in \sigma ( \jjj \circ \rrr_{j} ( e^{2 S_{j} / h} \Upsilon_{\alpha}^{\sharp} ) )$, one has 
\begin{equation*}
n ( P , D_{j , \lambda}^{K} ) = \sum_{\beta \in \aaa} m_{\equiv} \Big( \lambda , \jjj \circ \rrr_{j} \big( e^{2 S_{j} / h} \Upsilon^{\sharp}_{\beta} \big) \Big) ,
\end{equation*}
where $n ( P , D_{j , \lambda}^{K} )$ is the number of eigenvalues of $P$ in $D_{j , \lambda}^{K} = D ( e^{- 2 S_{j} / h} \lambda , e^{- 2 S_{j} / h} h^{K} )$ counted with their multiplicity and with the convention that $m_{\equiv} ( \lambda , M ) = 0$ if $\lambda \notin \sigma_{\equiv} ( M )$.
\end{theorem}

In the previous result, if $e^{- S_{j} / h}$ does not appear in the graded writing of $\Upsilon_{\alpha}^{\sharp}$, the matrix $\jjj \circ \rrr_{j} ( e^{2 S_{j} / h} \Upsilon_{\alpha}^{\sharp} )$ acts on the trivial vector space $\{ 0 \}$ and its spectrum is empty.

In order to emphasize the connection with the formulation of
 Theorem 5.8 in \cite{Mi19}, let us note, with a
 slight abuse of notation, $\rrr^{k} = \rrr_{2} \circ \cdots \circ \rrr_{2}$ ($k$ times),
 with the convention $\rrr^{0} = \Id$.
 It then follows from Appendix \ref{s2}
 that $\rrr^{j-1}=\rrr_{j}$ for every $j\in\{1,\dots,p\}$.
Thus, for every $j\in\{1,\dots,p\}$,  the term
$\jjj \circ \rrr_{j } ( e^{2 S_{j} / h} \Upsilon_{\alpha}^{\sharp} )$
is nothing but $\jjj \circ \rrr^{j-1 } ( e^{2 S_{j} / h} \Upsilon_{\alpha}^{\sharp} )$,
which is the writing appearing  in \cite[Theorem 5.8]{Mi19}.

\begin{remark}\sl
A similar result holds true without the assumption $\uuu^{( 0 ) , II} = \emptyset$. This requires to construct adapted quasimodes 
as in \cite[Section 3]{Mi19} (see for example formula (3.14) there) with cut-off functions as above.
\end{remark}

\appendix
\addcontentsline{toc}{section}{Appendix}

\addtocontents{toc}{\protect\setcounter{tocdepth}{0}}
\section{Spectrum of transport operators}

We begin the appendix with a result used to solve some transport equations.

\begin{lemma}\sl \label{a30}
For $m \in \N^{*}$, let $\ppp_{hom}^{m}$ denote the set of complex polynomials in $d$ variables which are homogeneous of degree $m $. Let $A \in \mmm_{d} ( \C )$ be a matrix and let $\lll_{A} : \ppp_{hom}^{m} \rightarrow \ppp_{hom}^{m}$ be given by
\begin{equation*}
\lll_{A} p = A x \cdot \nabla p .
\end{equation*}
If $\sigma ( A ) \subset\{ \Re z > 0 \}$, then $\sigma ( \lll_{A} ) \subset \{ \Re z > 0 \}$.
\end{lemma}

\begin{proof}
Assume first that $A$ is diagonalizable and denote by
$\lambda_{1} , \ldots , \lambda_{d}$ its eigenvalues.
After a linear change of variable in $\C^{d}$ (leaving $\ppp_{hom}^{m}$ invariant), we can assume that $A$ is diagonal. Then, the monomials of degree $m$ form a basis of eigenvectors of $\lll_{A}$. Moreover, the eigenvalues of $\lll_{A}$ are the numbers
of the form $\vert \gamma \vert_{\lambda} := \sum_{j=1}^{d} \lambda_{j} \gamma_{j}$, where
 $\gamma = ( \gamma_{1} , \dots , \gamma_{d} ) \in \N^{d}$ satisfies  $\sum_{j=1}^{d} \gamma_{j} = m$. By density of the diagonalizable matrices in $\mmm_{d} ( \C )$, the last point holds for any matrix $A \in \mmm_{d} ( \C )$. Since $\Re \lambda_{j} > 0$ for all $j$ implies $\Re \vert \lambda \vert_{\gamma} > 0$ for all $\gamma\in \N^{d}$ such that $\sum_{j=1}^{d} \gamma_{j} > 0$, the lemma follows.
\end{proof}

\section{About the supersymmetric structure} \label{s4}

We now give an example showing that all the operators considered in this paper do not have a nice supersymmetric structure. We send the reader to H{\'e}rau, Hitrik, and Sj{\"o}strand \cite{HeHiSj13} and to the last author \cite{Mi16} for general discussions about supersymmetric structure for differential operators of second order. We say that $P$ as in \eqref{a3} admits a temperate supersymmetric structure if there exist a smooth $d \times d$ matrix $G ( x , h )$ and $M > 0$ such that
\begin{equation} \label{e1}
P = - \sum_{i , j = 1}^{d} ( h \partial_{x_{i}} - \partial_{x_{i}} f ( x ) ) \circ G_{i , j} ( x , h ) \circ ( h \partial_{x_{j}} + \partial_{x_{j}} f ( x ) ) ,
\end{equation}
and $\Vert G ( x , h ) \Vert \lesssim h^{- M}$ locally in $x$. Note that \eqref{e1} implies $P ( e^{- f / h} ) = P^{\dagger} ( e^{ - f / h} ) = 0$ and that the present definition of temperate supersymmetric structure is weaker than that of \cite{Mi16}.

\begin{proposition}\sl \label{e2}
In dimension $d = 2$, there exists an operator $P$ satisfying the assumptions of Theorem \ref{a66} and having no temperate supersymmetric structure.
\end{proposition}

The counterexample constructed in the proof also shows that the determination of the supersymmetric structure (that is the matrix $G$) of an operator having a temperate supersymmetric structure is an instable question. On the contrary, since all the closed forms on $\R^{d}$ are exact, all the operators $P$ satisfying the assumptions of Theorem \ref{a66} have a supersymmetric structure which may not be temperate (see Theorem 1.2 of \cite{HeHiSj13}).

\begin{figure}
\begin{center}
\begin{picture}(0,0)%
\includegraphics{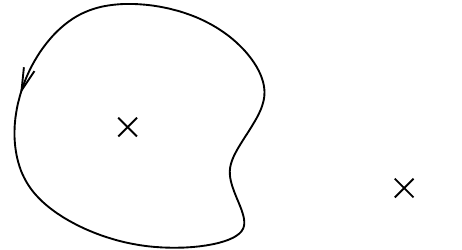}%
\end{picture}%
\setlength{\unitlength}{1184sp}%
\begingroup\makeatletter\ifx\SetFigFont\undefined%
\gdef\SetFigFont#1#2#3#4#5{%
  \reset@font\fontsize{#1}{#2pt}%
  \fontfamily{#3}\fontseries{#4}\fontshape{#5}%
  \selectfont}%
\fi\endgroup%
\begin{picture}(7380,3967)(-2264,-6675)
\put(4201,-6286){\makebox(0,0)[b]{\smash{{\SetFigFont{9}{10.8}{\rmdefault}{\mddefault}{\updefault}$\rho_{2}$}}}}
\put(-224,-5311){\makebox(0,0)[b]{\smash{{\SetFigFont{9}{10.8}{\rmdefault}{\mddefault}{\updefault}$\rho_{1}$}}}}
\put(5101,-3436){\makebox(0,0)[b]{\smash{{\SetFigFont{9}{10.8}{\rmdefault}{\mddefault}{\updefault}$\R^{2}$}}}}
\put(-2249,-3886){\makebox(0,0)[b]{\smash{{\SetFigFont{9}{10.8}{\rmdefault}{\mddefault}{\updefault}$\gamma$}}}}
\end{picture}%
\end{center}
\caption{The geometrical setting in the proof of Proposition \ref{e1}.}
\label{f3}
\end{figure}
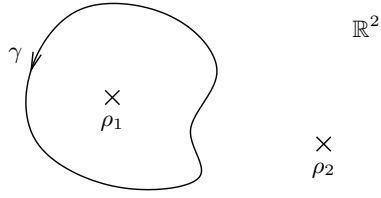

\begin{proof}
First, we consider an operator $P_{0}$ satisfying all the assumptions of Theorem \ref{a66} and having a temperate supersymmetric structure (for instance, the Witten Laplacian described in \eqref{a16}). Of course, $P_{0}$ is of the form \eqref{a3}. We assume in addition that the Morse function $f$ is such that there exist two points $\rho_{1} , \rho_ {2} \in \R^{2}$, a simple smooth loop $\gamma$ around $\rho_{1}$ but not $\rho_{2}$, and $C_{0} > 0$ such that
\begin{equation} \label{e4}
\max f_{\vert \gamma} < C_{0} < \min ( f ( \rho_{1} ) , f ( \rho_{2} ) ) ,
\end{equation}
(see Figure \ref{f3}). Let $\chi \in C^{\infty}_{c} ( \R^{2} ; [ 0 , 1 ] )$ be such that $\chi ( \rho_{1} ) = 1$ and $\supp \nabla \chi$ is sufficiently close to $\gamma$ (so, in particular, $\chi ( \rho_{2} ) = 0$). We define $P = P_{0} + P_{per}$ with the perturbation operator
\begin{equation*}
P_{per} = \sum_{j = 1}^{2} b_{j}^{per} ( x , h ) \circ h \partial_{x_{j}} + h \partial_{x_{j}} \circ b_{j}^{per} ( x , h ) ,
\end{equation*}
and the smooth and compactly supported vector field
\begin{equation} \label{e5}
b^{per} ( x , h ) = e^{2 f / h} e^{- 2 C_{0} / h} d^{*} \chi \qquad \text{where} \qquad d^{*} = \left( \begin{gathered}
\partial_{x_{2}} \\
- \partial_{x_{1}}
\end{gathered} \right) .
\end{equation}
If the support of $\nabla \chi$ is close enough to $\gamma$, the function $b^{per}$ and all its derivatives are exponentially small from \eqref{e4}. In particular, $b^{per}$ satisfies \eqref{a7} with a null classical expansion and $P_{per}$ does not change the principal symbol of $P_{0}$. Moreover, a direct computation gives
\begin{equation} \label{e7}
P_{per} ( e^{- f / h} ) = - ( P_{per} )^{\dagger} ( e^{ - f / h} ) = e^{f / h} h \div \big( b^{per} e^{- 2 f / h } \big) = 0 .
\end{equation}
Thus, $P$ (as $P_{0})$ satisfies the assumptions of Theorem \ref{a66}. It remains to show that $P$ has no temperate supersymmetric structure. Since $P_{0}$ has such a structure, it is equivalent to show that $P_{per}$ has no temperate supersymmetric structure. We prove this point by contradiction and assume that $P_{per}$ can be written as in \eqref{e1} for some polynomially locally bounded matrix $G ( x , h )$. Since $P_{per}= - ( P_{per} )^{\dagger}$, the matrix $G$ must be antisymmetric, say
\begin{equation*}
G ( x , h ) = \left( \begin{array}{cc}
0 & g ( x , h ) \\
- g ( x , h ) & 0
\end{array} \right) ,
\end{equation*}
for some smooth function $g$ with $\vert g ( x , h ) \vert \lesssim h^{- M}$ locally. Expanding \eqref{e1} gives
\begin{align*}
P_{per} &= - ( h \nabla - \nabla f ) \cdot G ( h \nabla + \nabla f )  \\
&= h ( d^{*} g ) \cdot h \nabla - h \nabla \cdot ( G \nabla f ) - ( G \nabla f ) \cdot h \nabla  \\
&= \Big( \frac{h}{2} d^{*} g - G \nabla f \Big) \cdot h \nabla + h \nabla \cdot \Big( \frac{h}{2} d^{*} g - G \nabla f \Big) .
\end{align*}
Then, $G$ satisfies the relation
\begin{equation} \label{e8}
b^{per} = \frac{h}{2} d^{*} g - G \nabla f = \frac{h}{2} e^{2 f / h} d^{*} \big( e^{- 2 f / h} g \big) ,
\end{equation}
which is similar to \cite[(2.4)]{Mi16}. Comparing with \eqref{e5}, this equation is equivalent to
\begin{equation*}
d^{*} \big( e^{- 2 f / h} g \big) = \frac{2}{h} e^{- 2 C_{0} / h} d^{*} \chi .
\end{equation*}
Since $d^{*} \psi = 0$ implies that $\psi$ is constant in dimension $2$, this gives
\begin{equation} \label{e6}
g = \frac{2}{h} e^{2 ( f - C_{0} ) / h} ( \chi + C ( h ) ) ,
\end{equation}
for some constant $C ( h ) \in \R$. Computing $g$ at $x = \rho_{2}$ where $\chi = 0$, \eqref{e4}, \eqref{e6} and $\vert g ( \rho_{2} , h ) \vert \lesssim h^{- M}$ imply that $C ( h )$ must be exponentially small. On the other hand, Computing $g$ at $x = \rho_{1}$ where $\chi = 1$ leads to $g ( \rho_{1} , h ) \geq h^{- 1} e^{2 ( f ( \rho_{1} ) - C_{0} ) / h}$ for $h$ small enough. Then, \eqref{e4} shows that $g ( \rho_{1} , h )$ is exponentially large, in contradiction with $\vert g ( \rho_{1} , h ) \vert \lesssim h^{- M}$. Summing up, $P_{per}$ and then $P$ have no temperate supersymmetric structure and the proposition follows.
\end{proof}

\section{Iteration of the Schur complement} \label{s2}

Let us conclude the appendix with a lemma about the Schur complement. We recall that, for a matrix $M \in \mmm_{d + d^{\prime}} ( \C )$ with $d , d^{\prime} \in \mathbb N^{*}$ written by blocks
\begin{equation*}
M = \left( \begin{array}{cc}
A & B \\
C & D
\end{array} \right) \qquad \text{with} \qquad
A \in \mmm_{d} ( \C ) \text{ invertible},
\end{equation*}
the Schur complement of the block $D\in \mmm_{d^{\prime}} ( \C )$ of $M$ is the matrix defined by
\begin{equation*}
\rrr ( M ) = D - C A^{- 1} B \in \mmm_{d^{\prime}} ( \C ).
\end{equation*}
Moreover, by the Schur complement method, $M$ is invertible if and only if $\rrr ( M )$ is invertible.

\begin{lemma}\sl \label{g21}
For $d_{1} , d_{2} , d_{3} \in \mathbb N^{*}$ and matrices $M \in \mmm_{d_{1} + d_{2} + d_{3}} ( \C )$ and $M^{\prime} \in \mmm_{d_{2} + d_{3}} ( \C )$ written by blocks
\begin{equation*}
M = \begin{pmatrix} A & B & C \\
D & E & F \\
G & H & I \end{pmatrix} \qquad \text{and} \qquad
M^{\prime} = \begin{pmatrix} A^{\prime} & B^{\prime} \\
C^{\prime} & D^{\prime} \end{pmatrix} ,
\end{equation*}
we denote respectively, when they make sense, by $\rrr_{1} ( M )$, $\rrr_{1,2} ( M )$ and $\rrr_{2} ( M^{\prime} )$ the Schur complements of the blocks $\left( \begin{smallmatrix} E & F \\ H & I \end{smallmatrix} \right) \in \mmm_{d_{2} + d_{3}} ( \C )$ of $M$, $I \in \mmm_{d_{3}} ( \C )$ of $M$ and $D^{\prime} \in \mmm_{d_{3}} ( \C )$ of $M^{\prime}$.

Then, when $M$ has the previous form with $A$ and 
$\left( \begin{smallmatrix}
A & B \\
D & E
\end{smallmatrix} \right)$
invertible, the Schur complements $\rrr_{1,2} ( M )$, $\rrr_{1} ( M )$ and $\rrr_{2} ( \rrr_{1} ( M ) )$ make sense and satisfy
\begin{equation*}
\rrr_{2} ( \rrr_{1} ( M ) ) = \rrr_{1,2} ( M ) .
\end{equation*}
\end{lemma}

\begin{proof}
First, since $A$ and 
$\left( \begin{smallmatrix} A & B \\
D & E \end{smallmatrix} \right)$
are invertible, the respective Schur complements $\rrr_{1} ( M )$ and $\rrr_{1,2} ( M )$ of the blocks 
$\left( \begin{smallmatrix} E & F \\
H & I \end{smallmatrix} \right)$
and $I$ of $M$ make sense. Moreover, by the Schur complement method, the invertibility of $A$ and $\left( \begin{smallmatrix} A & B \\ D & E \end{smallmatrix} \right)$ imply that the Schur complement $E - D A^{- 1} B$ of the block $E$ of
$\left( \begin{smallmatrix} A & B \\
D & E \end{smallmatrix} \right)$
is invertible, and a straightforward computation shows that
\begin{equation*}
\begin{pmatrix} A & B \\
D & E \end{pmatrix}^{-1} = \begin{pmatrix} A^{-1} + A^{-1} B ( E - D A^{-1} B )^{- 1} D A^{- 1} & - A^{-1} B ( E - D A^{- 1} B )^{- 1} \\
- ( E - D A^{- 1} B )^{- 1} D A^{- 1} & ( E - D A^{- 1} B )^{- 1} \end{pmatrix} .
\end{equation*}
It follows that
\begin{align}
\rrr_{1 , 2} ( M ) &= I - \begin{pmatrix} G & H \end{pmatrix} \begin{pmatrix} A & B \\
D & E \end{pmatrix}^{- 1} \begin{pmatrix} C \\ F \end{pmatrix}   \nonumber \\
&= I - G A^{- 1} C + ( H - G A^{- 1} B ) ( E - D A^{- 1} B )^{- 1} ( D A^{- 1} C - F) . \label{g28}
\end{align}

Besides, it holds
\begin{equation*}
\rrr_{1} ( M ) = \begin{pmatrix} E & F \\
H & I \end{pmatrix} - 
\begin{pmatrix} D \\ G \end{pmatrix} A^{- 1} \begin{pmatrix} B & C \end{pmatrix}
= \begin{pmatrix} E - D A^{- 1} B & F - D A^{- 1} C \\
H - G A^{- 1} B & I - G A^{- 1} C \end{pmatrix} .
\end{equation*}
Since $E - D A^{- 1} B$ is invertible, the Schur complement $\rrr_{2} ( \rrr_{1} ( M ) )$ of the block $I - G A^{- 1} C$ of $\rrr_{1} ( M )$ makes thus sense and satisfies
\begin{equation} \label{g29}
\rrr_{2} ( \rrr_{1} ( M ) ) = I - G A^{- 1} C - ( H - G A^{- 1} B ) ( E - D A^{- 1} B )^{- 1} ( F - D A^{- 1} C ) .
\end{equation}
The statement of Lemma \ref{g21} then follows from \eqref{g28} and \eqref{g29}.
\end{proof}

\addtocontents{toc}{\protect\setcounter{tocdepth}{1}}

\bibliographystyle{amsplain}
\providecommand{\href}[2]{#2}


\end{document}